\documentclass[11pt,regno]{amsart}
\usepackage[normalem]{ulem}
\usepackage{euscript,epstopdf,amscd,amsgen,times}
\usepackage{amssymb,amsmath,amsthm,amssymb,latexsym,mathtools}
\usepackage{color}
\usepackage[dvipsnames]{xcolor}
\usepackage{graphicx}
\usepackage{mathrsfs}
\usepackage[abs]{overpic}

\makeatletter
\newcommand*\bigcdot{\mathpalette\bigcdot@{.5}}
\newcommand*\bigcdot@[2]{\mathbin{\vcenter{\hbox{\scalebox{#2}{$\m@th#1\bullet$}}}}}
\makeatother

\newtheorem*{notation*}{Notation}

\newtheorem*{proposition*}{Proposition}
\newtheorem*{theorem*}{Theorem}

\newtheorem*{Stheorem}{Stabilisation of cycles}
\newtheorem*{Rtheorem}{Robust tangencies}

\newtheorem*{Cav}{Caveat}

\newtheorem{defi}{Definition}[section]
\newtheorem{teo}[defi]{Theorem}

\newtheorem{prop}[defi]{Proposition}

\newtheorem{coro}{Corollary}[section]
\newtheorem{teor}{Theorem}[section]

\newtheorem{cl}[defi]{Claim}

\newtheorem{remark}[defi]{Remark}
\newtheorem{notation}[defi]{Notation}
\newtheorem{lemma}{Lemma}[section]

\usepackage{amsfonts}
\usepackage{amssymb}
\usepackage{color}

\numberwithin{equation}{section}
\DeclareMathOperator{\diff}{Diff}

\DeclareMathOperator{\loc}{loc}
\DeclareMathOperator{\V}{\mathcal{P}^{\mathrm s}_{Y}}

\def\ru{{\mathrm{u}}}
\def\ruu{{\mathrm{uu}}}

\def\rs{{\mathrm{s}}}
\def\rss{{\mathrm{ss}}}

\def\loc{{\mathrm{loc}}}

\def\bv{\mathbf{v}}

\def\ve{\varepsilon}
\def\puf{\rho}

\def\vpp{\mathrm{BH}}

\def\bfs{{\mathbf{s}}}

\def\BH{{\mathrm{BH}}}

\def\gesp{{\geqslant}}
\def\lesp{{\leqslant}}

\newcommand{\eqdef}{\stackrel{\scriptscriptstyle\rm def}{=}}

\makeatletter
\newcases{mycases}{\quad}{%
  \hfil$\m@th\displaystyle{##}$}{$\m@th\displaystyle{##}$\hfil}{\lbrace}{.}
\makeatother

\title[Stabilisation and robust tangencies of cycles]{
Nontransverse heterodimensional cycles:
stabilisation and robust tangencies}
\author[L. J. D\'iaz and S. A. P\'erez ]{Lorenzo J. D\'iaz and Sebasti\'an A. P\'erez}
\address{Departamento de Matem\'atica PUC-Rio, Marqu\^es de S\~ao Vicente 225, G\'avea, Rio de Janeiro 225453-900, Brazil}
\email{lodiaz@mat.puc-rio.br}

\address{Instituto de Matem\'aticas, Pontificia Universidad Cat\'olica de Valpara\'iso, Blanco Viel 596, Cerro Bar\'on, Valpara\'iso, Chile.}
\email{sebastian.perez.o@pucv.cl}

\begin{document}

\begin{abstract} 
We consider three-dimensional diffeomorphisms having simultaneously heterodimensional cycles 
and heterodimensional tangencies associated to saddle-foci. These cycles lead to a completely nondominated bifurcation setting.
For every $r\gesp 2$, we exhibit a class
of such diffeomorphisms whose heterodimensional cycles can be $C^r$ stabilised 
and  (simultaneously)  approximated by diffeomorphisms with $C^r$ robust homoclinic tangencies. 
The complexity of our nondominated setting with plenty of homoclinic and heteroclinic intersections
is used to overcome the difficulty of performing $C^r$ perturbations, $r\geqslant 2$, which are remarkably   more difficult than $C^1$ ones. 
Our proof is reminiscent of the Palis-Takens' approach to get
surface diffeomorphisms with infinitely many sinks 
(Newhouse phenomenon)
in the unfolding of homoclinic tangencies of surface diffeomorphisms. This proof
involves a scheme of renormalisation along nontransverse heteroclinic orbits converging to a center-unstable H\'enon-like family displaying blender-horseshoes. A crucial step is the analysis of the
embeddings of these blender-horseshoes in a nondominated context.

\end{abstract}

\thanks{This paper is part of the PhD thesis of SP (PUC-Rio) supported by CNPq and CAPES - Finance Code 001 (Brazil).
LJD was partially supported by INCTMat-Faperj and CNPq (Brazil). SP was supported by Programa Postdoctorado FONDECYT 3190174 (Chile) and partially 
supported by CMUP grant PTDC/CTM/BIO-4043-2014 -- project UID/MAT/00144/2019 (Portugal).
The authors  thank the hospitality of 
CMUP,
PUC (Chile), and USACH (Chile).}


\maketitle

\begin{flushright}
To Jacob Palis, in the occasion of his 80th birthday
\end{flushright}

\section{Introduction}
\label{s.intro1}
Palis' density  conjecture  \cite{Pal:00} claims that bifurcations through \emph{cycles} (either homoclinic tangencies or
heterodimensional cycles) associated to 
\emph{saddles}
(hyperbolic periodic points)
are the main mechanisms for destroying hyperbolic dynamics:
any
nonhyperbolic
system can be approximated by diffeomorphisms displaying one of
those bifurcations.
A homoclinic tangency associated to a saddle occurs when the invariant 
(stable and unstable)
sets of a saddle have a nontransverse intersection. 
A heterodimensional cycle associated with a pair of saddles of different 
 \emph{indices} (dimension of the unstable bundle) 
 occurs when the invariant sets of these saddles intersects cyclically.  
 Note that heterodimensional cycles can only occur in dimension at least three and
that there are  settings (as the one in this paper) where both types of bifurcations
 occur simultaneously with overlapping effects.

As a consequence of the Kupka-Smale genericity theorem\footnote{Periodic points of generic diffeomorphisms are hyperbolic and their invariant manifolds
are in general position (i.e., either they  intersect transversely or they are disjoint).},
a cycle associated to saddles is a fragile configuration that can be destroyed by small perturbations.
However, these configurations can become  robust (indestructible by small perturbations) when these saddles are embedded  in some special type of horseshoes. Hence, it is natural to consider also heterodimensional cycles and tangencies associated to (basic) hyperbolic sets (for the precise definition see Section~\ref{ss.cycles}).
One aims to understand when a bifurcation through a fragile cycle associated to saddles can lead to such robust cycles.

Bonatti\footnote{Formulated in Bonatti's  talk 
{\emph{The global dynamics of $C^1$ generic diffeomorphisms or flows},} in the
Second Latin American Congress of Mathematicians, Canc\'un, M\'exico (2004). See also \cite{Bon:11}.}
stated a stronger version of Palis' conjecture using robust cycles:
the union of the $C^r$ open sets of hyperbolic diffeomorphisms (satisfying the Axiom A and the no-cycles properties) 
and of diffeomorphisms with  $C^r$ robust cycles is dense in the space of $C^r$ diffeomorphisms,
see  \cite[Conjecture 1.10]{BonDia:12}. 
For results and recent progress in the previous conjectures, see ~\cite{PujSam:00, CroPuj:15,CroSamYan:15} for the Palis' one and
~\cite{New:70, Mor:11,Asa:08,BonDia:12} for Bonatti's one. Some of these results will be discussed below.

The latter conjecture has several motivations, one of them comes from the study of 
global dynamics of diffeomorphisms when considering
 the decomposition
of the chain recurrence set into its chain of recurrence classes. 
Note first that two saddles involved in a cycle are always in the same class of recurrence.
One aims to put these saddles robustly into the same class. If such saddles are contained in a pair of transitive hyperbolic sets involved in a robust cycle then the continuations
of the hyperbolic sets (and hence the 
ones of the initial saddles) are also in the same class of recurrence. This gives a way to put saddles with different indices  into prescribed recurrence classes. 
This process is known as {\emph{stabilisation of a cycle}}.
More precisely, a heterodimensional cycle of a $C^r$ diffeomorphism  $f$
associated to saddles $P$ and $Q$ can be {\emph{$C^r$ stabilised}} if there are
diffeomorphisms arbitrarily
$C^r$ close to $f$ with a $C^r$ robust cycle associated to transitive hyperbolic sets 
containing the continuations of $P$ and $Q$.
The stabilisation of a homoclinic tangency associated to a saddle is defined analogously.

The stabilisation of cycles depends on the type of cycle, differentiability, and dimension. To avoid  technicalities,  we will restrict our discussion to
dimensions two and three\footnote{This allows us to skip the technical discussion 
of the so-called {\emph{coindex}} of a heterodimensional cycle, since in dimension three the coindex is always one. For phenomena that may occur  in higher dimensions, as
for instance robust tangencies of large codimension, we refer to \cite{BarKiRai:14, BarRai:17}.}.
We first consider homoclinic tangencies. For surface diffeomorphisms this question is completely solved:
  there are no $C^1$ robust tangencies and hence
no homoclinic tangency can be
$C^1$ stabilised, \cite{Mor:11}. On the other hand,
if $r\geqslant 2$ then 
every  such  a tangency can be 
$C^r$ stabilised, \cite{New:}. 
In dimension three, a combination of \cite{New:,Rom:95,PalVia:94} and the theory of normal 
hyperbolicity
implies that,
every $C^r$ homoclinic tangency can be 
$C^r$ stabilised for $r\geqslant 2$. In the $C^1$ case, the stabilisation of homoclinic tangencies
involves geometrical constraints
 and, in general, it is not known which tangencies 
can be stabilised (see also 
\cite{BonCroDiaGou:13}).
For instance, combining normally hyperbolic surfaces and \cite{Mor:11}, one can get
 homoclinic
tangencies that cannot be  $C^1$ stabilised, see also \cite[Sections 4.3--6]{Bon:11}.

Consider now heterodimensional cycles.
First, every three-dimensional heterodimensional cycle leads to $C^1$ robust cycles \cite{BonDia:08}, although these 
cycles may be not related to the saddles in the initial cycle. In \cite{BonDia:12} there is given a class of heterodimensional cycles that cannot be $C^1$
stabilised ({\emph{twisted} cycles).  Finally, in \cite{BonDiaKir:12} it is proved that every nontwisted  cycle
can be $C^1$ stabilised. 
The techniques used in these works are genuinely $C^1$. 
Due to the absence of suitable tools, the stabilisation problem  in higher differentiability is widely open.

To explain our results, we recall that, in dimension three, two saddles with different indices have a {\emph{heterodimensional tangency}} if their two dimensional invariant manifolds have some nontransverse intersection.
These tangencies were introduced in \cite{DiaNogPuj:06} as a source of robustly nondominated/wild dynamics, see also \cite{KirSom:12,BarPer:19}.
In this paper, we consider a class of three-dimensional $C^r$ diffeomorphisms 
whose heterodimensional cycles involve heterodimensional  tangencies (see Figure~\ref{fig:S}).
For every $r\geqslant 2$,
 we state the $C^r$ stabilisation of such cycles and  show that they also provide $C^r$ robust homoclinic tangencies,
 see the Stabilisation  and Robust tangencies theorems below. Let us now provide further details of our statements.

\begin{figure}[h]
\begin{overpic}[scale=0.05]{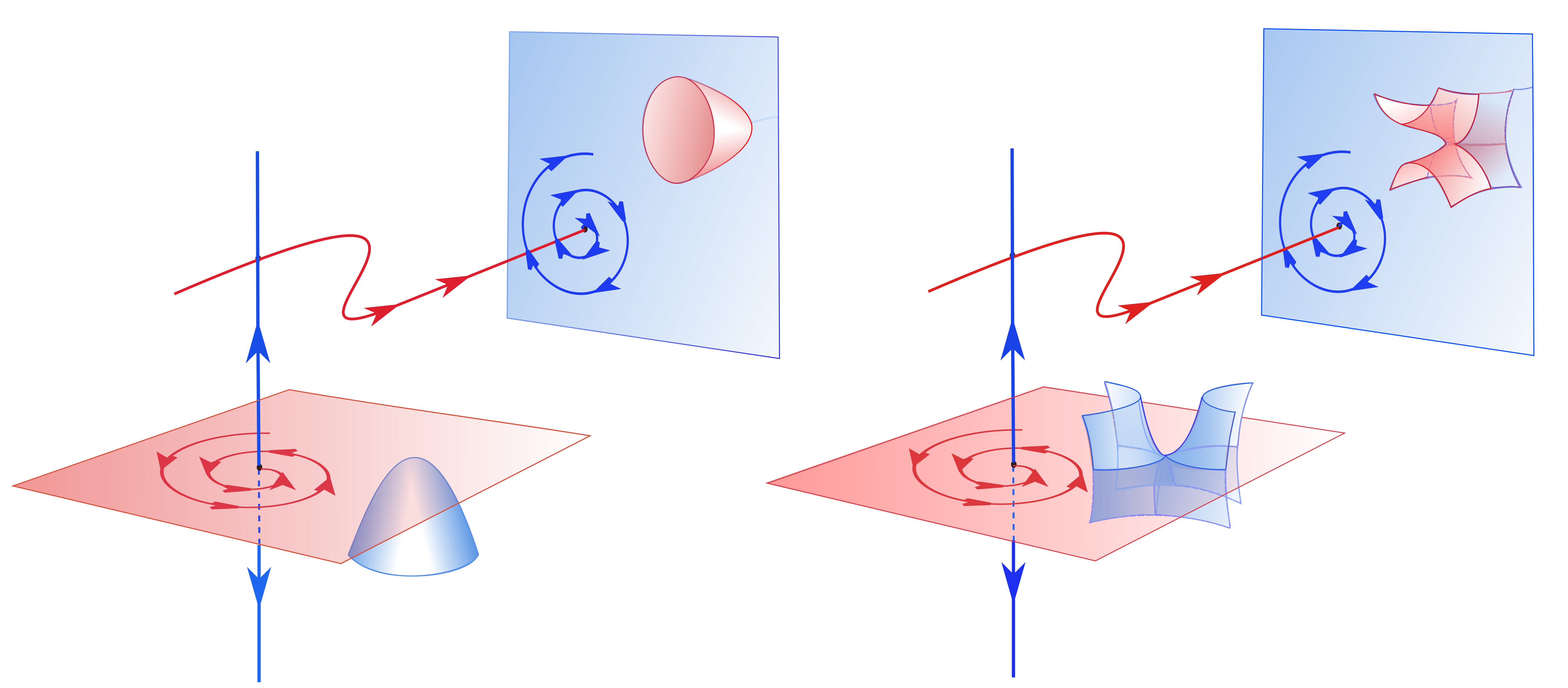} 
 	\put(30,50){{\large$Q$}}
  	\put(112,70){{\large$P$}}
 	\put(160,50){{\large$Q$}}
 	\put(244,70){{\large$P$}}
 \end{overpic}
\caption{Heterodimensional cycles with heterodimensional tangencies.}
\label{fig:S}
\end{figure}
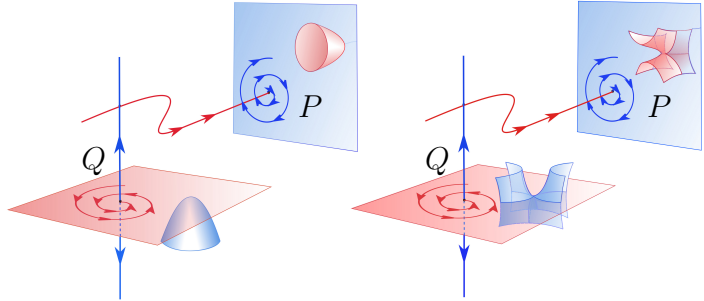
Let $M$ be a three-dimensional compact manifold.
We consider a set 
$\mathcal{H}^r_{{\BH}}(M)$ of $C^r$ diffeomorphisms of $M$ having a
heterodimensional cycle with a heterodimensional tangency associated to saddle-foci
$P$ and $Q$ of indices two and one satisfying the following conditions: 
\begin{itemize}
\item
Linearising assumptions at $P$ and $Q$ and spectral conditions 
implying some sort of locally dissipative behaviour (Section~\ref{ss.linear}).
\medskip
\item
The one-dimensional invariant manifolds 
$W^{\mathrm s}(P,f)$ and $W^{\mathrm u}(Q,f)$ have a quasi-transverse intersection
along the orbit of some point $X$ and the 
two dimensional invariant manifolds 
$W^{\mathrm u}(P,f)$ and $W^{\mathrm s}(Q,f)$ 
have a heterodimensional tangency 
along the orbit of some point $Y$. This tangency may be of hyperbolic or elliptic type (Section~\ref{ss.commentsonthegeometry}).
The type of tangency plays an important role in the resulting dynamics.
\medskip
\item
Conditions on the ``transitions''  from $P$ to $Q$
and from $Q$ to $P$ 
along
the orbits of the heteroclinic points $X$ and $Y$ (Section~\ref{ss.semilocal}).
\end{itemize}
\medskip
The precise description of the set $\mathcal{H}^r_{\BH}(M)$ is given in
Section~\ref{s:bif}.  Our main results are the following, see Theorem~\ref{t:1} for further details.

\begin{Stheorem}
Let  $r\geqslant 2$. Any cycle in 
$\mathcal{H}^r_{\BH}(M)$
can be 
$C^r$ stabilised.
\end{Stheorem}


The next result deals with diffeomorphisms in $\mathcal{H}^r_{\BH}(M)$ whose heterodimensional tangency
is of elliptic type  (see the lefthand side of Figure~\ref{fig:S}).
This leads to the definition of the subset
$\mathcal{H}^r_{\BH, \mathrm{e}^+}(M)$ of $\mathcal{H}^r_{\BH}(M)$, see Section~\ref{ss.commentsonthegeometry}  for the precise definition and a discussion.

\begin{Rtheorem} Let $r\geqslant 2$. Every diffeomorphism in
$\mathcal{H}^r_{\BH, \mathrm{e}^+}(M)$
can be 
$C^r$ approximated by diffeomorphisms with a $C^r$ robust homoclinic tangency associated to a basic set containing the continuation of the saddle-focus of index two.
\end{Rtheorem}

\subsection{Our approach: a renormalisation scheme leading to blender-horseshoes}
To explain the strategy of the proof of our results let us first recall the approach  in {\cite[Chapter~6]{PalTak:93}}%
\footnote{In \cite{PalTak:93} it is proved the generic coexistence of infinitely many sinks,
in this proof the occurrence of robust tangencies is a key step. In this homoclinic case, these robust tangencies imply the stabilisation of the tangency, defined similarly as in the case of a cycle.}
to stabilise homoclinic tangencies of $C^2$ diffeomorphisms.
The construction in \cite{PalTak:93} has  the following main
  ingredients: {\bf{(a)}}  a renormalisation scheme at a homoclinic tangency,
{\bf{(b)}} convergence of the scheme to a quadratic one-parameter family, 
{\bf{(c)}} existence of parameters of the family corresponding to 
{\emph{thick horseshoes}} (horseshoes with 
large ``fractal-like dimension"), and {\bf{(d)}}
control of the localisation of the thick horseshoe guaranteeing  
  that
it is   {\em{homoclinically related}}\footnote{Two hyperbolic sets with the same index are \textit{homoclinically related} if their invariant manifolds intersect   cyclically and transversally.} to the continuation of the initial saddle.
A key property in this approach is that thick  horseshoes
 are $C^2$ robust, thus their existence for the limit map extends to nearby systems. 
 
Our strategy to get the $C^r$ stabilisation 
of cycles in $\mathcal{H}^r_{\BH}(M)$ translates the ideas of \cite{PalTak:93} to a heterodimensional setting following the approach started in  \cite{DiaKirShi:14}. In our construction, the
ingredients 
{\bf{(a)}}--{\bf{(d)}} above
are replaced by:
{\bf{(a')}}  a renormalisation scheme at a heterodimensional tangency,
{\bf{(b')}} convergence of the scheme to a center-unstable H\'enon-like family, 
{\bf{(c')}} existence of parameters corresponding to 
{\emph{blender-horseshoes,}}
{\bf{(d')}} prove that the blender-horseshoes are homoclinically related to the initial saddle of index two
and have a robust cycle with the initial saddle of index one.
Let us observe that, in very rough terms,
blender-horseshoes are local hyperbolic plugs used to get robust heterodimensional cycles,
where they play a role similar to the one of the thick horseshoes for homoclinic tangencies, 
see Section~\ref{s.BH} for details.
As above, a key step  is to
analise how the blender-horseshoes are embedded in 
the global dynamics.   

Here we rely on preliminary results in  \cite{DiaPer:19,DiaPer:19b} towards the development of the strategy {\bf{(a')}}--{\bf{(d')}}. 
In our context, the ``limit'' family is the center-unstable H\'enon-like family given by
\[
G_{\varpi}(x,y,z) \eqdef (y,\mu+y^2+\eta_1\,y\,z+\eta_2\,z^2,\xi\,z+y), \quad \varpi= (\xi, \mu,  \eta_1, \eta_2).
\]
For diffeomorphisms in $\mathcal{H}^r_{\BH}(M)$, 
the renormalisation scheme  and its convergence to the family $G_{\varpi}$ (steps   {\bf{(a')}}  and {\bf{(b')}}).
were obtained in \cite{DiaPer:19}. 
A crucial property  (step  {\bf{(c')}})
is that there is an open set $\mathcal{O}_{\mathrm{BH}}$ of parameters $\varpi$ for which the family $G_\varpi$ exhibits blender-horseshoes, see
\cite{DiaPer:19b}.
In the final step {\bf{(d')}},
we
analyse how these blender-horseshoes are embedded in the global dynamics
(the blender-horseshoe is homoclinically related
to $P$ and has a robust cycle with $Q$). 
This is a major difficulty in our nondominated setting. 
It turns out that  the lack of domination is simultaneously a difficulty and, in some sense, an advantage.
First, the existence of nonreal 
multipliers   makes the renormalisation scheme and the ``existence and localisation" of blenders 
a difficult task.
On the other hand, the dynamics at the bifurcation is very rich and, in particular, enables us to find new
 homoclinic and heteroclinic orbits close to the initial cycle. As a heuristic principle,  
 this richness allows us to overcome the difficulty of performing $C^r$ perturbations, $r\geqslant 2$, which are notably  more problematic than $C^1$ ones. 

The lack of domination also means that there are plenty of possibilities for unfolding  the cycles involving many parameters. 
For instance, comparing with the setting of homoclinic tangencies where any transverse direction of unfolding behaves in the same way,
the lack of domination implies that any direction of unfolding is different. Thus we have eight natural parameters: 
six parameters corresponding to the unfolding
of the nontransverse intersections (three for the heterodimensional tangency and three for the quasi-transverse heteroclinic intersection),  and  two parameters associated to the arguments of the saddle-foci, see Section~\ref{ss.embeddingfamily}. We  see that  ``unfoldings  following appropriate directions" 
lead to robust cycles. However, the complexity of these cycles is huge  and a  complete 
description of the bifurcations is beyond reach. 
 
We now recall some definitions
and state  precisely our results.

\subsection{Stabilisation of cycles and robust tangencies:
precise statements} \label{ss.cycles}  
Let $M$ be a compact boundaryless manifold. 
Let $\mathrm{Diff}^r(M)$ be the space
of $C^r$ diffeomorphisms of $M$ endowed with the $C^r$ uniform topology. Consider $f\in \mathrm{Diff}^r(M)$ and
$\Lambda_f$ a hyperbolic {\emph{transitive set}} (i.e. with a dense orbit) 
 of $f$. Recall that there is a $C^r$ neighbourhood
$\mathcal{U}_f$ of $f$ such that every $g\in \mathcal{U}_f$  has a hyperbolic set $\Lambda_g$ that is 
topologically conjugate to $\Lambda_f$  
called the 
{\emph{continuation}} of $\Lambda_f$. The \emph{index} 
of $\Lambda_f$ is the dimension of its unstable bundle (by transitivity, this number is well defined).

Consider  $f\in \mathrm{Diff}^r(M)$ having a pair of  transitive  hyperbolic sets
$\Lambda_f$ and $\Upsilon_f$
  with different indices. These sets form 
 a {\emph{heterodimensional cycle}} if their invariant stable and unstable
 sets intersect cyclically, i.e.,
 $W^\rs( \Lambda_f)\cap W^\ru( \Upsilon_f)\ne\emptyset$
 and  $W^\ru( \Lambda_f)\cap W^\rs( \Upsilon_f)\ne\emptyset$.
 This cycle is {\emph{$C^r$ robust}} if there is a  $C^r$ neighbourhood
$\mathcal{U}_f$ of $f$ consisting of diffeomorphisms $g$
such that
the sets
$\Lambda_g$ and $\Upsilon_g$ have a heterodimensional cycle.
The notion of a {\emph{$C^r$ robust homoclinic tangency}}
associated to $\Lambda_f$ is stated similarly:
there is a  $C^r$ neighbourhood
$\mathcal{U}_f$ of $f$ such that for every $g\in \mathcal{U}_f$ 
the  invariant stable and unstable sets of $\Lambda_g$ have some nontransverse
intersection. 
Recall that robust cycles cannot be
associated to  {\emph{trivial}} hyperbolic sets (i.e., periodic orbits).

A heterodimensional cycle of a $C^r$ diffeomorphism  $f$
associated to saddles $P_f$ and $Q_f$ can be {\emph{$C^r$ stabilised}} if there are
diffeomorphisms $g \in \mathrm{Diff}^r(M)$ arbitrarily
$C^r$ close $f$ with a $C^r$ robust cycle associated to transitive hyperbolic sets 
$\Lambda_g$ and $\Upsilon_g$ containing the continuations $P_g$ and $Q_g$, respectively.

Our main result is the following theorem.

\begin{teor}\label{t:1} Let $r\geqslant 2$ and $M$  be a compact boundaryless three-dimensional manifold. Given
$f\in \mathcal{H}^r_{\BH}(M)$, with a cycle associated to saddle-foci $P_f$ and $Q_f$
of indices two and one, there are diffeomorphisms $g$ arbitrarily $C^r$ close to 
$f$ 
 with a blender-horseshoe $\Lambda_{g}$ 
 of  index two  such that:
 \begin{itemize}
 \item [(i)] 
 $\Lambda_{g}$ and $Q_{g}$ has a  $C^r$ robust heterodimensional cycle
 and
\item [(ii)] 
$\Lambda_{g}$ and $P_{g}$ are homoclinically related. 
\end{itemize}
Moreover, if $f\in \mathcal{H}^r_{\BH,\mathrm{e}^+}(M)$ then the blender-horseshoe $\Lambda_g$ has a $C^r$ robust homoclinic tangency.
\end{teor}

\subsection{Steps of the proofs}\label{ss.MR}
We now explain the steps of the proof of Theorem~\ref{t:1}.
Consider $f \in \mathcal{H}^r_{\BH}(M)$ with a cycle 
associated to saddle-foci $P_f$ and $Q_f$  as in the theorem.
A preliminary step is to perturb the original cycle to obtain a new diffeomorphism in $\mathcal{H}^r_{\BH}(M)$
(that continue denoting by $f$)
having transverse homoclinic points and  new additional quasi-transverse heteroclinic points associated to $P_f$  and $Q_f$ 
 (see Proposition \ref{p.Lnbis}).
We can now apply the renormalisation scheme to this new cycle,
getting diffeomorphisms $g$ arbitrarily $C^r$ close to $f$ 
whose dynamics in a neighbourhood of the cycle is close to a  H\'enon-like map 
$G_{\varpi}$ with $\varpi\in \mathcal{O}_{\mathrm{BH}}$. By Proposition~\ref{p.blendersparatodos},
 each diffeomorphism $g$ has a
blender-horseshoe $\Lambda_g$ of index two. 
We will  see that the following  holds:

\smallskip

\noindent (ia)  {\em{The two-dimensional manifolds 
$W^\ru(\Lambda_g,g)$ and $W^{\rs}(Q_g,g)$ intersect transversely,}} see Proposition~\ref{p.why}. 
The difficulty of this step is to control the size of the unstable manifold of $\Lambda_g$,
assuring that  it is
sufficiently ``large" so that it is connected to  the stable manifold of $Q_g$. 
We overcome this difficulty with an analysis 
motivated by the constructions in \cite[Section 6.4]{PalTak:93} for homoclinic tangencies of
surface diffeomorphisms.
 
 \smallskip

\noindent (ib)  {\em{The one-dimensional manifolds 
  $W^\rs(\Lambda_g,g)$ and $W^{\ru}(Q_g,g)$ have  nonempty intersection,}}
  see Proposition~\ref{p.1dim}. This step is inspired by
  \cite[Theorem 1.4]{DiaKirShi:14} (see Remark~\ref{r.kirikishinohara} for a discussion) 
  and involves quantitative aspects of the renormalisation scheme in \cite{DiaPer:19}.
  In this step the new quasi-transverse heteroclinic points above play an important role.
  
  \smallskip

\noindent (ii)  {\em{The saddle $P_g$ and the blender-horseshoe $\Lambda_g$ are homoclinically related,}}
see Proposition~\ref{p.homoclinicallyrelated}. This step is a relatively simple consequence of (ia) and (ib)
where the existence of transverse homoclinic points of $Q_f$ is used.

\medskip

Conditions (ia) and  (ii) are $C^r$ open, $r\geqslant 1$, while (ib) is not (due to deficiency of the sum of the dimensions). The blender-horseshoe allows us to make 
this nontransverse intersection $C^r$ robust, $r\geqslant 1$. Thus conditions (ia) and (ib) give a $C^r$ robust cycle between $Q_g$ and $\Lambda_g$. As $\Lambda_g$ and $P_g$ are homoclinically related, they are contained in a larger hyperbolic set,  implying the stabilisation of the initial cycle. 

In the second part of the theorem, about robust tangencies, we
consider diffeomorphisms with elliptic tangencies in $\mathcal{H}^r_{\BH, \mathrm{e}^+}(M)$ (lefthand side of Figure~\ref{fig:S})
and
 study the intersections between the two-dimensional manifolds  of the saddle-foci in the cycle. We see that these intersections generate
``tubes crossing the reference domain of the blender-horseshoe", see Section~\ref{ss.blendertubes}. These tubes will provide  robust tangencies. This step involves the constructions in \cite{BonDia:12} using {\em{folding manifolds}}. 

\begin{remark}
\label{r.kirikishinohara}
{\em{In \cite{DiaKirShi:14}  it is obtained  a renormalisation scheme 
for $C^r$ diffeomorphisms $f$, $r\geqslant 2$, with a configuration
somewhat similar to the one here, where the saddle-foci are replaced by
a pair of saddles with real multipliers.  In \cite{DiaKirShi:14}
the intersection between the one-dimensional manifolds in (ib) is obtained for $C^{1+\alpha}$ perturbations of $f$.
Let us observe a perhaps counterintuitive fact: the intersections between the ``big'' two dimensional manifolds in (ia) are more difficult to obtain than the   intersections between the ``small'' one dimensional manifolds in  (ib). Indeed, in \cite{DiaKirShi:14} the intersections (ia) and (ii)  were not achieved.}}
\end{remark}

\subsection*{Organisation of  the paper}  
The bifurcation setting is described in Section~\ref{s:bif}. 
In Section~\ref{s.translationandrotation}, we introduce the  perturbations used in our constructions.
In Section~\ref{s.newquasi}, we prove that the set of diffeomorphisms having additional ``special" homoclinic  and quasi-transverse heteroclinic intersections is dense in
$\mathcal{H}^r_{\BH }(M)$. These ``special"  homoclinic and heteroclinic points  will play an important role in our proof. 
Blender-horseshoes and their occurrence in center-unstable H\'enon-like families are discussed
in Section~\ref{s.BH}. 
In Section~\ref{s.ren}, we review some relevant ingredients renormalisation scheme in \cite{DiaPer:19} used in our constructions.
In
Section~\ref{s.interplay}, we study 
the interplay between the blender-horseshoes given by the renormalisation scheme and the additional heteroclinic points.  Section~\ref{s.orbitsanditineraries} deals with
orbits and itineraries  associated to the renormalisation scheme.
The proof of Theorem~\ref{t:1} is completed in 
Sections \ref{s.proofof2}--\ref{s.homocliinictangencies}.
Section~\ref{s.proofof2} deals with the intersections between the two-dimensional invariant manifolds
of $Q$ and of the blender-horseshoe. In Section~\ref{s.1-conecc}, we state the occurrence 
 of robust  intersections between the one-dimensional invariant manifolds of $Q$ and  of the blender-horseshoe. In
  Section~\ref{s.homocliinic relation},  we see that the saddle $P$ and the blender-horseshoe are homoclinically related.
  Finally, in Section~\ref{s.homocliinictangencies} we prove the part of the theorem corresponding to robust tangencies.
Section~\ref{s.explicit1} is an appendix  collecting some explicit calculations of the 
renormalisation scheme borrowed from \cite{DiaPer:19}.

\section{The bifurcation setting}
\label{s:bif}
In this section, we describe precisely  the 
set $\mathcal{H}^r_{\BH}(M)$, see Definition~\ref{d.HB}.
We close this section with some comments on the geometry of the cycle. 
Throughout this section we consider diffeomorphisms $f$
having a pair of saddle-foci of different indices 
$P=P_f$ and $Q=Q_f$.

\subsection{The set $\mathcal{H}^r_{\BH}(M)$} \label{ss.Hrhet}
We now explain the conditions in the definition of $\mathcal{H}^r_{\BH}(M)$: linearising dynamics and 
nontransverse intersections and transition maps.

\subsubsection{Linearisable local dynamics} 
\label{ss.linear}
$\,$

\medskip

\noindent
\textbf{(A)} {\emph{Saddle-foci periodic points}}: Let $\pi(P)$ and $\pi(Q)$ be  the periods of $P$ and $Q$.
We assume that 
 $f^{\pi (P)}$ 
and $f^{\pi (Q)}$
are $C^r$ linearisable  in small neighbourhoods 
$U_P$ of $P$ and $U_Q$ of $Q$.  
Denote  the eigenvalues of
$Df^{\pi (P)}(P)$ and $Df^{\pi (Q)}(Q)$  
by
\begin{equation}
\label{e.linear-local}
\begin{split}
& \big(\lambda_P,\sigma_P\,e^{-2\pi i\varphi_P},\sigma_P\,e^{2\pi i\varphi_P}\big) \mbox{ where\,\, $0<|\lambda_P|<1<\sigma_P,\, \varphi_P\in [0,1]$,}
\\
&\big(\lambda_Q\,e^{-2\pi i\varphi_Q},\lambda_Q\,e^{2\pi i\varphi_Q},\sigma_Q\big)
\mbox{ where\,\, $0<\lambda_Q<1<|\sigma_Q|, \, \varphi_Q\in [0,1]$.}
\end{split}
\end{equation}
We assume that
\begin{equation}\label{e.espectralconditions}
0<\Big|
\big||\lambda_{P}|^{\frac1{2}}\,\sigma_{P}
\big|^{\eta}\sigma_Q \Big|<1,\quad\mbox{where}\quad \eta=\dfrac{\log|\lambda_{Q}^{-1}|}{\log |\sigma_P|}.
\end{equation}

In what follows, we assume 
that  in the linearising local coordinates the sets $U_P$ and $U_Q$ are
the ``cubes'' $[-a_P,a_P]^3$ and $[-a_Q,a_Q]^3$, for some 
$a_P,a_Q>0$.
For simplicity, we also assume that the periods $\pi(P)$ and $\pi(Q)$ are equal to one. 

\subsubsection{Nontransverse intersections and transition maps}\label{ss.semilocal}
 
 $\,$
 
 \medskip

\noindent
\textbf{(B)} \emph{Quasi-transverse intersection and its transition map}: The one-dimensional invariant manifolds of $P$ and $Q$ intersect \textit{quasi-transversely} along the orbit of a
 point $X=X_f$, that is $X\in W^{\mathrm s}(P,f)\cap W^{\mathrm u}(Q,f)$ and
$$
T_X W^{\mathrm s}(P,f)+T_X W^{\mathrm u}(Q,f) =T_X W^{\mathrm s}(P,f)\oplus T_X W^{\mathrm u}(Q,f).
$$

After replacing $X$ by some iterate, we can assume that $X\in U_Q$.
Associated to $X$ there is 
a  \textit{transition map} corresponding to some iterate of $f$
 going from $U_Q$  to $U_P$ 
 defined as follows.
There are $N_1\in \mathbb{N}$
such that 
$$
f^{N_1}(X)\eqdef \widetilde{X}\in U_P
\quad \mbox{and} \quad
f^i (X) \not\in U_P \quad \mbox{for every} \quad 0\leqslant i < N_1
$$  
and a 
 small neighbourhood $U_{X}$ of $X$ contained in $U_Q$ 
 such that  
 $$
 f^{N_1}(U_X) \eqdef U_{\widetilde X} \subset U_P.
 $$
In the local coordinates at $P$ and $Q$, the restriction $\mathfrak{T}_1$ of $f^{N_1}$ to
$U_X$  is of the form:
\begin{equation}\label{e.transition1}
\mathfrak{T}_1 (X+W) =
f^{N_1}(X+W)=\widetilde{X}+A(W)+\widetilde{H}(W),
\end{equation}
where
\begin{equation}
\label{e.Aisoftheform}
A=\begin{pmatrix} \alpha_1 & \alpha_2 & \alpha_3 \\
0  & \beta_2 & 0\\
0 & 0 & \gamma_3
\end{pmatrix},
\qquad \alpha_1 \beta_2 \gamma_3 \neq 0,
\end{equation}
and
 $\widetilde{H}:\mathbb{R}^3\to\mathbb{R}^3$  is  such that 
$\widetilde{H}(\textbf{0})=\textbf{0}$ and $D\widetilde{H}(\textbf{0})$ is the null matrix.
Note that  $\alpha_1 \beta_2 \gamma_3 \neq 0$ is not an additional assumption since
$f^{N_1}$ is a diffeomorphism.

\medskip

\noindent
\textbf{(C)} \emph{Heterodimensional tangency 
and its transition map}: 
The two-dimensional invariant manifolds of $P$ and $Q$ intersect along the orbit of a point $Y=Y_f$ that is a
\textit{heterodimensional tangency}, that is,  the orbit of $Y$ is contained in the set 
$$
\big(W^{\mathrm u}(P,f)\cap W^{\mathrm s}(Q,f)\big)\setminus \big(W^{\mathrm u}(P,f)\pitchfork W^{\mathrm s}(Q,f) \big).
$$

As above, after replacing $Y$ by some iterate, we can assume that $Y\in U_P$.
Associated to $Y$ there is 
a  \textit{transition map} corresponding to some iterate of $f$
 going from $U_P$  to $U_Q$ 
defined as follows.
There are $N_2\in \mathbb{N}$
such that 
$$
f^{N_2}(Y)\eqdef \widetilde{Y}\in U_Q\quad \mbox{and}\quad
f^i (Y) \not\in U_Q \quad \mbox{for every} \quad 0\leqslant i < N_2
$$  
and a 
 small neighbourhood $U_{Y}$ of $Y$ contained in $U_P$ 
 such that  
 $$
 f^{N_2}(U_Y) \eqdef U_{\widetilde Y} \subset U_Q.
 $$
In the local coordinates at $Q$ and $P$, the restriction $\mathfrak{T}_2$ of $f^{N_2}$ to
$U_Y$  is  of the form: 

\begin{equation}\label{e.transition2}
\mathfrak{T}_2(Y+W)=f^{N_2}(Y+W)=\widetilde{Y}+B(W)+{H}(W),
\end{equation}
where $B$ is a quadratic map of the form
\begin{equation}\label{e.Bisoftheform}
B \begin{pmatrix} x\\y \\ z \end{pmatrix}=
\begin{pmatrix} a_1 x + a_2  y+ a_3 z \\
b_1 x+ b_2 y^2+b_3 z^2+b_4 y z \\
c_1 x + c_2  y + c_3 z 
\end{pmatrix}, \qquad 
b_1( a_2 c_3- a_3c_2) \neq 0.
\end{equation}
where 
${H}\colon\mathbb{R}^3\to\mathbb{R}^3$ is a map such that 
${H}(\textbf{0})=\textbf{0}$, $D{H}(\textbf{0})$ is the null matrix, and 
$$
\frac{\partial^2}{\partial y^2}{H}_2(\textbf{0})=\frac{\partial^2}{\partial z^2}{H}_2(\textbf{0})=\frac{\partial^2}{\partial y\partial z}{H}_2(\textbf{0})=\textbf{0},
$$ 
here 
${H}_i$ is the $i$-th component of $\widetilde{H}$. 
Note that $b_1( a_2 c_3- a_3c_2) \neq 0$ is not an additional assumption since
$f^{N_2}$ is a diffeomorphism.

The constants $a_1,...,c_3$ in the definition of $B$ satisfy the following 
conditions 
\begin{equation}
\label{e.ctes1sem2}
c_2=c_3 ,\qquad \gamma_3(a_3- a_2) >0
\end{equation}
that will guarantee  the convergence of the renormalisation scheme. 

\begin{notation}\label{n.elementsofthecycle}{\em{Given $f\in\mathcal{H}^r_{\BH}(M)$
we say that $P,Q,X$, and $Y$ are the {\em{elements of the cycle}} of $f$ and that
$N_1$ and $N_2$ are the {\em{transition times of the cycle.}}
}}
\end{notation}

 \begin{notation}[Coordinates of the heteroclinic points]\label{nr.coordinates}
{\em{In what follows, we will assume that, in our local coordinates, the heteroclinic 
points above are of the form: 
\[
\begin{split}
&\widetilde{X} =(1,0,0), \quad Y =(0,1,1) \quad \mbox{(in the neighbourhood $U_P$)},\\
&{X} =(0,1,0), \quad \widetilde{Y} =(1,0,1) \quad \mbox{(in the neighbourhood  $U_Q$)}.
\end{split}
\]
}}
\end{notation}

\subsubsection{Parameters of the transition maps} \label{ss.HrBH}
To each diffeomorphism $f$ satisfying ({\bf{A}})-({\bf{C}}) and $\xi>1$ we associate the following parameters
 \begin{equation} \label{e.barsigma}
 \bar \varsigma=\bar \varsigma(\xi,f)\eqdef (\varsigma_1,\varsigma_2,
\varsigma_3,\varsigma_4,\varsigma_5)\in\mathbb{R}^5,
\end{equation}
where
\begin{equation}
\begin{split}\label{e.2.24}
\varsigma_1&\eqdef\frac{\beta_2(a_2+a_3)}{\sqrt{2}},\,\,\,\,\,
\varsigma_2\eqdef\frac{\beta_2^2(b_2+b_3+b_4)}{2},\,\,\,\,\,
\varsigma_3\eqdef \xi^2\left(\frac{b_2+b_3-b_4}{(a_3-a_2)^{2}}\right),
\\
\varsigma_4&\eqdef \xi\sqrt{2}\left(\frac{\beta_2(b_3-b_2)}{a_3-a_2}\right),\,\,\,\,\,
\varsigma_5\eqdef \frac{\beta_2(c_2+c_3)}{\sqrt{2}},
\end{split}
\end{equation}
here $\beta_2$ is as in  \eqref{e.Aisoftheform}
and  $a_1,\dots, c_3$ are as in \eqref{e.Bisoftheform}.

 \begin{defi}[The set  $\mathcal{H}^r_{\BH}(M)$]
 \label{d.HB}
{\em{The set $\mathcal{H}^r_{\BH}(M)$ consists of the $C^r$ diffeomorphisms $f$ 
 satisfying ({\bf{A}})-({\bf{C}}) 
 such that
 \begin{equation}
\label{e.ctesand2}
 (a_2+a_3)(b_2+b_3+b_4)\neq 0
\end{equation}
and  whose vector $\bar \varsigma(\xi,f)$ satisfies
$$
 (\xi,\varsigma_1^2\varsigma_3\varsigma_2^{-1},\varsigma_1\varsigma_4\varsigma_2^{-1}) \in 
(1.18,1.19)\times  (-\ve_{\mathrm{BH}}, \ve_{\mathrm{BH}})^2,
$$ 
where $\ve_{\mathrm{BH}}$ is a number fixed in Theorem~\ref{t.BH-DKS}.
}}
\end{defi}

\begin{remark} \label{rk:constants1}
{\em{
Equations~\eqref{e.Aisoftheform} and~\eqref{e.ctesand2}  implies that $\varsigma_1\,\varsigma_2\,\varsigma_5\neq 0$. These conditions are used to 
get  blender-horseshoes in the renormalisation scheme.
}}
\end{remark}

\subsection{Geometry of the cycle: the sets $ \mathcal{H}^r_{\BH, \mathrm{h}}(M)$, $ \mathcal{H}^r_{\BH, \mathrm{e}}(M)$, and 
$ \mathcal{H}^r_{\BH, \mathrm{e}^+}(M)$}
\label{ss.commentsonthegeometry} For $R=P,Q$ consider 
$$W^{\ast}_{\mathrm{loc}}(R,f)\eqdef C\big(R, W^{\ast}_{\mathrm{loc}}(R,f) \cap U_R\big),\quad \ast =\rs, \ru,$$
 here $C(x,A)$ is the connected component of the set $A$ containing the point $x$.

The next definition classifies the two types of heterodimensional tangencies that we will consider.
Note that given any $f\in\mathcal{H}^r_{\BH}(M)$ the set
$U_P\setminus  W^{\mathrm u}_{\mathrm{loc}}(P,f)$ has two connected components.

 \begin{defi}[Elliptic and hyperbolic tangencies]
 \label{d.tan}
{\em{The heterodimensional tan\-gen\-cy at $Y$ is  
 \textit{elliptic} if there is  a neighbourhood
$\V$ of $Y$ in $W^{\mathrm s}(Q,f)\cap U_P$ such that the set $\V\setminus \{Y\}$ 
is contained in a connected component of 
$U_P\setminus  W^{\mathrm u}_{\mathrm{loc}}(P,f)$.
The tangency is \textit{hyperbolic}  if 
every neighbourhood  of $Y$ in $W^{\mathrm s}(Q,f)$
contains points  in both  components of $U_P\setminus  W^{\mathrm u}_{\mathrm{loc}}(P,f)$.
}}
\end{defi}
 In  Figure~\ref{fig:S}, the heterodimensional tangency in the left-hand side is elliptic while the
one in the right-hand side is hyperbolic.

We observe that if $f\in \mathcal{H}^r_{\BH}(M)$ then the heterodimensional  tangency at the point $Y$ is either hyperbolic or elliptic. We split the
set 
$ \mathcal{H}^r_{\BH}(M)$
in two parts,
$ \mathcal{H}^r_{\BH, \mathrm{h}}(M)$ and $ \mathcal{H}^r_{\BH, \mathrm{e}}(M)$ consisting of hyperbolic and elliptic heterodimensional tangencies, respectively.

For diffeomorphisms in  $f\in  \mathcal{H}^r_{\BH, \mathrm{e}}(M)$ we need to take in consideration the relative position of the tangency and the quasi-transverse heteroclinic points.
We consider the subset $\mathcal{H}^r_{\BH, \mathrm{e}^+}(M)$ of $\mathcal{H}^r_{\BH, \mathrm{e}}(M)$ such that
(with the notation above) the set $\V\setminus \{Y\}$  and $\widetilde X$ are in the same connected component of $U_P\setminus  W^{\mathrm u}_{\mathrm{loc}}(P,f)$, see Figure~\ref{fig:S}.
These geometrical considerations have the same flavour of those in \cite[Section 2]{PalTak:87}.
 
\section{Translation and rotation-like perturbations}
\label{s.translationandrotation}

In this section, we describe the two types of  $C^r$ perturbations used in our constructions.
We start by introducing a class of auxiliary bump functions.

\begin{Cav}
{\em{For simplicity, throughout this paper,  we will use the term {\em{perturbation}}
 to refer to arbitrarily small ones.}}
\end{Cav}

 \subsection{Auxiliary bump functions}\label{ss.auxilirybump}
 Consider a family of $C^r$ bump functions $b^\theta$, $\theta>1$, such that 
 \begin{equation}
\label{e.bu}
b^\theta(x)= \begin{cases}
                0\,, & |x|\geqslant \theta,\\
                 0 \leqslant b^\theta (x) \leqslant 1\,, &1 \leqslant |x| \leqslant \theta,\\
                  1\,, & |x|\leqslant 1.
                    \end{cases}
\end{equation}    
Associated to $b^\theta$ we consider the family of bump functions
\[
b^\theta_\rho (x) \eqdef b^\theta \left( \frac{x}{\rho} \right), \quad \rho>0
\]
and  the three-dimensional bump functions 
\begin{equation}
\label{e.bump}
\Pi^\theta_\rho \colon \mathbb{R}^3\to[0,1],\quad \Pi^\theta_\rho (x,y,z)=
b^\theta_\rho(x) \, b^\theta_\rho(y) \, b^\theta_\rho (z).
\end{equation}
Denote by $B(x,\tau)$ the open ball in $\mathbb{R}^3$  with center $x$ and radius $\tau$
and  by $\Vert \cdot \Vert_r$ the $C^r$ norm.
Note that the support of $\Pi^\theta_{\rho}$ is the closure of $B(\textbf{0}, {\theta \rho})$ and  that 
 \begin{equation}\label{e.Pinorm}
 \Vert \Pi^\theta_\rho \Vert_{r} \leqslant 
(\Vert b^\theta \Vert_r)^3 \,\rho^{-r}.
\end{equation}
In what follows, for simplicity, when $\theta=2$ we write 
$b^2_\rho=b_\rho$ and
$\Pi^2_\rho = \Pi_\rho$.

\subsection{Translation-like perturbations}
\label{ss.translationlike}
Given a point $Z_0\in \mathbb{R}^3$,
a vector  $\widetilde w \in \mathbb{R}^3$, and small $\rho>0$,
we consider the $C^r$ map
$T_{Z_0,\widetilde w,\rho}: \mathbb{R}^3\to \mathbb{R}^3$
defined by 
\begin{equation}
\label{e.translation-perturbation}
T_{Z_0,\widetilde w, \rho} (Z)=
\quad
\left\{
\begin{split}
& Z+ \Pi_{\rho}(Z-Z_0)\widetilde w,
 \quad
\mbox{if $Z \in B (Z_0, 2\rho)$},\\
& Z, \quad
 \mbox{if $Z\not\in B(Z_0, 2\rho)$}.
\end{split}
\right.
\end{equation}
By construction and by \eqref{e.Pinorm}, it holds
\[
\big\Vert
 T_{Z_0,\widetilde w, \rho}- \mathrm{id} \big\Vert_{r} \leqslant \big\Vert \Pi_{\rho} \big\Vert_{r}\, ||\widetilde w || 
 \leqslant
 (\Vert b \Vert_r)^3 \,\rho^{-r} \, || \widetilde w||.
\]
Therefore, for small $||\widetilde w||$, 
the map $ T_{Z_0,\widetilde w,\rho}$ is a $C^r$ perturbation of the identity supported in $B(Z_0,2{\rho})$. 
Finally observe that
$$
T_{Z_0,\widetilde w, \rho}\big(B(Z_0, 2\rho)\big)= B (Z_0, 2\rho).
$$

\subsection{Rotation-like perturbations}
\label{ss.rotationlike}
We now consider 
maps
$I^x_{\omega},\,I^y_{\omega}\colon \mathbb{R}^3\to\mathbb{R}^3,$
$\omega \in [-\pi, \pi]$,
defined by 
$$
I^x_{\omega}\eqdef \left( \begin{array}{ccc}
1 & 0 & 0 \\
0 & \cos 2\pi\omega & -\sin 2\pi\omega \\
0 & \sin 2\pi\omega & \cos 2\pi\omega\end{array} \right),
\qquad
I^y_{\omega}\eqdef \left( \begin{array}{ccc}
\cos 2\pi\omega & 0 & -\sin 2\pi\omega \\
0 & 1 & 0\\
\sin 2\pi\omega & 0 & \cos 2\pi\omega \end{array} \right),
$$
and for $\theta>1$ and $\kappa>0$  their associated  $C^\infty$ diffeomorphisms 
\begin{equation}\label{e.rotationfamily}
R^\ast_{{\omega, \theta, \kappa}} \colon \mathbb{R}^3\to \mathbb{R}^3, \quad R^\ast_{\omega, 
 \theta, 
\kappa}(W)=
I^\ast_{b^\theta (\kappa ||W||) \, \omega}
(W^T), \quad \ast=x,y,
\end{equation}
where $W^T$ denotes the  transpose of the vector $W\in\mathbb{R}^3$.

Note that the restriction of $R^\ast_{{\omega, \theta, \kappa}}$  to the set $[-\kappa^{-1},\kappa^{-1}]^3$ coincides with $I^\ast_{\omega}$ and  $R^\ast_{{\omega, \theta, \kappa}}
$ is the identity map in the complement of  $[-\theta \kappa^{-1}, \theta \kappa^{-1}]^3$.  Note also
$$
R^\ast_{{\omega,  \theta, \kappa}}\big( [-\theta \kappa^{-1}, \theta \kappa^{-1}]^3\big)= 
[-\theta \kappa^{-1}, \theta \kappa^{-1}]^3, \quad \ast=x,y,
$$ 
and that  there is a constant $C(\theta, \kappa)>0$ such that 
$$
\big\Vert
R^\ast_{{\omega,  \theta,  \kappa}}- \mathrm{id} \big\Vert_{C^r} < C(\theta, \kappa)  |{\omega}|,
\quad \ast=x,y.
 $$  
Thus, for every $\omega$ small enough, 
the map $R^\ast_{{\omega},  \theta,  \kappa}$ is a  $C^r$ perturbations of identity supported in
$[-\theta \kappa^{-1}, \theta \kappa^{-1}]^3$.

\section{New heteroclinic and homoclinic intersections}
\label{s.newquasi}
Recall the definitions of the sets $\mathcal{H}^r_{\BH}(M)$, $\mathcal{H}^r_{\BH, \mathrm{h}}(M)$, and 
$\mathcal{H}^r_{\BH, \mathrm{e}^+}(M)$ in Section~\ref{ss.commentsonthegeometry}.
The main result of this section is 
Proposition~\ref{p.Lnbis}  claiming that for every  $f$ in $\mathcal{H}^r_{\BH, \mathrm{h}}(M)$ (resp. $\mathcal{H}^r_{\BH, \mathrm{e}^+}(M)$)
there are local $C^r$ perturbations $f_\ve$ in
$\mathcal{H}^r_{\BH, \mathrm{h}}(M)$ (resp. $\mathcal{H}^r_{\BH, \mathrm{e}^+}(M)$) of $f$
with pairs of  additional quasi-transverse heteroclinic points in $W^{\mathrm s} (P,f_\varepsilon) \cap W^{\mathrm u} (Q,f_\varepsilon)$
and additional transverse homoclinic points in  $W^{\mathrm s} (Q,f_\varepsilon) \pitchfork W^{\mathrm u} (Q,f_\varepsilon)$. The proof of this proposition is done in Section~\ref{ss.proofpLnbis}.
To prove it, in Section~\ref{ss.preliminaryperturbation}, we state some preliminary results about the invariant manifolds
of the saddle-foci in the cycle.
In Section~\ref{ss.newtransitions}, we study the transitions  associated to the new heteroclinic points.
Finally, in Section~\ref{ss.unstable},  we consider parameterisations of special unstable discs 
throughout the new heteroclinic  points contained the unstable manifold of $Q$. The unfolding the cycle associated to these heteroclinic points will provide 
unstable discs intersecting robustly the stable manifold
of the blender-horseshoes. We now go to the details.
 
 Given $f\in\mathcal{H}^r_{\BH}(M)$ with elements $P,Q,X,Y$  (recall Notation~\ref{n.elementsofthecycle}) 
 define the closed 
 invariant set
 \begin{equation} \label{e.vdelciclo}
 \Gamma_{P,Q,X,Y}(f)\eqdef \mathrm{Orb}(X,f)\cup \mathrm{Orb}(Y,f)\cup \{P,Q\},
\end{equation} 
where $\mathrm{Orb}(W,f)$ denotes the $f$-orbit of the point $W$. 

Recall also the neighbourhoods $U_X$ and $U_Y$ of
$X$ and $Y$ in Section~\ref{ss.semilocal}.
}

In what follows, we use the notation $d_r(f,g)$ for the $C^r$ distance between two maps 
$f,g\in \mathrm{Diff}^r(M)$.  

\begin{prop}\label{p.Lnbis}
Let $f\in 
 \mathcal{H}^r_{\BH, \ast}(M)$, $\ast=\mathrm{h},\mathrm{e}^+$,
with elements $P,Q,X,Y$ and transition times $N_1$ and $N_2$.
For every $\varepsilon, \delta >0$ there is
$f_\varepsilon\in \mathcal{H}^r_{\BH,\ast}(M)$ with  $d_r(f,f_\ve)<\ve$
such that:
\begin{itemize}
\item [(1)]
 $f_\varepsilon$ coincides with $f$ on the sets $ \Gamma_{P,Q,X,Y}(f)$ and 
 $$ 
 \bigcup_{i=0}^{N_1-1} f^i(U_X) \, \cup \, \bigcup_{i=0}^{N_2-1} f^i(U_{Y,\ve})
 $$
 where $U_{Y,\ve}$ is a neighbourhood of $Y$ contained in $U_Y$ depending on $\ve$. 
 \item[(2)]
$f_\varepsilon$ has two quasi-transverse heteroclinic points
$$
X_{1,\varepsilon},\, X_{2,\varepsilon} \in  f_\ve^{-N_1} \big(W^{\mathrm s}_{\loc} (P,f_\varepsilon)\big)\cap W^{\mathrm u} (Q,f_\varepsilon) \cap B(X,\delta)
$$
such that
  $\mathrm{Orb}(X_{1,\varepsilon},f_\varepsilon)$, $\mathrm{Orb}(X_{2,\varepsilon},f_\varepsilon)$, and $\mathrm{Orb}(X,f_\varepsilon)$ are
pairwise disjoint and  $X_{1,\varepsilon}, X_{2,\varepsilon} \to X$  as $\ve\to 0$.
\item[(3)]
 $f_\varepsilon$ has two transverse intersection points 
$$
Z^\pm_{\varepsilon} \in
W^{\mathrm s} (Q,f_\varepsilon)\pitchfork W^{\mathrm u}_{\loc} (Q,f_\varepsilon)
$$
such that in the local coordinates
$$
Z^\pm_{\varepsilon} =(0, 1 \pm \zeta^\pm_\ve, 0), \quad 0< \zeta^\pm_\ve< \delta.
$$
\end{itemize}
\end{prop}

A preliminary step of the proof of this proposition is Lemma~\ref{l.closure} in 
Section~\ref{ss.preliminaryperturbation} claiming
that the closure of the one dimensional  invariant manifold  of $P$ (resp. $Q$)
contains the two dimensional invariant manifold of $Q$ (resp. $P$).

\subsection{Density properties of  $W^\mathrm{s} (P,f)$ and  $W^\mathrm{u} (Q,f)$} \label{ss.preliminaryperturbation}
Consider $f\in \diff^r(M)$ with a 
 saddle focus $R$ with $f(R)=R$ such that the eigenvalues of $Df(R)$ are
$\lambda\in \mathbb{R}$ and $\sigma \,e^{\pm 2\,\pi i \varphi}\in \mathbb{C}$, where 
$0<|\lambda | <1<\sigma$ and $\varphi\in[0,1)$, and that is  
$C^r$ linearisable in a neighbourhood
$U_{R}$ of $R$.   We identify $U_R$  with the Cartesian product of the local invariant manifolds of $R$  (where the $x$- and $yz$-spaces are the stable and unstable eigenspaces of $Df(R)$, respectively.) We assume that there are (see Figure~\ref{fig:discos}):
\begin{itemize}
\item
A  one-dimensional 
$C^r$ disc $L\subset U_R$   such that $L$ is  quasi-transverse to $W^\mathrm{s}_{\mathrm{loc}}(R,f)$ at some point $W$ in interior of
$L$. We let $L_+$ and $L_-$ the two connected components of $L\setminus\{W\}$.
\item
A two-dimensional  $C^r$ disc $S\subset U_R$ intersecting transverselly 
$W^\mathrm{u}_{\mathrm{loc}}(R,f)$ in a curve $\gamma$ which is not contained in any radial direction of $W^\mathrm{u}_{\mathrm{loc}}(R,f)$
(i.e., a straight-line  containing the origin).
In this case, we say  that 
the curve $\gamma$ has a \em{nontrivial radial projection}.
\end{itemize}

\begin{figure}[h]
\centering
\begin{overpic}[scale=0.09,
]{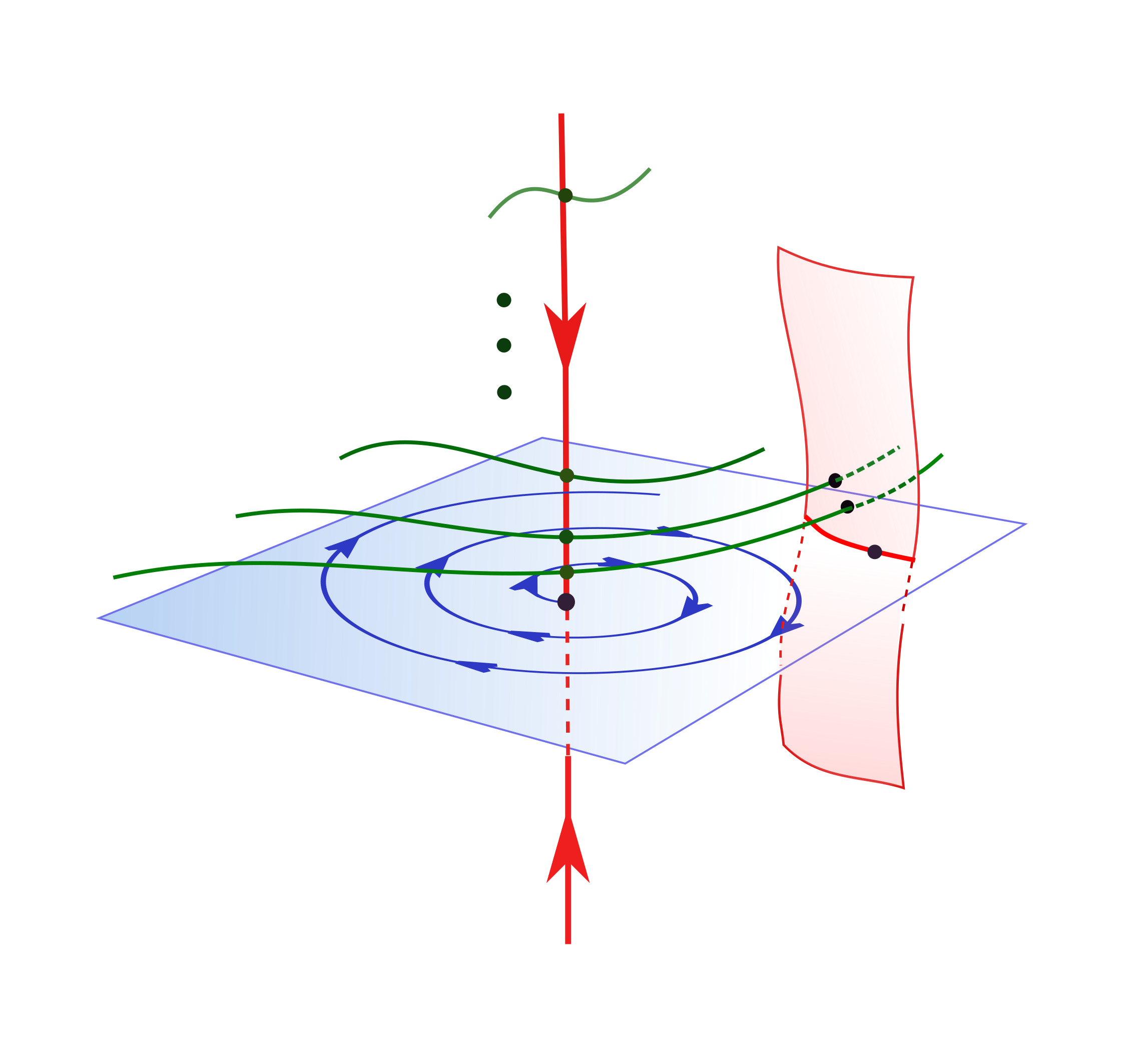} \scriptsize
             \put(112,168){\large{$L^+$}}
               \put(79,167){\large{$L^-$}}
                \put(10,105){\large{$f^j(L)$}}
                 \put(150,150){\large{$S$}}
                       \put(108,76){\Large{$R$}}    
                          \put(150,81){\Large{$\gamma$}}                                   
 \end{overpic}
\caption{The discs $L, L^\pm$ and $S$ and the curve $\gamma$}
\label{fig:discos} \end{figure}
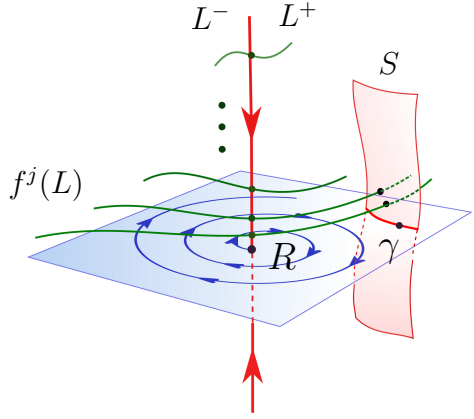

We need the following simple auxiliary  lemma.
 
 \begin{lemma}[Accelerating angles]\label{l.closure} Consider a diffeomorphism $f$, a saddle $R$, a disc $L$, a local surface $S$, 
and a curve 
$\gamma \subset S \cap W^\mathrm{u}_{\mathrm{loc}}(R,f)$  as above. 
Then there is $g$ arbitrarily $C^r$ close to $f$ such that
$W^\mathrm{\ast}_{\mathrm{loc}}(R,g)= W^\mathrm{\ast}_{\mathrm{loc}}(R,f),$ $\ast = \rs, \ru$, and
\begin{itemize}
\item[(a)] 
$W^\mathrm{u}_{\mathrm{loc}}(R,g)$ is simultaneously  contained in
the closure of the sequences of discs $\big(g^{j}(L_+)\big)$
and 
$\big(g^{j}(L_-)\big)$, $j\geqslant 1$,
and 
\item [(b)] there are infinitely many  $j_\pm \geqslant 1$ such that
$g^{j_\pm}(L_\pm)$
meets transversely $S$ at points arbitrarily close to $\gamma$. 
\end{itemize}
\end{lemma}

\begin{proof}
The result is obvious if the argument $\varphi$ of $Df(R)$ is irrational, in that case we can take $g=f$.
Otherwise, it is enough to consider a sequence  $(\alpha_j)\to 0 $ such that
$\varphi + {\alpha_j}$ is irrational,
rotations $I^x_{\alpha_j}$ (defined on $U_R$) with argument $\alpha_j$
(recall the definition in Section~\ref{ss.rotationlike}),
and local perturbations  $g=f_j$ of $f$ of the form $f_j = I^x_{\alpha_j} \circ f$ in the set $U_R$. These perturbations can be chosen supported on an small neighbourhood of the closure of $U_R$.
\end{proof}

\begin{remark} \label{r.alcontrario}
{\em{There are the corresponding version of Lemma~\ref{l.closure} for saddle foci with index one.}}
\end{remark}

\begin{remark}\label{r.red}
{\em{
Changing the surface ``$S$" by a one-dimensional disc and ``transversality'' by 
``quasi-transversality'', the part (b) of Lemma~\ref{l.closure}  can be stated as follows: 
there is arbitrarily large $j_\pm$ such that
$g^{j_\pm}(L_\pm)$
meets quasi-transversely $S$  at some point arbitrarily close to $\gamma$.}}
\end{remark}

\begin{remark}\label{r.theperturbationremains}
{\em{Assume that $f\in \mathcal{H}^r_{\BH}(M)$ and
that $g$ is obtained perturbing $f$ using Lemma~\ref{l.closure}. Then $g$ can be taken such that 
$g\in \mathcal{H}^r_{\BH}(M)$
and  $\Gamma_{P,Q,X,Y}(g)=\Gamma_{P,Q,X,Y}(f)$, recall  \eqref{e.vdelciclo}. }}
\end{remark}

\subsection{Proof of Proposition~\ref{p.Lnbis}}
\label{ss.proofpLnbis}

We first consider hyperbolic tangencies, that is, we assume that
 $f\in \mathcal{H}^r_{\BH,\mathrm{h}}(M)$.
Consider the points $Y$ and $\widetilde{Y}=f^{N_2}(Y)$
corresponding to the heterodimensional tangency in
condition {\bf{(C)}}
in Section~\ref{ss.semilocal}
 and their neighbourhoods $U_Y$  and $U_{\widetilde Y}$ in the definition of the transition map 
$\mathfrak{T}_2$ in \eqref{e.transition2}. 

Using the notation in Definition~\ref{d.tan},
we can select  small two-discs 
$$
\mathcal{P}_Y^\rs\subset
W^{\mathrm{s}}(Q,f)\cap U_Y \quad \mbox{and} \quad
\mathcal{P}^\ru_{\widetilde Y}\subset
W^{\mathrm u}(P,f)\cap U_{\widetilde Y}
$$
containing  $Y$ and $\widetilde Y$ (respectively) in their interiors and assume that
$\mathcal{P}_Y^\rs$ contains a pair of disjoint 
surfaces $S_+$ and $S_-$  intersecting transversely $W^\ru_{\mathrm{loc}} (P,f)$ throughout  curves $\gamma_+$ and $\gamma_-$
with nontrivial radial projection,
see Figure~\ref{fig:newint}. 

Consider now the quasi-transverse heteroclinic points $X$ and  $\widetilde{X}= f^{N_1}(X)$ in condition {\bf{(B)}}
in Section~\ref{ss.semilocal}. Fix small $\delta>0$ such that $B(X,\delta)$ is contained in $U_X$.
Consider small curves  
$$
L^\ru \subset U_P\cap f^{N_1} \big( W_{\loc} ^{\mathrm u}(Q,f) \cap B(X,\delta) \big), \quad 
L^\rs\subset W^{\mathrm s}(P,f)\cap U_Q
$$
 containing $\widetilde X$ and $X$ in their interiors, respectively. See Figure~\ref{fig:newint}.
As above, we let $L^\ru_\pm$ the connected components of $L^\ru   \setminus \{\widetilde X\}$.

Fix small  $\varepsilon >0$.
Applying item (b) of Lemma~\ref{l.closure} to $P$, the surfaces $S_\pm$, and the disc $L^\ru$,
we get a diffeomorphism $\widetilde f_\varepsilon$ with 
$d_r (f,\widetilde f_\ve) < \frac{\ve}{2}$
and  arbitrarily large numbers $i_+, i_-\geqslant  0$ such that  
$\widetilde f_\varepsilon^{i_\pm}(L^\mathrm{u}_\pm)$ transversely intersects
$S_\pm$ at some point $Z_{i_\pm}$. The points $Z_{i_\pm}$ can be chosen converging to some point of $\gamma_\pm$.
Item (3) of the proposition follows taking $Z_\varepsilon^\pm= \widetilde f_{\varepsilon}^{-(N_1+i_\pm)}  (Z_{i_\ell})$. 
Note that, arguing as before, we can assume that the  argument of the complex eigenvalue of $D \widetilde f_{\varepsilon}(Q)$ is irrational.

Note that 
by Remark~\ref{r.theperturbationremains},   
 the transitions of $\widetilde f_\ve$ and 
$f$ are the same, therefore $\widetilde f_\ve\in \mathcal{H}^r_{\BH, \mathrm{h}}(M)$
and $\Gamma_{P,Q,X,Y}(\widetilde f_\ve) =  \Gamma_{P,Q,X,Y}(f)$. Moreover, we observe that all perturbations that we will perform in what follows will keep this property.

To prove item (2),
consider small disjoint closed subdiscs  $\widetilde L^\ru_\pm = \widetilde L^\ru_\pm (\varepsilon)\subset L^{\ru}_\pm$
 as follows (see Figure~\ref{fig:newint}): 
  Let $\ell\in \{+,-\}$, there are numbers $i_\ell =i_\ell(\ve)$ with
\begin{itemize}
\item $\widetilde f^j_\varepsilon ( \widetilde L^\ru_\ell) \subset U_P$ for all $j\in \{0,\dots, i_\ell\}$,
\item
 $\widetilde f^j_\varepsilon( \widetilde L^\ru_\ell) \cap U_Y=\emptyset$ for all $j\in \{0,\dots, i_\ell-1\}$ and 
 $\widetilde f_\ve ^{i_\ell} ( \widetilde L^\ru_\ell) \subset U_Y$, 
 \item
  $\widetilde f_\varepsilon ^{i_\ell}  ( \widetilde L^\ru_\ell) \pitchfork S_{\ell}$ at the point ${Z}_{i_\ell}$,
  \item
 the family 
 $$
 \mathcal{L}^\ru\eqdef
 \big\{ \widetilde f_\varepsilon ^{j}( \widetilde L_\ell^\ru) \colon j\in \{0, \dots, i_\ell +N_2\},\,\, \ell \in \{+,-\}\big\}
 $$
consists of  pairwise disjoint sets,
 \item
 $\widetilde f_\varepsilon^{i_\ell  +N_2} ( \widetilde L^\ru_\ell) \subset U_{\widetilde Y}$,\,\, $\ell=+,-$.
  \end{itemize} 

By the definition of the transition $\mathfrak{T}_2$
and the choice of the neighbourhoods $U_Q,U_Y,U_{\widetilde{Y}}$,
the subfamily of $\mathcal{L}^\ru$ given by
 $$
 \mathcal{L}^\ru_0\eqdef
 \mathcal{L}^\ru \setminus
 \{\widetilde  f_\ve^{i_\ell +N_2}( \widetilde L_\ell^\ru) \}
 $$
is disjoint from $U_P$. 

\begin{figure}
\centering
\begin{overpic}[scale=0.053,
]{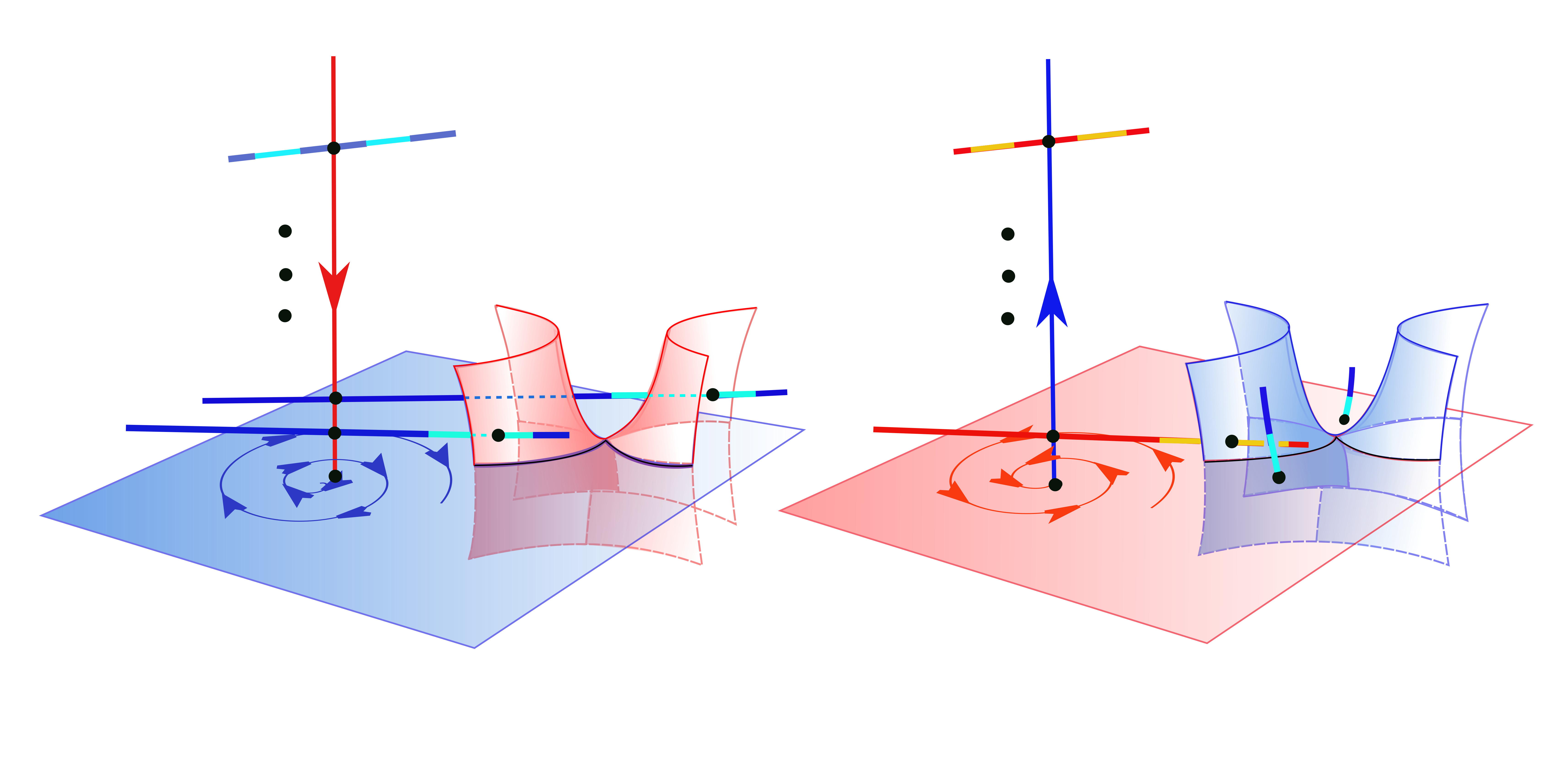} \scriptsize
 \put(80,50){\Large{$P$}}
   \put(83,126){\large{$\widetilde{X}$}}
        \put(51,150){\large{$\widetilde{L}^\mathrm{u}_+$}}
               \put(84,152){\large{$\widetilde{L}^\mathrm{u}_-$}}
                            \put(110,142){\large{$L^\mathrm{u}$}}
     \put(114,58){\large{$\gamma_+$}}     
         \put(165,70){\large{$\gamma_-$}}  
               \put(117,80){\large{$Z_{i_+}$}}     
                       \put(170,90){\large{$Z_{i_-}$}}     
              \put(110,110){\large{$S_+$}}     
                   \put(158,110){\large{$S_-$}} 
\put(255,50){\large{$Q$}}
                          \put(273,142){\large{$L^\mathrm{s}$}}
                \put(219,150){\large{${L}^\mathrm{s}_+$}}
               \put(253,152){\large{${L}^\mathrm{s}_-$}}
                               \put(249,128){\large{${X}$}}
                            \put(293,45){\large{$\widehat{Z}_+$}}  
                               
 \end{overpic}
\caption{The discs $L^\mathrm{u}_\pm$
in $W^\mathrm{u}(Q,f)$ and  
$L^\mathrm{s}_{\pm}$ in $W^\mathrm{s}(P,f)$.
}\label{fig:newint}
\end{figure}

To prove item (2)  of the proposition, consider the homoclinic points of $Q$
\[
 \widehat Z_\ell  \eqdef \widetilde f_\ve^{N_2}(Z_{i_{\ell}})\in W^\rs_{\mathrm{loc}}(Q,\widetilde f_\ve)\pitchfork 
 W^\ru (Q,\widetilde f_\ve),
 \quad \ell =+,-
\]
and take  arbitrarily small 
\begin{equation}
\label{e.puf}
0<\puf < \frac{\ve}{2(\Vert f \Vert_r +\ve)}
\end{equation}
such that
\[
 \begin{split}
 &B(\,\widetilde{Y}, 2\puf\,) \cup  B(\,\widehat{Z}_+, 2\puf\,) \cup  B(\,\widehat{Z}_-, 2\puf\,) \subset U_{\widetilde Y}
 \\
 &\mbox{$B(\,\widehat{Z}_+, 2\puf\,),\, B(\,\widehat{Z}_-, 2\puf\,),\, B(\,\widetilde{Y}, 2\puf\,)$ are pairwise disjoint.}
 \end{split}
\]
Arguing  as above, considering $\widetilde f_\ve^{-1}$ and
applying item (a) of Lemma~\ref{l.closure} to $Q$,  the disc $L^\rs$, and the points
$\widehat{Z}_\pm$ 
 we get  disjoint closed subdiscs $L^\rs_\pm= L^\rs_{\pm} (\puf)$ 
 of $L^\rs$   satisfying the following  conditions 
  (see Figure~\ref{fig:newint}): 
  Let  $K$ be the $C^r$ norm of the map $b_\puf$ in \eqref{e.bu}, there are numbers $k_\pm(\puf)=k_\pm$ such that:
   \begin{itemize}
  \item
 the family of sets
 $$
 \{ \widetilde f^{-i}_{\ve} (L^\rs_{+})\colon i=0,\dots, k_{+}\} \cup \{ \widetilde f_\ve^{-j} (L^\rs_{-} )\colon j=0,\dots, k_{-}\} 
 $$
 is pairwise disjoint,
 \item
 $\widetilde f^{-i}_\ve (L^\rs_{\pm}) \subset \Big(U_Q \setminus B\big(\, \widehat{Z}_\pm, \frac{\puf^{r+1}}{K^3}\,\big)\Big)$ for every $i\in \{0,\dots, k_{\pm}-1\}$,
 \item
  $\widetilde f_\ve^{-k_{\pm} } (L^\rs_{\pm}) \subset B\big(\, \widehat{Z}_\pm, \frac{\puf^{r+1}}{K^3}\,\big)$.
 \end{itemize}

Let $X_{\pm}$ be the closest point of 
  $\widehat Z_\pm$  in $\widetilde f_\ve^{-k_{\pm} } (L^\rs_{\pm})$
 and define the vector 
 \begin{equation}
 \label{e.mayonesa}
 \mathbf{w}_{\pm}\eqdef X_{\pm}-\widehat Z_\pm, \qquad \Vert \mathbf{w}_{\pm} \Vert  \leqslant \frac{\puf^{r+1}}{K^3}.
 \end{equation}
 Using the
 function $\Pi_\puf$ in \eqref{e.bump}, consider the perturbation of the identity given by
 $$
  \vartheta_{\pm,\puf} 
(\widehat{Z}_\pm +W) \eqdef \widehat{Z}_\pm +W + \Pi_\puf (W) \,  \mathbf{w}_{\pm}, \quad \mbox{if} \quad
\widehat{Z}_\pm +W\in
B(\widehat Z_\pm, 2\puf)
$$
 and  the identity otherwise.
 Since 
  $\Vert \mathbf{w}_{\pm} \Vert  \leqslant \frac{\puf^{r+1}}{K^3}$ and 
  $ \Vert \Pi^\theta_\rho \Vert_{r} \leqslant 
K^3 \,\rho^{-r}$ (recall \eqref{e.Pinorm}),
it holds
 $$
 \Vert \vartheta_{	\pm,\puf}-\mathrm{id}\Vert_{r} 
 \leqslant
 \Vert \Pi_\puf \Vert_{r}\cdot \Vert \mathbf{w}_{\pm} \Vert  
 < \puf.
 $$
 
 Finally, consider the perturbation of $\widetilde f_\ve$ defined by
\[
 f_\ve = f_{\ve,\puf} \eqdef \vartheta_{\pm, \puf} \circ \widetilde f_\ve.
 \]
 Recalling the choice of $\puf$ in \eqref{e.puf}, we get 
 $$
d(f,f_{\ve})_r \leqslant
d(f,\widetilde f_{\ve})_r +d(\widetilde f_\ve ,f_{\ve})_r \leqslant
\frac{\ve}{2} +\rho \Vert \widetilde f_\ve \Vert_r
<
\frac{\ve}{2} +\rho (\Vert f \Vert_r+\ve) 
 <\ve.
$$
By construction, $f_{\ve,\puf}$ coincides with $\widetilde f_\ve$ outside 
$\widetilde f^{-1}_\ve\big(  B(\widehat Z_+, 2\puf)\cup B(\widehat Z_-, 2\puf) \big)$ and by \eqref{e.mayonesa} we have that the points $X$, ${X}_{1, \ve}$, and ${X}_{2, \ve}$ with
$$
{X}_{1, \ve}\eqdef f_\ve^{k_+} (X_+),\quad {X}_{2, \ve} \eqdef  f_\ve^{k_-} ({X}_{-})\in  f_{\ve,\puf} ^{-N_1} \big(W^{\mathrm s}_{\loc} (P,  f_{\varepsilon,\puf})\big)\cap
W^\ru(Q, f_{\ve,\puf}) 
$$
are quasi-transverse heteroclinic points of $f_{\ve,\puf}$ 
with different orbits. 

Thus, $f_{\ve,\puf} \in \mathcal{H}^r_{\BH, \mathrm{h}}(M)$ 
and satisfies items (1)-(3)  in the proposition.

We now study the elliptic case when $f\in \mathcal{H}^r_{\BH,\mathrm{e}^+}(M)$.  We apply the variation of Lemma~\ref{l.closure} in  Remark~\ref{r.red} 
and observe that an arbitrarily small modification of the angle provides two transverse intersections (for the same iterate). The rest of the proof is identical to the hyperbolic case.

Finally, note that  in our construction the transitions are preserved, thus if
 $f \in \mathcal{H}^r_{\BH,\ast}(M)$ then $f_{\ve,\puf} \in \mathcal{H}^r_{\BH,\ast}(M)$, $\ast=\mathrm{h},\mathrm{e}^+$. The proof of the proposition is now complete.
 \hfill $\qed$

\subsection{Transitions for the new heteroclinic points}
\label{ss.newtransitions}
Given $f\in  \mathcal{H}_{\BH, \ast}^r (M)$,
$\ast = \mathrm{h}, \mathrm{e}^+$, and small $\varepsilon>0$, consider its perturbation $f_\varepsilon\in \mathcal{H}_{\BH, \ast}^r (M)$ 
and the
heteroclinic points $X_{1, \ve}$ and $X_{2, \ve}$ given by Proposition~\ref{p.Lnbis}. Take disjoint neighbourhoods $U_{1, \ve}$ and $U_{2,\ve}$ of 
 $X_{1, \ve}$ and $X_{2,\ve}$ contained in $U_X$ 
 where  the transition map
 $\mathfrak{T}_1$ in \eqref{e.transition1}
  is defined. By shrinking these neighbourhoods, we can take a small neighbourhood
 $U_{0,\ve}\subset U_X$
 of $X_{0,\ve}=X$ disjoint from $U_{1, \ve}$ and $U_{2,\ve}$.
 We write
\begin{equation}\label{e.xipoint}
X_{i,\ve}\eqdef
X+Z_{i,\ve}, \quad
Z_{i,\ve}\eqdef (x_{i,\ve},y_{i,\ve},z_{i,\ve}), \quad i= 0,1,2.
\end{equation}
 Note that $Z_{1,\ve},Z_{2,\ve}\to \mathbf{0}$ as $\varepsilon\to 0$.

 Denote by $\mathfrak{T}_{1,i,\ve}$ the restriction of $\mathfrak{T}_1$ to  $U_{i,\ve}$, 
  \begin{equation}\label{e.mailen}
\mathfrak{T}_{1,i,\ve}: U_{i,\ve}\to U_{\widetilde{X}},  \quad 
\mathfrak{T}_{1,i,\ve}(Z)= f^{N_1}_\ve (Z)=f^{N_1}(Z), \quad
 i=0,1,2.
\end{equation}
 Hence, recalling \eqref{e.transition1},
$$
\widetilde{X}_{i,\ve}
\eqdef \mathfrak{T}_{1,i,\ve}({X}_{i,\ve})=\widetilde{X}+A(Z_{i,\ve})+\widetilde{H}(Z_{i,\ve}),
$$
Using equation \eqref{e.xipoint} and that  ${X}_{i, \ve} \in   f_{\ve} ^{-N_1}\big(W^{\mathrm s}_{\loc} (P,f_{\varepsilon})\big)\cap U_X$
we get
\begin{equation}\label{e.wastedyears}
\widetilde{X}_{i,\ve}\eqdef 
\widetilde{X}+\widetilde{Z}_{i,\ve}\in W^{\mathrm{s}}_{\mathrm{loc}}(P,f_\ve),
\quad 
\mbox{where}
\quad
\widetilde{Z}_{i,\ve} =A(Z_{i,\ve})+\widetilde{H}(Z_{i,\ve}).
\end{equation}
Note that 
$\widetilde{X}, \widetilde{X}_{i,\ve} \in W^{\mathrm{s}}_{\mathrm{loc}}(P,f_\ve)$ and that 
(in local coordinates) $W^{\mathrm{s}}_{\mathrm{loc}}(P,f_\ve)$ is contained in $\{ (x,0,0)\} \subset U_P$.
Hence
\begin{equation} \label{e.simeone}
\widetilde{Z}_{i,\ve}=(\tilde{x}_{i,\ve},0,0).
\end{equation}

Then  the map $\mathfrak{T}_{1,i,\ve}$ can be written as follows: for $X_{i,\ve}+W\in U_{i,\ve}$ we have
\begin{equation}\label{e.newtransition}
\mathfrak{T}_{1,i,\ve}(X_{i,\ve}+W)=\widetilde{X}_{i,\ve}+A(W)+\widetilde{H}_\ve^{i}(W),
\end{equation}
where
$\widetilde{H}^{i}_\ve \colon \mathbb{R}^3\to \mathbb{R}^3$
is defined by 
\begin{equation}
\label{e.Ee}
\widetilde{H}^{i}_\ve (W)\eqdef
\widetilde{H}(Z_{i,\ve}+W)-\widetilde{H}(Z_{i,\ve}).
\end{equation}
\begin{remark}
{\em{
\label{r.regular}
Let  $\widetilde{H}^{i}_\ve=(\widetilde{H}^{i}_{\ve,1}, \widetilde{H}^{i}_{\ve,2}, \widetilde{H}^{i}_{\ve,3})$.
Then
$$
\widetilde{H}^{i}_{\ve,1}(\textbf{0})=
\widetilde{H}^{i}_{\ve,2}(\textbf{0})=
\widetilde{H}^{i}_{\ve,3}(\textbf{0})=0.
$$
Although   
the maps $\widetilde{H}^{i}_{\ve,j}$ do not satisfy the same ``flat conditions" at $\mathbf{0}$  satisfied by 
the  terms $\widetilde{H}_i$ of $\mathfrak{T}_1$, see~\eqref{e.transition1}, 
the following convergence property  holds: 
\begin{equation*}
\frac{\partial}{\partial x}\widetilde{H}^{i}_{\ve,k}(\textbf{0})\to 0, \quad 
\frac{\partial}{\partial y}\widetilde{H}^{i}_{\ve,k}(\textbf{0})\to 0, \quad
\frac{\partial}{\partial z}\widetilde{H}^{i}_{\ve,k}(\textbf{0})\to 0,\quad
 \ve\to 0, \quad k=1,2,3.
\end{equation*}
 }}
\end{remark}

\subsection{Parameterisations of unstable discs throughout the points $X_{i,\ve}$}\label{ss.unstable}
Take an unitary vector
$\bv_{i,\ve}\in T_{X_{i,\ve}}W^{\mathrm{u}}(Q, f_\ve)$. For small $\delta>0$, 
consider the parameterised  segment of the local unstable manifold
of $Q$
 containing $X_{i,\ve}$ 
in  $U_{i,\ve}$ obtained considering its Taylor expansion,
\begin{equation}\label{e.1seg}
L_{i,\ve}^\mathrm{u}(\delta)\eqdef
\big\{ X_{i,\ve}+t\, \bv_{i,\ve}+\widetilde{\rho}_{i,\ve}(t): |t|<\delta
\big\}\subset W^{\mathrm u}(Q,f_\ve),
\end{equation}
here $\widetilde\rho_{i,\ve}$ is an $C^r$ map
satisfying
\begin{equation}
\label{e.componentes}
 \widetilde{\rho}_{i,\ve}(0)=\frac{d}{dt}\widetilde{\rho}_{i,\ve}(0)=\textbf{0}. 
\end{equation}
\begin{remark}\label{r.llema}
{\em{ By the $\lambda$-lemma 
 we have that
 $\bv_{i,\ve}\to \textbf{e}_2=(0,1,0)$
 and
$\Vert \widetilde{\rho}_{i,\ve}\Vert_r\to 0$ as $\ve\to 0$. 
By a $C^r$ perturbation of $f_{\ve}$, we can assume that $\bv_{i,\ve}=\textbf{e}_2$. To see this, let $\pi_{i,\ve}$ be the plane generated by $ \bv_{i,\ve}$ and $\textbf{e}_2$ in $T_{X_{i,\ve}}M$ and $\alpha_{i,\ve}$ the  smallest angle (modulus $2\pi$) of  the rotation  
map taking
$ \bv_{i,\ve}$ into $\textbf{e}_2$. Note that $\alpha_{i,\ve} \to 0$ as $\ve\to 0$. Performing a rotation-like $C^r$ perturbation  as in Section~\ref{ss.rotationlike}
 at $X_{i,\ve}$ 
 around of the orthogonal direction to $\Pi_{i,\ve}$,
  we get a diffeomorphism $O(\alpha_{i,\ve})$ $C^r$ close to $f_{\ve}$ 
 (that we continue to call $f_\ve$) such that
 \begin{equation}\label{e.1seg1}
L_{i,\ve}^\mathrm{u}(\delta)=
\big\{ X_{i,\ve}+t\,\textbf{e}_2+\widetilde{\rho}_{i,\ve}(t): |t| <\delta
\big\} \subset W^{\mathrm u}(Q,f_\ve), \quad i=1,2.
\end{equation}
With a slight abuse of notation, the higher order terms $\widetilde{\rho}_{i,\ve}$ 
in \eqref{e.1seg1} are  denoted as the ones in \eqref{e.1seg}. As the latter  
are obtained as ``small rotations''  of the terms in \eqref{e.1seg},
they satisfy the flat conditions in~\eqref{e.componentes}.
 }}
\end{remark}

\section{Blender-horseshoes and center-unstable H\'enon-like families}\label{s.BH} 
In this section, we introduce  
blender-horseshoes and their main properties
(Section~\ref{ss.blender-horseshoessubsection}) and
 explain how they may lead to robust tangencies (Section~\ref{ss.blendertubes}).
We also
state
their occurrence in center-unstable H\'enon-like families (Section~\ref{ss.blender-horseshoesreno}).
Finally, we study the geometry of the unstable manifolds
of these blenders  
(Lemma~\ref{l.lem}). 
All blenders considered in this paper are blender-horseshoes, thus if there is no misunderstanding in some cases we will refer 
to them simply as blenders.

\subsection{Blender-horseshoes}\label{ss.blender-horseshoessubsection}
We refrain to give a precise definition of a blender-horseshoe (for details see \cite{BonDia:12,DiaGelSan:19}), instead we will focus on their relevant properties.
We also restrict our discussion to our three  dimensional context. A \textit{blender-horseshoe}  is a locally maximal hyperbolic set  $\Gamma_f$ of a diffeomorphism $f: M \to M$ that is conjugate to the complete full shift on two symbols and satisfies a geometrical
condition stated in Lemmas \ref{l.opensuperposition} and \ref{l.superpositionregion}.
We now go to the details.

There is an open neighbourhood $\Delta$ of $\Gamma$ such that 
$$
\Gamma_f = \bigcap_{i\in \mathbb{Z}} f^i(\,\overline{\Delta}\,) \subset \Delta
$$
The set $\Gamma_f$ is also partially hyperbolic: there is a dominated splitting
with one-dimensional bundles 
$E^\rs\oplus E^\mathrm{cu} \oplus E^{\ru\ru}$ of $T_{\Gamma_f} M$ such that
  $E^\ru\eqdef E^\mathrm{cu} \oplus E^{\ru\ru}$ and $E^\rs$ are the unstable and stable bundles of $\Gamma$, respectively.
  The bundle $E^{\ru\ru}$ is the strong unstable direction.
 We consider a $Df$-invariant cone fields $\mathcal{C}^{\ru\ru}$ and $C^{\mathrm{u}}$ around $E^{\ru\ru}$ and $E^{\mathrm{u}}$, 
 a $Df^{-1}$-invariant cone field $\mathcal{C}^{\rs\rs}$ around $E^{\rss}$, and a center unstable cone field $C^{\mathrm{cu}}$ around $E^{\mathrm{cu}}$.
 The latter is not $Df$ invariant, but the norm of the vectors in $C^{\mathrm{cu}}$    are uniformly expanded by $Df$.

As a hyperbolic set, the blender-horseshoe $\Gamma_f$ has a continuation $\Gamma_g$ for
every $g$ sufficiently close. The important fact is that
for diffeomorphisms nearby these continuations are also blender-horseshoes.
Blender-horseshoes can be also defined for endomorphisms, see \cite[Definition 2.7]{DiaPer:19b}, 
with
the following reformulation of the continuation property:
every map (diffeomorphism or endomorphism) close to an endomorphism with a blender-horseshoe also has a blender-horseshoe with the same reference domain.

 A  key ingredient of a blender-horseshoe
 is its 
{\em{superposition region.}} To describe it  define 
the local stable manifold of $\Gamma_f$ by
\begin{equation}\label{e.localstable}
W^\rs_{\mathrm{loc}} (\Gamma_f, f)\eqdef 
\big\{\,x \in M \, \colon \,f^i(x) \in \overline{\Delta} \,\,\,\, \mbox{for every $\,\,  i\geqslant 0$}\,\big\} 
\end{equation}
and 
observe that 
the set $\Gamma_f$ has two fixed points $P^+$ and $P^-$, called the {\emph{reference fixed points
of the blender}}.  One defines ``large'' 
one-dimensional $\ruu$-discs (contained in $U$) tangent to $\mathcal{C}^{\ru\ru}$  {\em{at the right and at
the left of $W^\rs_{\mathrm{loc}} (P^\pm,f)$.}}  The 
set of such discs 
{{at the right of $W^\rs_{\mathrm{loc}} (P^-,f)$ and at the left of $W^\rs_{\mathrm{loc}} (P^+,f)$}} form the 
\textit{superposition region} denoted by $\mathcal{D}_f$. 
We say that these  $\ruu$-discs are {\em{in-between}}. 

The next two lemmas (see for instance  \cite[Lemma 3.13]{BocBonDia:16} and \cite[Lemma 2.5]{DiaPer:19b}) state an intersection property for discs in the superposition region of the blender that will play a key role to get
robust heterodimensional cycles.

\begin{lemma}[The superposition region]
\label{l.opensuperposition}
Let $\Gamma_f$ be a blender-horseshoe of $f$ and $D$ compact disc whose  interior contains a disc in the superposition region of $\Gamma_f$.
Then for every $C^1$-neighbourhood $\mathcal{U}$  of $D$ there is a   $C^1$-neighbourhood\footnote{Let $r>1$ and $f\in \mathrm{Diff}^r(M)$. Any $C^1$-neighbourhood
of $f$ contains a  $C^r$ neighbourhood of $f$.} $\mathcal{V}$ of $f$ such that every
compact disc in   $\mathcal{U}$ contains a disc in the superposition region of $\Gamma_g$ for every $g\in \mathcal{V}$. 
\end{lemma}

For further explanation of this lemma see Remark~\ref{r.diskstripsinbetween}. See also  Lemma~\ref{l.iterationsofstrips} for a 
 ``quantitative version'' of this result.
 
\begin{lemma}
\label{l.superpositionregion} 
Let $\Gamma_f$ be a blender-horseshoe of $f$ and $D$ a compact disc containing a disc  the superposition region of $\Gamma_f$.
Then $W^\rs_{\mathrm{loc}} (\Gamma_f, f)\cap D\neq\emptyset$ and this intersection is quasi-transverse.
\end{lemma}

This lemma follows from the fact that 
the image of any  disc $D$ in the superposition region contains a disc in the superposition region.
Arguing inductively, it follows that any disc in the superposition region contains a point whose forward orbit is contained in
$\Delta$. The lemma now follows from the characterisation of 
$W^\rs_{\mathrm{loc}} (\Gamma_f, f)$ in \eqref{e.localstable}. 
This construction also guarantees that  the obtained intersection
between $D$ and $W^\rs_{\mathrm{loc}} (\Gamma_f, f)$ is
quasi-transverse: the disks in the superposition region are tangent to a strong unstable cone field,
$W^\rs_{\mathrm{loc}} (\Gamma_f, f)$ is tangent to 
a stable cone field, and these cone fields have no common directions.  

\subsection{Blender-horseshoes: tubes and folding manifolds}
\label{ss.blendertubes}
We  now analise when the
local stable manifold $W^{\rs}_{\mathrm{loc}} (\Gamma_f,f)$   of a blender-horseshoe $\Gamma_f$ has a (robust) tangency with a surface
$S$
``passing throughout its domain $\Delta$".  In \cite{BonDia:12} it is proved that occurrence when  $S$ is a  {\em{folding manifold.}}
Motivated by this fact, we introduce the notion of {\em{$\ru$-tubes}} and prove that they generate folding manifolds
after iterations. In this way, Proposition~\ref{p.tang} provides a mechanism guaranteeing the robust tangencies in Theorem~\ref{t:1}.
We now go into the details of this construction that follows the ideas in \cite[Section 4]{BonDia:12}. 

\subsubsection{Folding manifolds and tubes}\label{d.tubos}
Throughout this section, we consider a diffeomorphism $f$  with a blender-horseshoe $\Gamma_f$ with domain $\Delta$ 
and reference fixed points $P^-$ and $P^+$. Next definition is an extension of \cite[Definition 4.2]{BonDia:12}.

\begin{defi}[Strips, tubes, and folding manifolds]
\label{d.stripstubesetc}
{\em{
Consider  a surface with boundary $S$ of the form
$$
S = \bigcup_{t\in [0,1]} D_t \subset \Delta,
$$
where
$(D_t)_{t\in [0,1]}$  is a family of $\ruu$-discs depending continuously on $t$.
We say that $S$ is a 
\begin{itemize}
\item
 {\em{$\ru$-strip}}   if 
the  family of discs $(D_{t})_{t\in [0,1]}$ is pairwise disjoint  and $S$ is tangent to the unstable cone field $\mathrm{C}^\ru$, 
 \item
  {\em{$\ru\ru$-tube}} if 
the  family of discs $(D_{t})_{t\in [0,1)}$ is pairwise disjoint and $D_0=D_1$.
\item
a {\em{folding surface}} if
the  family of discs $(D_{t})_{t\in [0,1]}$ is pairwise disjoint,
$D_0$ and $D_1$ both intersect $W^{\rs}_{\loc} (P^-,f)$
(or both intersect  $W^{\rs}_{\loc} (P^+,f)$), and $D_t$ is in-between $P^-$ and $P^+$ for every
$t\in (0,1)$.
\end{itemize}
A strip or a tube $S$  is {\em{in-between}} if every 
$D_t$  is in-between $P^+$ and $P^-$. Note that a folding surface cannot be in-between.
}}
\end{defi}

In what follows, we will use the letters $S$, $T$, and $F$ to denote strips, tubes, and folding manifolds, respectively.
The main result of this section is the following:

\begin{prop}\label{p.tang}
Let $\Gamma_f$ be a blender-horseshoe of a diffeomorphism $f$ and  $T$ a $\ruu$-tube in-between.
Then $W^{\rs}_{\loc}(\Gamma_f, f)$ and $T$ have a tangency point.
\end{prop}

An immediate consequence of this proposition is the following version of  \cite[Corollary 4.11]{BonDia:12} where
folding manifolds are replaced by $\ruu$-tubes:

\begin{coro}\label{c.tang}
Let $\Gamma_f$ be a blender-horseshoe of $f$ and  $T$ a $\ruu$-tube in-between
contained in the unstable manifold of a saddle $R_f$ of index two.
Then there is a $C^r$ neighbourhood $\mathcal{V}_f$ of $f$ such that
$W^{\rs}_{\loc}(\Gamma_g, g)$ and 
$W^{\ru}(R_g, g)$  have 
 a tangency point for every $g\in \mathcal{V}_f$.
\end{coro}

Proposition~\ref{p.tang} will follow from \cite[Proposition 4.4]{BonDia:12}: any folding manifold has a tangency
with $W^{\rs}_{\loc} (\Gamma_f,f)$. The main ingredient of the proof in \cite{BonDia:12} is the fact that
the image of
any folding manifold contains a folding manifold,
see \cite[Lemma 4.5]{BonDia:12}. We reformulate that lemma in our context. For that we need to 
analise 
 the central width of iterations of $\ru$-tubes and $\ru$-strips. We refer to see  \cite[Section 2]{DiaGelSan:19} for a detailed analysis
 of iterations of strips. We now go into the details.

First,  a {\em{central curve}} is a curve tangent to the center unstable cone field $\mathcal{C}^{\mathrm{cu}}$. 
The {\em{central  width}} of a $\ru$-strip $S = \cup_{t\in [0,1]} D_t$ is defined by
$$
\mathrm{w}(S) =\inf \{ \mathrm{length}(\ell) \colon \mbox{$\ell\subset S$ is a central curve joining $D_0$ and $D_1$}\}.
$$
Note that there is $\kappa= \kappa (\Gamma_f)>0$ such that $\mathrm{w} (S) \lesp \kappa$ for every $\ru$-strip in-between.

To define the {\em{central width of a $\ruu$-tube $T$}} in-between (denoted with a  slight abuse of notation also by $\mathrm{w} (T)$),
we note that $(\Delta\setminus T)$ has two connected components, one of them is disjoint from $P^-$ and $P^+$.
We denote this component by $\Delta_T$ and let
$$
\mathrm{w}(T) =\sup \{ \mathrm{w}(S) \colon \mbox{$S$ is a $\ru$-strip contained in $\Delta_T$}\}.
$$
Note that the width of any $\ruu$-tube in-between is bounded by the constant $\kappa$ above. 

\begin{lemma}\label{l.perfectamente}
Let $\Gamma_f$ be a blender-horseshoe. Then
there is $\lambda>1$ such that  for every $\ruu$-tube $T$ in-between
 one of the following possibilities holds true:
 \begin{itemize}
  \item [(a)]
 $f(T)$ has tangency with either $W^{\rs}_{\mathrm{loc}}(P^-,f)$ or with
  $W^{\rs}_{\mathrm{loc}}(P^+,f)$,
\item [(b)]
 $f(T)$ contains a folding manifold,
 \item[(c)]
 $f(T)$ contains
a $\ruu$-tube $T'$ in-between  with
$\mathrm{w} (T') \gesp \lambda\, \mathrm{w} (T)$.
 \end{itemize}
\end{lemma}

Proposition~\ref{p.tang} easily follows from this lemma.
Observe that, by the comments above, in cases (a) and (b) we get a tangency between $T$ and
$W^{\rs}_{\loc} (\Gamma_f,f)$ and we are done. Otherwise, we let $T_0=T$ and get a new tube $T_1$ in-between 
such that $T_1 \subset f(T_0)$ and
$\mathrm{w} (T_1) 
\gesp \lambda\, \mathrm{w} (T_0)$.
We can now
apply  Lemma~\ref{l.perfectamente} to $T_1$ and argue recurrently.
But  case (c) cannot occur infinitely many consecutive times: 
in such a case we get a sequence of tubes $(T_n)$ in-between with
$$
T_n \subset f(T_{n-1}) \quad \mbox{and} \quad 
\mathrm{w} (T_n) \gesp \lambda\, \mathrm{w} (T_{n-1}) \gesp  \lambda^n\, \mathrm{w} (T_0).
$$
Since the widths of the tubes $T_n$ 
are bounded by $\kappa$ there is a first  step of the recurrent construction when we fall in cases (a) or (b). 
We now prove the lemma.

\subsubsection{Proof of Lemma~\ref{l.perfectamente}}

We have the following version of Lemma~\ref{l.perfectamente} for $\ru$-strips imported
from \cite[Lemma 4.5]{BonDia:12} (see also the proof of \cite[Proposition 2.3]{DiaGelSan:19}).

\begin{lemma}\label{l.iterationsofstrips}
Consider a blender-horseshoe $\Gamma_f$ of $f$.
Then
there is $\lambda>1$ such that
for every $\ru$-strip $S$ in-between
 then one of the two possibilities holds:
 \begin{itemize}
  \item[(i)]
 $f(S)$ intersects transversely either $W^{\rs}_{\mathrm{loc}}(P^-,f)$ or 
  $W^{\rs}_{\mathrm{loc}}(P^+,f)$,
\item[(ii)]
 $f(S)$ contains 
a $\ru$-strip $S'$ in-between with
$\mathrm{w} (S') \gesp \lambda\, \mathrm{w} (S)$.
\end{itemize}
\end{lemma}
The main ingredient of the proof of this lemma is in Remark~\ref{r.diskstripsinbetween}.

\begin{remark}[Iterations of $\ru$-strips in-between]\label{r.iterationsofstrips}{\em{Lemma~\ref{l.iterationsofstrips} 
implies that the orbit of any $\ru$-strip in-between transversely meets either
$W^s_{\loc} (P^-,f)$ or $W^s_{\loc} (P^+,f)$. More precisely, let $S_0=S$ and assume that
$f(S_0)$ intersects transversely neither $W^{\rs}_{\mathrm{loc}}(P^-,f)$ nor 
  $W^{\rs}_{\mathrm{loc}}(P^+,f)$. In that case,
we get a new strip $S_1$ in-between 
such that $S_1 \subset f(S_0)$ and
$\mathrm{w} (S_1) 
\gesp \lambda\, \mathrm{w} (S_0)$.
We can now
apply  Lemma~\ref{l.iterationsofstrips} to $S_1$ and argue recurrently. As above,
this possibility cannot occurs infinitely many consecutive times.}}
  \end{remark}

We are now ready to prove the lemma. Given any $\ve>0$, associated to the $\ruu$-tube $T_0=T$ there is an internal strip 
$S_0\subset \Delta_{T_0}$, with $\mathrm{w}(S_0)\gesp  \mathrm{w}(T_0)-\ve$. Note that this strip is in-between.
Consider now $f(T_0)$ and assume that cases (a) and (b) in Lemma~\ref{l.perfectamente} do not hold. This implies that the image of $f(S_0)$ 
satisfies (ii) in Lemma~\ref{l.iterationsofstrips}.  We now see  that and in that case there is $\ru$-tube $T_1\subset f(T_0)$ such that $S_1 \subset \Delta_{T_1}$.
Since this holds for every $\ve>0$ item (c) in the lemma follows.

We now explain how the tube $T_1$ is obtained. We will refer to \cite[Section 2]{DiaGelSan:19} for details. A blender-horseshoe has 
a Markov partition associated  to two disjoint ``subrectangles" $\Delta_{\mathbb{A}},\Delta_{\mathbb{B}}\subset \Delta$  such 
that 
$$
\Gamma_f= \bigcap_{i\in \mathbb{Z}} f^i(\Delta_{\mathbb{A}} \cup  \Delta_{\mathbb{B}} ).
$$ 
Moreover, the map that associates to each $x\in \Gamma_f$ the sequence $(\xi_i(x))_{\in \mathbb{Z}}\in 
\{\mathbb{A},\mathbb{B}\}^\mathbb{Z}$ defined by $f^i(x)\in \Delta_{\xi_i(x)}$ is a conjugation between $\Gamma_f$ and the complete shift on the symbols 
$\mathbb{A},\,\mathbb{B}$.
In particular, the reference fixed points of the blender satisfy
$P^-\in \Delta_{\mathbb{A}}$ and $P^+\in \Delta_{\mathbb{B}}$. In what follows, we denote by $f_{\mathbb{E}}$ 
the restriction of $f$ to ${\mathbb{E}}$, $\mathbb{E}\in\{\mathbb{A}, \mathbb{B}\}$. Given a set $X$ we let 
$X_{\mathbb{E}}\eqdef  X \cap \Delta_{\mathbb{E}}$.

Next remark explains the mechanism guaranteeing Lemma~\ref{l.opensuperposition} and is a consequence of the definition of a blender-horseshoe.

\begin{remark}[Iterations of $\ruu$-discs ans $\ru$-strip]\label{r.diskstripsinbetween}
{\em{There is $\lambda>1$ with the following property: 
Consider a $\ruu$-disc $D$ and a $\ru$-strip $S$ in-between.
\begin{itemize}
\item
 Then either $f(D_\mathbb{A})$ or $f(D_\mathbb{B})$ contains a $\ruu$-disc in-between.
\item
Suppose that $f(S_{\mathbb{E}})$ is a $\ru$-strip in-between,
then
$\mathrm{w}  f(S_{\mathbb{E}}) \gesp \lambda \mathrm{w}(S_{\mathbb{E}})$.
\end{itemize}
 }}
\end{remark}

By Remark~\ref{r.diskstripsinbetween},  for a given $t\in [0,1]$ either $f(D_{t, \mathbb{A}})$ or $f(D_{t, \mathbb{B}})$  is a $\ruu$-disc in-between.
We have the following three cases:
\begin{itemize}
\item[(A)]
 $f(D_{t, \mathbb{A}})$  is a $\ruu$-disc in-between for every $t\in [0,1]$, 
\item[(B)]
 $f(D_{t, \mathbb{B}})$  is a $\ruu$-disv in-between for every $t\in [0,1]$, 
\item[(C)]
Cases (A) and (B) do not hold.
\end{itemize}
In case (A), by Remark~\ref{r.diskstripsinbetween} we have  that $T_1=f(T_{t,\mathbb{A}})$ is a $\ruu$-tube such that $\mathrm{w} (T_1) \gesp \lambda \mathrm{w} (T)$.
Analogously, in case (B) $T_1= f(T_{t,\mathbb{B}})$ is a $\ruu$-tube satisfying  $\mathrm{w} (T_1) \gesp \lambda \mathrm{w} (T)$. 
In both cases,  item (3) in the lemma holds. Thus it remains to consider case (iii). 

In case (C),
there is $c\in [0,1)$ such that $f(D_{c, \mathbb{A}})$ is  a $\ruu$-disc in-between but  $f(D_{c, \mathbb{B}})$ is not a $\ruu$-disc in-between (this possibility includes the case where
$f(D_{c, \mathbb{B}})$ is not a $\ruu$-disc). 
After changing the parameterisation of the tube we can take  $c=0$.  Recall that $D_0=D_1$.
Let
\[
\begin{split}
t_1&\eqdef  \sup\{ t\in [0,1] \colon \mbox{$f(D_{s, \mathbb{A}})$ is in-between for all $s\in [0,t]$}\},\\
t_2&\eqdef  \inf\{ t\in [0,1] \colon \mbox{$f(D_{s, \mathbb{A}})$ is in-between for all $s\in [t,1]$}\}.
\end{split}
\]
Note that $t_1=1$ if and only if $t_2=0$.
Moreover, if this does not hold then $t_1\lesp t_2$.
 Thus, a priory, there are 
the following  cases, 
(i) $t_1=1$ and $t_2=0$,
(ii)
$t_1=t_2\in (0,1)$, and
(iii)
$0<t_1< t_2<1$.
Case (i) implies that we are in case (A) above, a contradiction. Thus it can be discarded.

Note that, by continuity,  if $t_1<1$ then
$f(D_{t_1, \mathbb{A}})\cap W^{s}_{\mathrm{loc}}(P^-,f)\ne \emptyset$. Similarly,
if $t_2>0$ then
$f(D_{t_2, \mathbb{A}})\cap W^{s}_{\mathrm{loc}}(P^-,f)\ne \emptyset$. 
Note that the local stable manifolds are tangent to the stable cone field and the $\ruu$-discs are tangent to the strong unstable cone field, hence
the previous intersections are quasi-transverse. 

In case (iii), consider the interval $[t_2-1,t_1]$ and let $\widehat D_t=D_t$ if $t\in [0,t_1]$ and 
$\widehat D_t=D_{t+1}$ if $t\in [t_2-1,1]$. By construction,
$$
S\eqdef \bigcup_{t\in [t_2-1,t_1]} f(\widehat D_{t,\mathbb{A}})
$$ is  a folding manifold, thus item (b) in the lemma holds.

Finally, in case (ii) we have that
$T_1=\bigcup_{t\in [0,1]} f(D_{t,\mathbb{A}})$ is tangent to $W^{s}_{\mathrm{loc}}(P^-,f)$ and we are in case (a) in the lemma. This completes the proof of the lemma.
\qed

\subsection{Blender-horseshoes in the center-unstable H\'enon-like family}
\label{ss.blender-horseshoesreno}
Let us start by defining the H\'enon-like families of endomorphisms that we will consider.

\subsubsection{Center-unstable H\'enon-like families of endomorphisms}
\label{sss.cuHL}
We consider the parameterised families of H\'enon-like endomorphisms
$
E_{\xi,\mu,\bar \varsigma}, \,  G_{\xi, \mu, \bar \eta} \colon\mathbb{R}^3\to \mathbb{R}^3,
$
defined by 
\begin{equation}\label{e.E-henonlike}
\begin{split}
E_{\xi,\mu,\bar \varsigma}(x,y,z)&\eqdef
(\xi\,x +\varsigma_1 \,y,\, \mu + \varsigma_2\,y^2 + 
\varsigma_3\,x^2 + \varsigma_4\,x\,y,\, \varsigma_5\,y), \\
G_{\xi, \mu, \bar \eta}(x,y,z) &\eqdef (y,\mu+y^2+\eta_1\,y\,z+\eta_2\,z^2,\xi\,z+y),
\end{split}
\end{equation}
where $\xi>1, \mu\in \mathbb{R}$,  $\bar \varsigma \eqdef (\varsigma_1,\varsigma_2,
\varsigma_3,\varsigma_4,\varsigma_5)\in \mathbb{R}^5$, and 
$\bar \eta = (\eta_1,\eta_2)\in \mathbb{R}^2$.
These families are called {\em{center-unstable H\'enon-like.}}

These families are conjugate, this  allows us to translate properties from one family to the other.
More precisely:
\begin{remark}
\label{r.conjugadosperonorevueltos}
{\em{ 
  Consider 
  the families of endomorphisms
$$
\widehat{E}_{\xi,\bar \varsigma},\,  \widehat G_{\xi, \bar \eta}\colon \mathbb{R}^4\to \mathbb{R}^4
$$
defined by
\begin{equation}
\label{e.E-Ggorro}
\begin{split}
(\mu,x,y,z) &\mapsto \widehat{E}_{\xi,\bar \varsigma} (\mu, x,y,x) \eqdef \big( \mu, E_{\xi,\mu,\bar \varsigma} (x,y,z) \big),\\
(\mu,x,y,z) & \mapsto
\widehat G_{\xi,\bar \eta}( \mu,x,y,z) \eqdef \big( \mu,  G_{\xi, \mu, \bar \eta}( x,y,z) \big).
\end{split}
\end{equation}
Consider the map
\begin{equation}
\label{e.bareta}
\bar\eta (\bar\varsigma) = (\eta_1 (\bar \varsigma), \eta_2 (\bar \varsigma))=(\varsigma_1^2\varsigma_3\varsigma_2^{-1},\varsigma_1\varsigma_4\varsigma_2^{-1}).
\end{equation}
Suppose that  $\bar\varsigma$ is such that
\[
\varsigma_1\varsigma_2\varsigma_5\neq 0,
\]
then
$\widehat E_{\xi, \bar\varsigma}$ and $\widehat G_{\xi,\bar\eta (\bar\varsigma)}$ are conjugate:
$$
\widehat \Theta^{-1}_{\bar \varsigma} \circ \widehat E_{\xi, \bar\varsigma} \circ \widehat\Theta_{\bar \varsigma}= \widehat G_{\xi,\bar\eta (\bar\varsigma)}.
$$
where 
\begin{equation}\label{e.todaslastetas}
\begin{split}
& \widehat\Theta_{\bar \varsigma}  \colon\mathbb{R}^4\to\mathbb{R}^4,\quad 
\widehat\Theta_{\bar \varsigma}(\mu,x,y,z) = \big(\varsigma_2^{-1} \mu, \Theta_{\bar \varsigma}  (x,y,x)\big),\\
&\Theta_{\bar \varsigma}  \colon\mathbb{R}^3\to\mathbb{R}^3,\quad 
\Theta_{\bar \varsigma}(x,y,z)=
(\varsigma_2^{-1}\varsigma_1z,
\varsigma_2^{-1}y,
\varsigma_2^{-1}\varsigma_5x).
\end{split}
\end{equation}
}}
\end{remark}

\subsubsection{Occurrence of blender-horseshoes}
\label{sss.blender-horseshoesreno}
Consider the set
 \begin{equation}
 \label{e.deltacube} 
 \Delta\eqdef [-4,4]^2\times[-40,22].
 \end{equation}

\begin{teo}[Theorem 1 in \cite{DiaPer:19b}]
\label{t.BH-DKS}
There is  $\varepsilon_{\mathrm{BH}}>0$ such that
for every 
$$
(\xi, \mu, \bar \eta) \in 
\mathcal{O}_{\mathrm{BH}}\eqdef (1.18,1.19)\times (-10,-9)\times (-\varepsilon_{\mathrm{BH}},\varepsilon_{\mathrm{BH}})^2
$$
the 
endomorphism $G_{\xi,\mu, \bar\eta}$ 
has a blender-horseshoe $\Lambda_{\xi, \mu, \bar\eta}$ with domain of reference $\Delta$
such that 
$$
\Lambda_{\xi, \mu, \bar\eta} = \bigcap_{i\in \mathbf{Z}}  G^i_{\xi, \mu, \bar \eta} (\Delta)
\subset \mathrm{interior} (\Delta).
$$
As a consequence, 
every diffeomorphism or endomorphism sufficiently $C^1$ close to $G_{\xi, \mu, \bar \eta}$ has a 
blender-horseshoe in $\Delta$.
\end{teo}

\begin{remark}\label{r.sinnombre}
{\em{
By Remark~\ref{r.conjugadosperonorevueltos}, the map $E_{\xi,\mu, \bar\varsigma}$, with $\bar \varsigma=(\varsigma_1,\dots, \varsigma_5)$, has blender-horseshoes if $\varsigma_1\varsigma_2\varsigma_5\neq 0$ and 
$$
 (\xi,\mu, \varsigma_1^2\varsigma_3\varsigma_2^{-1},\varsigma_1\varsigma_4\varsigma_2^{-1}) \in 
(1.18,1.19)\times (-10,-9) \times   (-\ve_{\mathrm{BH}}, \ve_{\mathrm{BH}})^2.
$$ 
}}
\end{remark}

\medskip

Let us say a few words about the blenders in Theorem~\ref{t.BH-DKS}. Let $P^{\pm}_{\xi,\mu,\bar \eta}$ be the {\em{reference  fixed points}} of the blender $\Lambda_{\xi,\mu,\bar \eta}$ of $G_{\xi,\mu,\bar \eta}$ in in $\Delta$.
The fixed points of   $G_{\xi,\mu}\eqdef G_{\xi,\mu,\bar 0}$ \footnote{With the notation in 
\cite{DiaPer:19b}, say $P^-_{\xi,\mu}=P_{\xi,\mu}$ and $P^+_{\xi,\mu}=Q_{\xi,\mu}$} in $\Delta$, satisfy
\begin{equation}
\label{e.referencesaddle}
P^-_{\xi,\mu}=(p^-_{\xi,\mu}, p^-_{\xi,\mu}, \widetilde p^-_{\xi,\mu}), \qquad
P^+_{\xi,\mu}=(p^+_{\xi,\mu}, p^+_{\xi,\mu}, \widetilde p^+_{\xi,\mu})
\end{equation}
with
\[
\begin{split}
-2.7&<
p^-_{\xi,\mu}  <-2.5, 
\quad
13<
\widetilde p^-_{\xi,\mu}
< 15,
\\
3.5&<
p^+_{\xi,\mu}
< 3.71,
\quad
-20.6<
 \widetilde p^+_{\xi,\mu}
 < -18.4.
 \end{split}
\]
For the map $G_{\xi,\mu}$, the strong unstable cone field is given by
\begin{equation}
\label{e.conodecone}
\mathcal{C}^{\mathrm{uu}} (Z) \eqdef
\big\{\,(u, v, w)\in\mathbb{R}^3\,:\,
\sqrt{u^2 + w^2}<\tfrac{1}{2} |v|\,\big\},
\end{equation}
see \cite[Lemma 3.10]{DiaPer:19b}. We also have that
$$
W^\rs_{\mathrm{loc}}(P^\pm_{\xi, \mu},G_{\xi,\mu}) = \big\{ \,(p^\pm_{\xi,\mu}+t , p^\pm_{\xi,\mu}, \widetilde p^\pm_{\xi,\mu}) \,\colon 
\, -4-p^\pm_{\xi,\mu} \leqslant t   \leqslant 4-p^\pm_{\xi,\mu}\,\big\}.
$$
The superposition region of the blender $\Lambda_{\xi,\mu}\eqdef \Lambda_{\xi,\mu,\bar 0}$ consists of (large) discs tangent to the cone field  $\mathcal{C}^{\ru\ru}$
which are at the right of $W^\rs_{\mathrm{loc}}(P^-_{\xi,\mu},G_{\xi,\mu})$ and at the left
of $W^\rs_{\mathrm{loc}}(P^+_{\xi,\mu},G_{\xi,\mu})$. 
These  observations
 imply the following:

\begin{remark}[A disc in the superposition region] 
\label{r.thecurveL}
{\em{
Consider
 the disc
\begin{equation}\label{e.soloL}
L\eqdef
\big\{\,(0, y, 0) : |y| < 4\,\big\}\subset \Delta
\end{equation}
in the superposition region the blender $\Lambda_{\xi, \mu, \bar\eta}$ of $G_{\xi, \mu, \bar\eta}$
for every 
$(\xi, \mu, \bar \eta)\in 
\mathcal{O}_{{\mathrm{BH}}}$. 
Then,  by Lemma~\ref{l.opensuperposition},
for every diffeomorphism $F$ close enough to $G_{\xi, \mu, \bar \eta}$
every disc sufficiently close to $L$ is in
the superposition region of the blender $\Lambda_{F}$.}}
\end{remark}

\begin{lemma}\label{l.lem}
 The unstable manifold 
$W^{\mathrm u}\big(P^+_{\xi,\mu,\bar \eta},G_{\xi,\mu,\bar \eta}\big)$
is unbounded in the $y$- and $z$-directions.
\end{lemma}

\begin{proof}
We will show that $W^{\mathrm u}\big(P^+_{\xi,\mu},G_{\xi,\mu}\big)$
contains the set 
$$
G_{\xi,\mu}(\Pi^+_{\xi,\mu}),
\quad \mbox{where}
\quad
\Pi^+_{\xi,\mu} \eqdef \big\{(\,x,y,z) \colon y \geqslant p^+_{\xi,\mu}\,\big\}\subset \mathbb{R}^3.
$$
The general case follows studying 2-dimensional projections of $W^{\mathrm u}\big(P^+_{\xi,\mu,\bar \eta},G_{\xi,\mu,\bar \eta}\big)$.

Consider  
the plane $\Pi_{\xi,\mu}\eqdef \{\,y=p^+_{\xi,\mu}\,\}\subset \mathbb{R}^3$ 
and note that
$$
G_{\xi,\mu}(\Pi_{\xi,\mu}) = \mathcal{L}_{\xi,\mu}\eqdef
\big\{\,(p^+_{\xi,\mu},p^+_{\xi,\mu},\widetilde{p}^+_{\xi,\mu}+t):t\in \mathbb{R}\,\big\}\subset \Pi_{\xi,\mu}.
$$
It is easy to see that $\mathcal{L}_{\xi,\mu}$ is $G_{\xi.\mu}$-invariant and that
the restriction of $G_{\xi.\mu}$ to $\mathcal{L}_{\xi,\mu}$ is an expanding linear map (with
expansion factor $\xi>1$).
Consider the cone field 
$$
\mathcal{C}^{\mathrm u}(Z) \eqdef
\big\{(\,u, v, w)\in\mathbb{R}^3:
|u|<\tfrac{1}{2}\sqrt{v^2 + w^2}\,\big\}.
$$
By \cite[Claims 3.12 and 3.14]{DiaPer:19b}, this cone field is $DG_{\xi,\mu}$-invariant and uniformly expanding
for every $Z=(x,y,z)$ with $|y|> \sqrt{5}$.
A simple calculation implies that for every $Z\in
G_{\xi,\mu} (\Pi^+_{\xi,\mu})$ the following holds
$$
G_{\xi,\mu} (\Pi^+_{\xi,\mu}) \subset \Pi^+_{\xi,\mu}
\qquad
\mbox{and}
\qquad
T_Z \big(G_{\xi,\mu} (\Pi^+_{\xi,\mu})\big) \subset
\mathcal{C}^{\mathrm u}(Z).
$$ 
These properties and the expanding property of  $\mathcal{C}^{\mathrm u}$ imply the lemma.
\end{proof}

\begin{notation*}
{\em{We will consider blenders in $\mathbb{R}^3$ and in the ambient manifold $M$. We will denote the first ones by $\Lambda$ and the second ones by
$\Upsilon$.}} 
\end{notation*}

\section{The renormalisation scheme}\label{s.ren}

In this section, we outline the renormalisation scheme in 
\cite{DiaPer:19}, see Proposition~\ref{p.diazperez}.
For that we embed  any $f\in  \mathcal{H}^r_{\BH}(M)$  
in a bifurcating eight-parameter family 
$$
\mathbb{R}^8\ni \bar\upsilon \, \to \, f_{\bar\upsilon}\in \mathrm{Diff}^r(M)\quad\mbox{with}\quad  f_\textbf{0}=f
$$
and construct a renormalisation scheme for $f$ consisting of:
\begin{itemize}
\item
a sequence of local charts $\Psi_k$ from $\mathbb{R}^3$ to $U_Q$,
\item
a sequence of reparameterisations  $\mathbb{R} \ni \mu \mapsto \bar\upsilon^\xi_k (\mu)\in \mathbb{R}^8 $ with $\xi>1$ and $\bar\upsilon^\xi_k(\mu) \to \textbf{0}$ 
on compact sets, 
\item
sequences
 $n_k, \, m_k\in \mathbb{N}$ such that
 the ``return maps'' $f^{N_2+m_k+N_1+n_k}_{\bar\upsilon_k^\xi(\mu)}$, defined on a neighbourhood of the heterodimensional tangency 
$\widetilde Y$, satisfies
$$
\Psi_k^{-1}  \circ f^{N_2+m_k+N_1+n_k}_{\bar\upsilon_k^\xi(\mu)} \circ \Psi_k \to 
E_{\xi,\mu,\bar \varsigma}
$$
where the convergence is $C^r$,
$E_{\xi,\mu,\bar \varsigma}$ is defined
\eqref{e.E-henonlike},
and 
$\bar \varsigma=\bar \varsigma(\xi,f)$ satisfies  \eqref{e.2.24}.
\end{itemize}

This section is organised as follows. 
We define the unfolding family $ 
(f_{\bar\upsilon})_{\bar\upsilon\in \mathbb{R}^8}$ in Section~\ref{ss.embeddingfamily}
and review the renormalisation scheme   in Section~\ref{ss.reno}. The convergence of the scheme is stated in
Section~\ref{ss.convergence}. 

\subsection{The unfolding family} \label{ss.embeddingfamily}
Given $f\in  \mathcal{H}^r_{\BH}(M)$
and
$$
\bar\upsilon=(\bar{\mu},\bar{\nu},\alpha,\beta)\in \mathbb{R}^3\times \mathbb{R}^3\times \mathbb{R}\times \mathbb{R}= \mathbb{R}^8, 
$$
we consider a (smooth) family 
$(f_{\bar\upsilon})_{\bar\upsilon\in\mathbb{R}^8}$ in $\mathrm{Diff}^r(M)$ 
with $f_{\textbf{0}}=f$, $f_{\bar\upsilon}(Q)=Q$, and 
$f_{\bar\upsilon}(P)=P$  
such that 
\begin{itemize}
\item the parameter $\bar\mu$ unfolds the heterodimensional tangency $\widetilde Y$,
\item  the parameter $\bar\nu$ unfolds the quasi-transverse intersection $\widetilde X$, and
 \item the parameters $\alpha$ and $\beta$ modify the arguments 
of $Df(P)$ and $Df (Q)$.
\end{itemize}

Recall the translation-like
perturbations $T_{W,\bar \omega, \varrho}$ 
and rotation-like perturbations $R^{\ast}_{\alpha, \theta, \kappa}$
in  \eqref{e.translation-perturbation} and \eqref{e.rotationfamily}, respectively.
For $\bar\upsilon= (\bar \mu,\bar \nu,{\alpha},\beta)$ and small $\rho>0$ 
consider the perturbation of the identity defined by
\[
\Omega_{\bar\upsilon,\rho}(Z)=
\Omega_{(\bar \mu,\bar \nu,{\alpha},\beta),\rho}(Z)=
\left\{
\begin{split}
&R^x_{\alpha, \theta_1, \kappa_1}(Z), \quad \mbox{if $Z\in U_P$}
,\\
&
T_{\widetilde{X},\bar \nu, \rho}(Z), \quad \mbox{if $Z\in V_{\widetilde{X}}$}
,\\
& R^y_{\beta, \theta_2, \kappa_2}(Z), \quad \mbox{if $Z\in U_Q$},\\
&T_{\widetilde{Y},\bar \mu, \rho}(Z), \quad \mbox{if $Z\in V_{\widetilde{Y}}$}
,\\
& \mathrm{id}(Z), \quad \mbox{\,\,\,if $Z\not\in V_P\cup V_Q\cup V_{\widetilde X} \cup 
V_{\widetilde Y}$},
\end{split}
\right.
\]
where $V_P,V_Q,V_{\widetilde X}$, and $V_{\widetilde Y}$ are small  neighbourhoods of 
$P,Q, \widetilde{X},$ and $\widetilde{Y}$ contained in 
$U_P,U_Q,U_{\widetilde X}$, and $U_{\widetilde Y}$, respectively, recall Section~\ref{s:bif}.
The numbers $\theta_1,\,\theta_2>1$, $\kappa_1,\,\kappa_2>0$ are chosen such that
\[
\begin{split}
&\widetilde{X}, Y \in [-\kappa_1^{-1}, \kappa_1^{-1}]^3,
\qquad 
\widetilde{Y}, X \in [-\kappa_2^{-1}, \kappa_2^{-1}]^3,\\
& [-\theta_1 \kappa_1^{-1}, \theta_1 \kappa_1^{-1}]^3 \subset U_P,
\qquad
 [-\theta_2\kappa_2^{-1}, \theta_2 \kappa_2^{-1}]^3 \subset U_Q.
\end{split}
\]
Finally, we let
\begin{equation}
 \label{e.thefamily}
f_{\bar\upsilon,\rho} = \Omega_{\bar \upsilon,\rho}\circ f.
\end{equation}

\begin{remark}[The parameter $\rho$] 
\label{r.romark}
{\em{Above we emphasise the role of the parameter $\rho$ related to the size of the
translation-like perturbation of  $f$.
}}
\end{remark}

\begin{remark}[Support of the rotation-like part of $\Omega_{\bar\upsilon,\rho}$]
\label{r.supUpsilonrot}
{\em{
Consider the  linear maps
\[
\begin{split}
&f_{\lambda, \sigma, \varphi} =\begin{pmatrix} \lambda & 0&  \\
0  & \sigma \sin 2\pi \varphi & \sigma \cos 2 \pi \varphi \\
0  & -\sigma \cos 2\pi \varphi & \sigma \sin 2 \pi \varphi
\end{pmatrix},\\
&
\widetilde 
f_{\lambda, \sigma, \varphi} =\begin{pmatrix} 
\sigma \sin 2\pi \varphi 
& 0&    \sigma \cos 2 \pi \varphi\\
0  & \lambda & 0 \\
 -\sigma \cos 2\pi \varphi & 0&  \sigma \sin 2 \pi \varphi
\end{pmatrix}.
\end{split}
\]
With this notation,
 the restriction of  
$f$ to  $U_P$ 
is the map $f_{\lambda_P, \sigma_P, \varphi_P}$, 
recall Section~\ref{ss.linear}.
 Note that if $Z\in f^{-1}\big([-\kappa_1,\kappa_1]^3\big)\cap[-\kappa_1,\kappa_1]^3$ then 
 $$
  \Omega_{\bar \upsilon, \rho} \circ f (Z)= R^x_{\alpha, \theta_1, \kappa_1}\circ f (Z) =
  R^x_{\alpha, \theta_1, \kappa_1} \circ
 f_{\lambda_P, \sigma_P, \varphi_P}=
 f_{\lambda_P, \sigma_P,\varphi_P+\alpha}(Z).
 $$
If $Z\in f^{-1}\big([-\theta_1 \kappa_1^{-1},\theta_1 \kappa_1^{-1}]^3\big)^c\cap U_P$ then
   $$
    \Omega_{\bar \upsilon, \rho} \circ f (Z)= R^x_{\alpha,\theta_1, \kappa_1}\circ f (Z)=f (Z).
   $$
  Similarly,
   the restriction of  
$f$ to  $U_Q$ 
is  $\widetilde f_{\lambda_Q, \sigma_Q, \varphi_Q}$
and analogous conditions  hold.}}
 \end{remark}

\begin{remark}[Support of the translation-like part of $\Omega_{\bar\upsilon, \rho}$]
\label{r.supUpsilontrans}
{\em{Note that
\begin{itemize}
\item
$   \Omega_{\bar \upsilon, \rho} \circ f (Z) =T_{\widetilde{X},\bar \nu, \rho}\circ f (Z)=f(Z)$ 
for every $Z\in M\setminus f^{-1} \big(B (\widetilde X, 2 \rho)\big)$,
\item
$  \Omega_{\bar \upsilon, \rho} \circ f (Z)= T_{\widetilde{X},\bar \nu, \rho}\circ f(f^{-1}(\widetilde{X}))=\widetilde{X}+\bar\nu$.
 \end{itemize}
 Analogously,  we have that
 \begin{itemize}
 \item
 $  \Omega_{\bar \upsilon, \rho} \circ f (Z) = T_{\widetilde{y},\bar \mu, \rho}\circ f (Z)=f(Z)$
 for every $Z\in M\setminus f^{-1} \big(B (\widetilde Y, 2 \rho)\big)$, 
 \item
$  \Omega_{\bar \upsilon,\rho} \circ f (Z)= T_{\widetilde{Y},\bar \mu, \rho}\circ f(f^{-1}(\widetilde{Y}))=\widetilde{Y}+\bar\mu$.
\end{itemize}
}}
\end{remark}

\subsection{The renormalisation scheme}\label{ss.reno} 
We now summarise the ingredients of the renormalisation scheme:
Sojourn times and adjusting arguments (Section~\ref{sss.sojourn}),
reparameterisations (Section~\ref{sss.repar}), 
and changes of coordinates (Section~\ref{sss.coordinate}).

\subsubsection{Sojourn times and adjusting arguments}\label{sss.sojourn}
Fix $\xi>1$ and consider
$$ 
\tau \eqdef \frac{\gamma_3(a_3- a_2)}{\sqrt{2}},
\quad
\mbox{where $\gamma_3$ is as in \eqref{e.Aisoftheform}  and $a_2,a_3$ are
as in \eqref{e.Bisoftheform}}.
$$
Note that $\tau>0$, see~\eqref{e.ctes1sem2}.
 By
 \cite[Lemma 6.1]{DiaPer:19}, associated to $\tau^{-1} \xi$,
 there is a residual subset $\mathcal{R}= \mathcal{R}_{\tau^{-1} \xi}$ of
 $(0,1)\times(1,\infty)$ consisting of pairs $(\sigma,\lambda)$ having 
 a sequence of {\em{sojourn times}}
 $\mathbf{s}_k =(m_k,n_k)$ in $\mathbb{N}^2$, 
 {\em{ adapted to $\tau^{-1}\xi$}} satisfying 
\begin{equation}
\label{e.limitofsojourn}
\lim_{k\to \infty} \sigma^{m_k}\lambda^{n_k}= \tau^{-1}\xi
\end{equation}
where $m_k$ and $n_k$ are related by the inequality 
$$
m_k<\eta\, n_k+\tilde{\eta}+1, \quad 
\eta\eqdef \frac{\log(\lambda^{-1})}{\log(\sigma)}  \quad
\mbox{and} \quad \tilde{\eta}\eqdef \frac{\log(\tau\,\xi^{-1})}{\log(\sigma)}.
$$

Our hypotheses allow us to consider  $(\sigma_P,\lambda_Q) \in \mathcal{R}$
having a sequence of sojourn times $\bfs_k=(m_k,n_k)$ adapted to
$\tau^{-1}\,\xi$, see  \cite[Lemma 5.1]{DiaKirShi:14} and \cite[Lemma 6.1]{DiaPer:19}.
The spectral condition in \eqref{e.espectralconditions} provides a constant $C> 0$ such that
\begin{equation}
\label{e.espectroparazero}
{\lambda_P}^{\frac{m_k}{2}}\sigma_P^{m_k}{{\sigma_Q}}^{n_k} <C\, 
\big(
({\lambda_P}^{\frac{1}{2}}\sigma_P)^{\eta}{{\sigma_Q}}\big)^{n_k}
\to 0 
\end{equation}

Associated to the sequence $(\bfs_k)$ there is a sequence 
$\Theta_{k}=(\zeta_{m_k}, \vartheta_{n_k})\in\mathbb{R}^2$, with  $\Theta_{k}\to (0,0)$,
of {\em{adjusting arguments}} 
leading to 
the following argument maps
\begin{equation}
\begin{split}
\label{e.alphakbetak}
&
\alpha_k (\theta)
\eqdef
\frac1{2\pi m_k}\left(\frac{\pi}{4}-2\pi m_k\theta+
2\pi [m_k \theta] + \zeta_{m_k}
\right);
\\
&
\beta_k(\omega)
\eqdef
\frac1{2\pi n_k}
\left(
\frac{\pi}{2}-
2\pi n_k\omega+
2\pi [n_k\omega]
+ \vartheta_{n_k}
\right),
\end{split}
\end{equation}
where  $[x]$ denotes the integer part
of $x\in \mathbb{R}$.  The sequence $(\vartheta_{n_k})$ is chosen such that  $\varphi_Q +
{\beta}_{k}$ is irrational (here $\varphi_Q$ is the argument in item {\bf{(A)}} in Section~\ref{ss.linear}). 
There are no further restrictions on the definition of
$(\zeta_{m_k})$.

\subsubsection{Reparameterisations}\label{sss.repar}
Associated to the sequences $(\bfs_k)$ and $(\Theta_k)$ 
we define 
the sequence of reparameterisations
$\bar \upsilon= \bar \upsilon^\xi_k$  of the family  $f_{\bar \upsilon,\rho}$ in \eqref{e.thefamily} by:
$$
\bar \upsilon^\xi_k\colon  \mathbb{R}\to 
\mathbb{R}^8,
\quad
\bar \upsilon^\xi_{k}(\mu)\eqdef
\big(\bar{\mu}^\xi_{k}(\mu),\bar\nu^\xi_{k},\alpha^\xi_{k}({\varphi_P}),
\beta^\xi_{k}(\varphi_Q)\big) \in \mathbb{R}^3\times \mathbb{R}^3\times \mathbb{R}\times
\mathbb{R},
$$
where (for simplicity, in what follows we eliminate the dependence\footnote{This dependence is given by the choice in \eqref{e.limitofsojourn}.} of the
coordinate maps of $\bar \upsilon^\xi_{k}$
on $\xi$):
\smallskip

\noindent
$\quad \bullet$
$\bar{\mu}_{k} : \mathbb{R} \rightarrow \mathbb{R}^3$   
is defined by
\begin{equation}
\label{e.mu}
\bar{\mu}_{k}(\mu)\eqdef(-\lambda_P^{m_k} a_1,\sigma_Q^{-n_k}+\sigma_Q^{-2n_k}
\sigma_P^{-2m_k}\mu -\lambda_P^{m_k} b_1, -\lambda_P^{m_k} c_1), 
\end{equation}
where $a_1,b_1,c_1$ are as 
in ~\eqref{e.Bisoftheform}. 
Note that 
$\bar{\mu}_{k}(\mu)\to (0,0,0)$ as $k \to \infty$.

\smallskip

\noindent
$\quad \bullet$
To define
$\bar \nu_{k} \in  \mathbb{R}^3$ consider first
$$
\varphi_{P,k}\eqdef
\varphi_P+{\alpha}_{k}(\varphi_P)\quad\mbox{and}\quad  
\varphi_{Q,k} \eqdef 
\varphi_Q+{\beta}_{k}(\varphi_Q)
$$
and the sequences
\begin{equation}
\label{e.fraksequences}
\begin{split}
\tilde{\mathfrak{c}}_k
&\eqdef\cos\big(
2\pi m_k (\varphi_{P,k})
\big),
\quad
\tilde{\mathfrak{s}}_k
\eqdef\sin\big(
2\pi m_k(\varphi_{P,k})
\big),
\\
\mathfrak{c}_k
&\eqdef\cos
\big(2\pi n_k(
\varphi_{Q,k}
)
\big),
\quad
\mathfrak{s}_k
\eqdef\sin\big(2\pi n_k
(\varphi_{Q,k})\big).
\end{split}
\end{equation}

\begin{remark}\label{r.convergenceofconvergents}{\em{
By the definition of $\alpha_k (\varphi_{P})$ and $\beta_k(\varphi_{Q})$ in \eqref{e.alphakbetak}, it follows that
${\mathfrak{c}}_k\to 0$, ${\mathfrak{s}}_k \to  1$, and
$\tilde{\mathfrak{c}}_k,\tilde{\mathfrak{s}}_k\to 1/\sqrt{2}$.}}
\end{remark}

Recalling the coordinated maps  $\widetilde{H}_2$ and $\widetilde{H}_3$ of $\widetilde{H}$ in \eqref{e.transition1},
we let
\begin{equation}
\label{e.removflatconditions}
\begin{split}
\tilde{\rho}_{2,k}\eqdef \frac1{2}\frac{\partial^2}{\partial x^2}\widetilde{H}_2(\textbf{0})(\mathfrak{c}_k-\mathfrak{s}_k)^2+\frac1{2}\frac{\partial^2}{\partial z^2}\widetilde{H}_2(\textbf{0})(\mathfrak{s}_k+\mathfrak{c}_k)^2,
\\
\tilde{\rho}_{3,k}\eqdef \frac1{2}\frac{\partial^2}{\partial x^2}\widetilde{H}_3(\textbf{0})(\mathfrak{s}_k-\mathfrak{c}_k)^2+\frac1{2}\frac{\partial^2}{\partial z^2}\widetilde{H}_3(\textbf{0})(\mathfrak{s}_k+\mathfrak{c}_k)^2.
\end{split}
\end{equation}

Finally, we let
\begin{equation}
\label{e.nu}
\begin{split}
\bar \nu_{k}\eqdef
\Big(
&-\lambda_Q^{n_k}\big(
\alpha_1\,(\mathfrak{c}_k-\mathfrak{s}_k)
+\alpha_3\,(\mathfrak{s}_k+\mathfrak{c}_k)\big),
\\
&
\sigma_P^{-m_k}(\tilde{\mathfrak{c}}_k+\tilde{\mathfrak{s}}_k)-{\lambda_Q}^{2n_k}
\tilde{\rho}_{2,k},\\
&
\sigma_P^{-m_k}(\tilde{\mathfrak{c}}_k-\tilde{\mathfrak{s}}_k)
-\lambda_Q^{n_k}\,\gamma_3\,(\mathfrak{c}_k+\mathfrak{s}_k)-\lambda_Q^{2n_k}
 \tilde{\rho}_{3,k}\Big),
\end{split}
\end{equation}
where $\alpha_1,\alpha_3,\gamma_3$ are as in 
~\eqref{e.Aisoftheform}.
Note that 
$\bar{\nu}_{k}\to \mathbf{0}$ as $k \to \infty$.
 
 \subsubsection{Change of coordinates}\label{sss.coordinate}
{{Using the local coordinates in $U_Q$, we consider the sequence of maps
$\Psi_k \colon U_k \to U_Q=[-a_Q, a_Q]^3$ defined by
 \begin{equation}\label{e.chart}
 \begin{split}
\Psi_{k}(x,y,z)&\eqdef 
(1+\sigma_P^{-m_k}
\sigma_Q^{-n_k}\,x , \\
& \quad\,\,\,\,\,\,\, \sigma_Q^{-n_k}+\sigma_P^{-2m_k}
{\sigma_Q}^{-2n_k}\,y ,1+ \sigma_P^{-m_k}
{\sigma_Q}^{-n_k}\,z ),
\end{split}
\end{equation}
where $U_k$ is the ``cube"  of $\mathbb{R}^3$ such that $\Psi_k (U_k)=U_Q$.}
Recall that $\widetilde{Y}=(1,0,1)$ and note that for any compact set $K\subset\mathbb{R}^3$ it holds $\Psi_k(K)\to \{\widetilde{Y}\}$ as $k\to\infty$.

 \subsection{Convergence of the renormalisation scheme}\label{ss.convergence}
 Fixed $\xi>1$, small $\rho>0$, and $f\in  \mathcal{H}^r_{\BH}(M)$,
consider the renormalisation scheme above and the sequence of one-parameter family of maps 
\begin{equation}\label{e.rrrr}
\begin{split}
&
\mathbb{R}\ni \mu \, \to \, \mathcal{R}_{\bar \upsilon^\xi_k (\mu),\rho}(f)\in \mathrm{Diff}^r(M)
\\
&\mathcal{R}_{\bar \upsilon^\xi_k (\mu), \rho}(f)
\eqdef
f_{\bar \upsilon^\xi_k(\mu), \rho}^{N_2}\circ 
f_{\bar \upsilon^\xi_k(\mu),\rho}^
{m_k}\circ f_{\bar \upsilon^\xi_k (\mu),\rho}^{N_1}\circ f_{\bar \upsilon^\xi_k (\mu),\rho}^{n_k},
\end{split}
\end{equation}
 called \textit{renormalised sequence of $f$}. Here we are emphasising the roles of $\xi$ and $\rho$.

\begin{prop}[Theorem 1,\cite{DiaPer:19}]
\label{p.diazperez}
 Fix $\xi>1$, small $\rho>0$, and $\mu\in \mathbb{R}$.
 Given any $f\in \mathcal{H}^r_{\BH}(M)$
  the sequence of maps 
  $$
   \Psi_k^{-1}\circ \mathcal{R}_{\bar \upsilon^\xi_k (\mu), \rho}(f) \circ \Psi_k \colon U_k \to \mathbb{R}^3, 
   \quad   k\in \mathbb{N},
  $$ 
  converges, on compact sets of $\mathbb{R}^3$ and in the $C^r$ topology, to the endomorphism $E_{\xi,\mu,\bar\varsigma}$ in~\eqref{e.E-henonlike}, where  $\bar \varsigma = \bar \varsigma (\xi,f)$ is as in \eqref{e.barsigma}.
\end{prop}

\begin{notation}[The parameters $\rho$ and $\xi$] 
\label{n.romark}
{\em{When the role of $\rho$ is not relevant it will omitted,
writing $f_{\bar \upsilon}$
and $\mathcal{R}_{\bar \upsilon^\xi_k (\mu)}$
 instead of    $f_{\bar \upsilon, \rho}$ and $\mathcal{R}_{\bar \upsilon^\xi_k (\mu), \rho}$.
Similarly with $\xi>1$.
}}
\end{notation}

\begin{remark}{\em{
Recall the perturbation $f_\ve$ of $f$ in Proposition~\ref{p.Lnbis}.
The renormalisation scheme  of $f$ associated to 
$X$ and $Y$ can be applied to 
$f_\ve$ and it is preserved. In this way,
we get the one-parameter  family
of diffeomorphisms $f_{\ve,\bar \upsilon_k(\mu), \rho}$.  
}}\end{remark}

\section{Interplay between blenders and heteroclinic points}
\label{s.interplay}
In Section~\ref{ss.bh-renormalisation},  see Proposition~\ref{p.blendersparatodos}, we state the occurrence of blender-horseshoes in the renormalisation scheme
for diffeomorphisms $f \in  \mathcal{H}^r_{\BH}(M)$ and their perturbations $f_\ve \in  \mathcal{H}^r_{\BH}(M)$
 given by Proposition~\ref{p.Lnbis}.

 Note that $f_\ve$ has additional  heteroclinic points
$X_{1,\ve}, X_{2,\ve}$. In Section~\ref{ss.unfoldingnew}, we see how these intersections  are unfolded without modifying the blenders given by the renormalisation scheme.

Before going into the details, recall the definitions of the transitions $\mathfrak{T}_{1,i,\ve}$ associated to $X_{i,\ve}$ and their domains $U_{i,\ve}$ in \eqref{e.mailen} and
consider the neighbourhoods
\[
 \widetilde{U}_{i,\ve}\eqdef \mathfrak{T}_{1,i,\ve} (U_{i,\ve})\quad
 \mbox{of}\quad \widetilde{X}_{i,\ve}= \mathfrak{T}_{1,i,\ve}({X}_{i,\ve}),
 \quad i=0,1,2. 
\]
Take sufficiently small  
$\rho=\rho(\ve) >0$ such that
\begin{equation}\label{e.cornaca}
B(\widetilde{X},2\,\rho)\subset \widetilde{U}_{0,\ve},\quad B(\widetilde{X}_{1,\ve},2\,\rho)\subset \widetilde{U}_{1,\ve},
\quad B(\widetilde{X}_{2,\ve},2\,\rho)\subset \widetilde{U}_{2,\ve}.
\end{equation}
In particular, these three balls are disjoint. We can now consider the renormalisation scheme
 $\mathcal{R}_{\bar \upsilon^\xi_k (\mu),\rho}(f_\ve)$ and observe that by the choice of $\rho$ 
 the renormalisation preserves the heteroclinic points $X_{1,\ve}$ and $X_{2,\ve}$.
 Note that as the transitions of $f$ are not modified, recalling \eqref{e.barsigma}, we have
$$
 \bar \varsigma=\bar \varsigma(\xi,f) 
 = \bar \varsigma(\xi,f_\ve). 
$$ 

\subsection{Blenders in the renormalisation scheme}
\label{ss.bh-renormalisation}

 Recall the definitions of  $\Psi_k$  and of $U_k$ in  \eqref{e.chart} and  of
$\mathcal{R}_{\bar \upsilon^\xi_k (\mu),\rho}(f_\ve)$ in
\eqref{e.rrrr}.
 Define the maps
 $$
 \widehat{\Psi}_k \colon   \mathbb{R} \times U_k \to  \mathbb{R} \times U_Q,
 \qquad
 \mathcal{R}_{\bar \upsilon^\xi_k (\mu),\rho}(f_\ve)
 \colon 
\mathbb{R} \times M \to \mathbb{R} \times M
 $$
 by
\[
\begin{split}
\widehat{\Psi}_k(\mu,X)&\eqdef  \big(\,\mu, \Psi_k (X)\, \big),
\\
\widehat{\mathcal{R}}_{k,\rho}(f_\ve) (\mu,X) &\eqdef \big(\, \mu, 
\mathcal{R}_{\bar \upsilon^\xi_k (\mu),\rho}(f_\ve)
 (X)\, \big)
\end{split}
\]
and consider (with slight abuse of notation on the domain of definitions) the maps  $\Phi_k$ and $\widehat \Phi_k$ defined by
\begin{equation}
\label{e.kcor}
\begin{split}
X\in \mathbb{R}^3\,\,&\mapsto\,\, \Phi_k(X)\eqdef  {\Psi}_k\circ \Theta_{\bar \varsigma} (X)\in U_Q,\\
(\mu,X)\in \mathbb{R}\times \mathbb{R}^3\,\,&\mapsto\,\, \widehat{\Phi}_k(\mu,X)\eqdef  
\widehat{\Psi}_k\circ \widehat\Theta_{\bar \varsigma} (\mu,X) \\
& \hspace{2.27cm}=\big(\varsigma_2^{-1}\mu, \Phi_k (X)\big)\in \mathbb{R}\times U_Q,
\end{split}
\end{equation}
where $\Theta_{\bar \varsigma}$ and $\widehat \Theta_{\bar \varsigma}$  are the conjugations in~\eqref{e.todaslastetas}. 

The following explicit form of the maps $\Phi_k^{-1}$ will be used in Section~\ref{ss.auxiliarydiagonal}.

\begin{remark}\label{r.coortilde}
{\em{Note that
for $(1+x,y,1+z)\in U_Q$ close to $\widetilde{Y}=(1,0,1)$ we have that
 \[
\label{e.coortilde}
 (\tilde{x},\tilde{y},\tilde{z})\eqdef \Phi_k^{-1}(1+x,y,1+z),
\quad 
\left\{
\begin{split}
\tilde{x}&=
\sigma_Q^{n_k}\,\sigma^{m_k}_P\,\varsigma_2\,\varsigma_5^{-1}\,z,\\
\tilde{y}&=\varsigma_2\,\sigma_Q^{2n_k}\,\sigma^{2m_k}_P\,(y-\sigma_Q^{-n_k}),\\
 \tilde{z}&=\sigma_Q^{n_k}\,\sigma^{m_k}_P\,\varsigma_2\,\varsigma_1^{-1}\,x.
 \end{split}
 \right.
  \]
  }}
  \end{remark}

Finally, consider the sequence of maps $\mathfrak{R}_{k,\rho} (f_\ve) \colon   \mathbb{R}\times\mathbb{R}^3 \to \mathbb{R} \times \mathbb{R}^3$
\[
(\mu,X)\,\,\mapsto\,\,
{\mathfrak{R}}_{k,\rho} (f_\ve) (\mu, X)\eqdef \widehat\Phi_k^{-1}\circ  \widehat{\mathcal{R}}_{k,\rho}(f_\ve)\circ {\widehat\Phi_k} (\mu, X).
\]
Finally, note that for each fixed $\mu$ the projection of 
${\mathfrak{R}}_{k,\rho} (f_\ve) (\mu, \cdot)$  in the ``second coordinate'' $\mathbb{R}^3$ is exactly the map $\mathcal{R}_{\bar \upsilon^\xi_k (\mu),\rho}(f_\ve)$.

Recalling the definition of $\bar\eta (\bar\varsigma)$ in \eqref{e.bareta} and of $\bar\varsigma (\xi,f)=\bar\varsigma (\xi,f_\ve)$ in \eqref{e.barsigma}
we define (with slight abuse of notation) the map
\[
 \bar\eta (\xi, f_\ve) \eqdef \bar \eta ( \bar\varsigma (\xi, f_\ve))=\bar\eta (\xi, f).
\]

Recalling Proposition~\ref{p.diazperez}, Theorem~\ref{t.BH-DKS} (and the definition of the set $\mathcal{O}_{\mathrm{BH}}$ there), 
 Remark~\ref{r.conjugadosperonorevueltos}, 
and the  definition of  $\widehat{G}_{\xi,\bar \eta}$ in \eqref{e.E-Ggorro} 
we get the following:

\begin{prop}
\label{p.blendersparatodos} 
Consider $f\in \mathcal{H}^r_{\BH}(M)$, $\xi>1$,
small $\ve>0$, and  $f_\ve$ as in Proposition~\ref{p.Lnbis}.
There is $\rho(\ve)$ such that the sequence  
of maps
$({\mathfrak{R}}_{k, \rho(\ve)}(f_\ve) )$
converges in the $C^r$ topology and 
on compact sets of $\mathbb{R}^4$ to $\widehat{G}_{\xi,\bar \eta (\xi,f)}$. 

As a consequence, 
for every $k$ large enough 
the map $\mathcal{R}_{\bar \upsilon^\xi_k (\mu),\rho(\ve)}(f_\ve)$
has a blender-horseshoe ${\Lambda}_{{\bar \upsilon^\xi_k (\mu),\rho(\ve)}
}$
 in
$\Delta=[-4,4]^2\times [-40,22]$.
\end{prop}

We now describe more precisely the blenders ${\Lambda}_{{\bar \upsilon^\xi_k (\mu),\rho(\ve)}}$.
For  $\mu\in (-10,-9)$, large $k$, and small $\ve$, consider the diffeomorphism
\[
f_{\ve, \bar \upsilon_k(\mu), \rho(\ve)}\eqdef
 \Omega_{\bar \upsilon_k(\mu), \rho(\ve)}\circ f_\ve,
\]
where  $\Omega_{\bar \upsilon_k(\mu), \rho(\ve)}$ is defined as in  \eqref{e.thefamily}.
Proposition~\ref{p.blendersparatodos} implies that $f_{\ve, \bar \upsilon_k(\mu), \rho(\ve)}$ has a blender-horseshoe defined as follows.
 Let
\begin{equation}
 \label{e.DeltaKDelta}
 \begin{split}
  \Delta(\mu) &\eqdef \{\mu\} \times \Delta, \\
 \Delta_{k}(\mu)&\eqdef  \Delta_{k} = \Phi_k (\Delta),\\ 
\widehat \Delta_{k}(\mu)&\eqdef 
\widehat \Phi_{k}\big(\Delta(\mu)\big)
=\big(\varsigma_2^{-1}\,\mu,\Phi_{k}\big(\Delta \big)\big)= \big(\varsigma_2^{-1}\,\mu,\Delta_k\big).
\end{split}
 \end{equation} 

Recalling that
$\mathcal{R}_{\bar \upsilon^\xi_k (\mu),\rho(\ve)}(f_\ve)=
f_{\ve,\bar \upsilon_k(\mu), \rho(\ve)}^{N_2+m_k+N_1+n_k}$,  see~\eqref{e.rrrr},
 we have that
 \begin{equation}
\label{e.ijk-blender}
\begin{split}
\Upsilon_{f_{\ve,\bar \upsilon_k(\mu), \rho(\ve)}}&\eqdef
\bigcap_{\ell \in\mathbb{Z}}\big(
\mathcal{R}_{\bar \upsilon^\xi_k (\mu),\rho(\ve)}(f_\ve)
\big)^\ell (\Delta_k (\mu))\\
&=
\bigcap_{\ell \in\mathbb{Z}}
f_{\ve, \bar \upsilon_k(\mu), \rho(\ve)}^{\,\,(N_2+m_k+N_1+n_k)\,\ell}
(\Delta_{k}).
\end{split}
\end{equation}
is a blender-horseshoe of $f_{\ve, \bar \upsilon_k(\mu),\rho(\ve)}^{N_2+m_k+N_1+n_k}$. 
Note that, by construction,
\begin{equation}\label{e.ijk-Grandeblender}
 {\Lambda}_{\bar \upsilon^\xi_k (\mu),\rho(\ve)} =
\Phi_k^{-1} \big(\Upsilon_{f_{\ve,\bar \upsilon_k (\mu),\rho(\ve)}} \big).
\end{equation}

\begin{notation}
\label{n.referencesaddles}
{\em{The reference saddles  $P^\pm_{\bar \upsilon^\xi_k (\mu),\rho(\ve)}$ of  ${\Lambda}_{\bar \upsilon^\xi_k (\mu),\rho(\ve)}$ are the continuations of the saddles
 $P^\pm_{\xi,\mu,\bar \eta}$ of the blender of  $G_{\xi,\mu, \bar\eta}$ in \eqref{e.referencesaddle}.
 }}
\end{notation}

\begin{remark}\label{r.opa}
{\em{Consider 
$\big(\xi, \mu,   \bar\eta (\xi, f)\big) \in \mathcal{O}_\vpp
$ and write  $\bar \eta = \bar\eta (\xi,f)$. Recall the definition of the disc $L\subset \mathbb{R}^3$ in the superposition region of the blender
of ${G}_{\xi,\mu,\bar\eta}$, see Remark~\ref{r.thecurveL}.
The second part of that remark and
the $C^r$ convergence 
$$
\Phi_{k}^{-1}\circ \mathcal{R}_{\bar \upsilon^\xi_k (\mu),\rho(\ve)}(f_\ve) \circ \Phi_{k}\to
{G}_{\xi,\mu, \bar\eta }
$$
on compact subsets of $\mathbb{R}^3$ 
imply that for every large
$k$ the set
$$
 \Phi_{k}^{-1}\circ \mathcal{R}_{\bar \upsilon^\xi_k (\mu),\rho(\ve)}(f_\ve) \circ \Phi_{k} (L)
$$
contains a disc in the superposition region 
of the blender ${\Lambda}_{\bar \upsilon^\xi_k (\mu),\rho(\ve)}$.}}
\end{remark}

\subsection{Unfolding the heteroclinic points $X_{i,\ve}$}
\label{ss.unfoldingnew}

By construction, we have that  $X_{1,\ve}$ and $X_{2,\ve}$ are  
quasi-transverse heteroclinic points  
of $f_{\ve, \bar \upsilon_k(\mu),\rho(\ve)}$ and (recalling the definitions in
\eqref{e.wastedyears})
$$
\widetilde{X}_{i,\ve}\in 
W^{\mathrm{s}}_{\mathrm{loc}}(P,f_{\ve, \bar \upsilon_k(\mu), \rho(\ve)})\cap  
W^{\mathrm{u}} (Q,f_{\ve, \bar \upsilon_k(\mu), \rho(\ve)}).
$$ 
\begin{remark}\label{r.l.tenemosunblender}{\em{
The choices of $\ve$ and $\rho(\ve)$ imply that the closure of the orbits of $\widetilde{X}_{i,\ve}$ and the orbit of the blender
$\Upsilon_{f_{\ve,\bar \upsilon_k(\mu), \rho(\ve)}}$ are disjoint.  Moreover, the orbit of the blender is also disjoint from
the neighbourhoods  ${U}_{i,\ve}$ of $X_{i,\ve}$ (and thus from the neighbourhoods  $\widetilde{U}_{i,\ve}$ of $\widetilde{X}_{i,\ve}$).}}
\end{remark}

We now consider a
``local unfolding of the heteroclinic point
$\widetilde{X}_{1,\ve}$  independent of the renormalisation process'':  this unfolding is given by a perturbation whose support is disjoint from 
$B\big(\widetilde{X},\rho(\ve)\big)$ and $B\big(\widetilde{Y},\rho(\ve)\big)$.
 For that, consider a family of local perturbations 
  of $f_{\ve}$ given by
 \begin{equation}
 \label{e.gnk}
g_{\ve, \bar {\upsilon}_{k}(\mu), \rho (\ve)}= \theta_{\ve, k} \circ f_{\ve, \bar {\upsilon}_{k}(\mu), \rho (\ve)}, 
 \end{equation}
 where
$\theta_{\ve, k}$ is a $C^r$ perturbation of identity  supported 
on $B\big(\widetilde{X}_{1,\ve},2\rho(\ve)\big)\subset  \widetilde{U}_{1,\ve}$ 
satisfying $\lim_{k\to\infty}d_r (\theta_{\ve,k}, \mathrm{id})=0$. 
To define $\theta_{\ve, k}$, recall the definitions of the bump function $\Pi_{\delta}$ in \eqref{e.bump} and the sequences  
$m_k,\, n_k$ in Section~\ref{sss.sojourn},
$\tilde{\mathfrak{c}}_k$ and $\tilde{\mathfrak{s}}_k$ in
\eqref{e.fraksequences},
 and 
$\bar \nu_k$ in ~\eqref{e.nu}, and
 consider the sequence of vectors
\begin{equation}
\label{e.deftau}
\bar{\tau}_k\eqdef \big(0,\sigma^{-m_k}_P(\tilde{\mathfrak{c}}_k+\tilde{\mathfrak{s}}_k),
\sigma^{-m_k}_P(\tilde{\mathfrak{c}}_k-\tilde{\mathfrak{s}}_k)
\big)\in \mathbb{R}^3, \qquad
 \bar{\tau}_k \to \mathbf{0}.
\end{equation}

The map
$\theta_{\ve,k} \colon M \rightarrow M$ is defined by:
\begin{equation}
\label{e.dtildetheta}
\begin{split}
\theta_{\ve,k} \big(Z\big)
&\eqdef Z+
{\Pi}_{\rho(\ve)}(W)\,\bar{\tau}_k, \quad \mbox{if} \quad
Z=\widetilde{X}_{1,\ve}+W\in 
B\big(\widetilde{X}_{1,\ve},\,2\rho(\ve)\big),\\ 
\theta_{\ve,k}(Z )&\eqdef Z, \quad \mbox{if} \quad 
Z\notin 
B\big(\widetilde{X}_{1,\ve},\,2\rho(\ve)\big). 
\end{split}
\end{equation}

Recalling that   $\Vert \Pi_{\rho(\ve)} \Vert_{r} \leqslant 
(\Vert b \Vert_r)^3 \,\rho(\ve)^{-r}$, see
 \eqref{e.Pinorm}, and that $|\tilde{\mathfrak{c}}_k\pm \tilde{\mathfrak{s}}_k |\leqslant 2$,  see
\eqref{e.fraksequences}, we get
\[
d_r (
{\theta}_{\ve,k}, \mathrm{id} ) \leqslant 2\, (\Vert b \Vert_r)^3 \,\rho(\ve)^{-1}\,
\sigma_P^{-m_k}
\to 0, \qquad k\to\infty.
\]

\begin{remark}\label{r.bychoiceseblenderparatodos}{\em{By definition of $\theta_{\ve,k}$,
for every  $W\not\in f_\ve^{-1} \big(B( \widetilde X_{1,\ve},2 \rho(\ve)\big)$
it holds that
$$
g_{\ve, \bar {\upsilon}_{k}(\mu), \rho (\ve)}(W)= \theta_{\ve,k} \circ f_{\ve, \bar\upsilon_k(\mu),\rho(\ve)}(W) 
= f_{\ve, \bar\upsilon_k(\mu), \rho(\ve)}(W).
$$
As a consequence, if $W\not\in f_\ve^{-1} \big(B (\widetilde X_{1,\ve},2 \delta)\big)$ then
$$
\mathcal{R}_{\bar \upsilon^\xi_k (\mu),\rho(\ve)}(f_\ve) (W) =
\mathcal{R}_{\bar \upsilon^\xi_k (\mu),\rho(\ve)} (g_{\ve, \bar {\upsilon}_{k}(\mu), \rho (\ve)}) (W).
$$
Hence, by Remark~\ref{r.l.tenemosunblender}, the maps 
$g_{\ve, \bar {\upsilon}_{k}(\mu), \rho (\ve)}$ and $f_{\ve, \bar\upsilon_k(\mu), \rho(\ve)}$ have the common  
blender
$\Upsilon_{f_{\ve,\bar \upsilon_k(\mu), \rho(\ve)}}$
 defined in ~\eqref{e.ijk-blender}.

Moreover,  by construction, the curve $L^{\mathrm{u}}_{2,\ve}(\rho(\ve))$ in \eqref{e.1seg1} containing $X_{2,\ve}$
is contained in $W^{\mathrm u}(Q, g_{\ve, \bar {\upsilon}_{k}(\mu), \rho (\ve)})$.
}}
\end{remark}

\section{Orbits and itineraries  associated to the renormalisation} \label{s.orbitsanditineraries}
This is a preparatory section to the proof of Theorem~\ref{t:1}. We study admissible points and returns: we select a set of points whose itineraries for the diffeomorphisms $g_{\ve, \bar \upsilon_k(\mu), \rho (\ve)}$ 
in \eqref{e.gnk}
are associated to 
the renormalisation scheme, see \eqref{e.baltazar}.

Recall the charts 
$\Phi_k =\Phi_{k,\bar \varsigma}=\Psi_k\circ \Theta_{\bar \varsigma} \colon U_k\subset \mathbb{R}^3\to U_Q$
 in~\eqref{e.kcor} and that for every compact set $K\subset \mathbb{R}^3$ it holds $\Phi_k(K)\to \widetilde{Y}$ as $k\to\infty$. 

 Recall  the definitions
 of the neighbourhoods $U_X$,  $U_{\widetilde{X}}$,   $U_{Y}$, and  $U_{\widetilde{Y}}$ of the heteroclinic 
 points of the cycle in Section~\ref{ss.semilocal}, and of the balls $B(\widetilde X, 2\rho(\ve))\subset U_{\widetilde X}$ and
 $B(\widetilde Y, 2\rho(\ve))\subset U_{\widetilde Y}$ in \eqref{e.cornaca}.

  Consider the subset $U_{\ve,\bar {\upsilon}_{k}(\mu), \rho(\ve)}$
of $B(\widetilde Y, 2\rho(\ve))\subset \Phi_k (U_k)=U_Q$ of points having the following itinerary for 
$f_{\ve,\bar {\upsilon}_{k}(\mu), \rho(\ve)}$: a point $w\in U_{\ve, \bar {\upsilon}_{k}(\mu), \rho(\ve)}$ if
$$
w\in f_{\ve,\bar {\upsilon}_{k}(\mu), \rho(\ve)}^
{-(N_2+m_k+N_1+n_k)}\big(  B(\widetilde Y, 2\rho(\ve))\big)\cap   B(\widetilde Y, 2\rho(\ve))
$$
and it satisfies (see Figure~\ref{fig:eshoy}):  
\begin{equation}\label{e.baltazar}
\begin{split}
f_{\ve, \bar {\upsilon}_{k}(\mu), \rho (\ve)}^{i}(w)\, &\in \, U_Q, \quad  \mbox{for every}\quad  0\leqslant i \leqslant n_k,\\
  f_{\ve, \bar {\upsilon}_{k}(\mu), \rho (\ve)}^{n_k} (w)\, &\in \, U_X,\\
f_{\ve, \bar {\upsilon}_{k}(\mu), \rho (\ve)}^{N_1+n_k}
(w) \, &\in \, B(\widetilde X, 2\rho(\ve)),\\
f_{\ve, \bar {\upsilon}_{k}(\mu), \rho (\ve)}^{j+ N_1+n_k}
(w) \, &\in \, U_P, \quad \mbox{for every}\quad 0\leqslant j \leqslant m_k, \\
 f_{\ve, \bar {\upsilon}_{k}(\mu), \rho (\ve)}^
{m_k+N_1+n_k}(w)\, &\in \, U_Y,  \, \mbox{and}
\\
f_{\ve, \bar {\upsilon}_{k}(\mu), \rho (\ve)}^
{N_2+m_k+N_1+n_k}(w) \, &\in \, B(\widetilde Y, 2\rho(\ve)).
\end{split}
\end{equation}
We say that the points in $U_{\ve, \bar {\upsilon}_{k}(\mu), \rho (\ve)}$ are {\em{$({\ve, \bar {\upsilon}_{k}(\mu), \rho (\ve)})$-admissible}} and that
$n_k+N_1+m_k+N_2$ is the
{\em{$({\ve, \bar {\upsilon}_{k}(\mu), \rho (\ve)})$-admissible return.}} 
We now define the
maps
\[
\begin{split}
F_{\ve, \bar {\upsilon}_{k}(\mu), \rho(\ve)} &\colon \Phi_k^{-1} (U_{\ve, \bar {\upsilon}_{k}(\mu),\rho(\ve)}) \to \mathbb{R}^3,\\
F_{\ve, \bar {\upsilon}_{k}(\mu),\rho(\ve)} &\eqdef  
\Phi_k^{-1}\circ
f_{\ve, \bar {\upsilon}_{k}(\mu), \rho (\ve)}^
{N_2+m_k+N_1+n_k}\circ \Phi_k.
\end{split}
\]

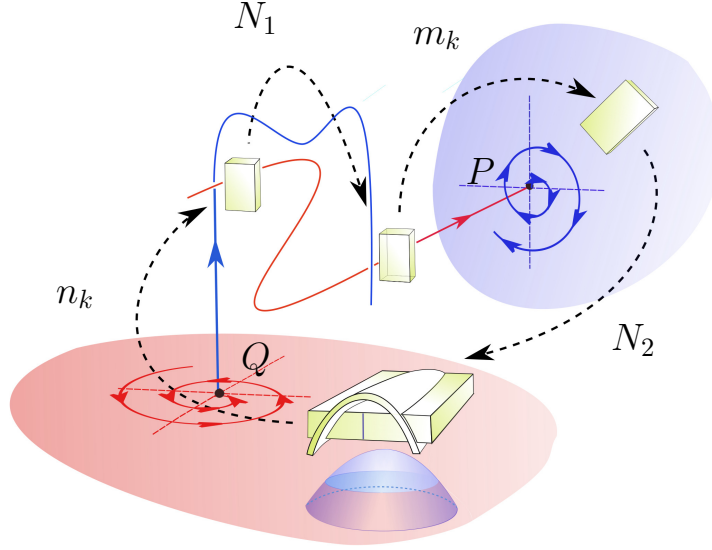
\begin{figure}
\centering
\begin{overpic}[scale=0.13,
]{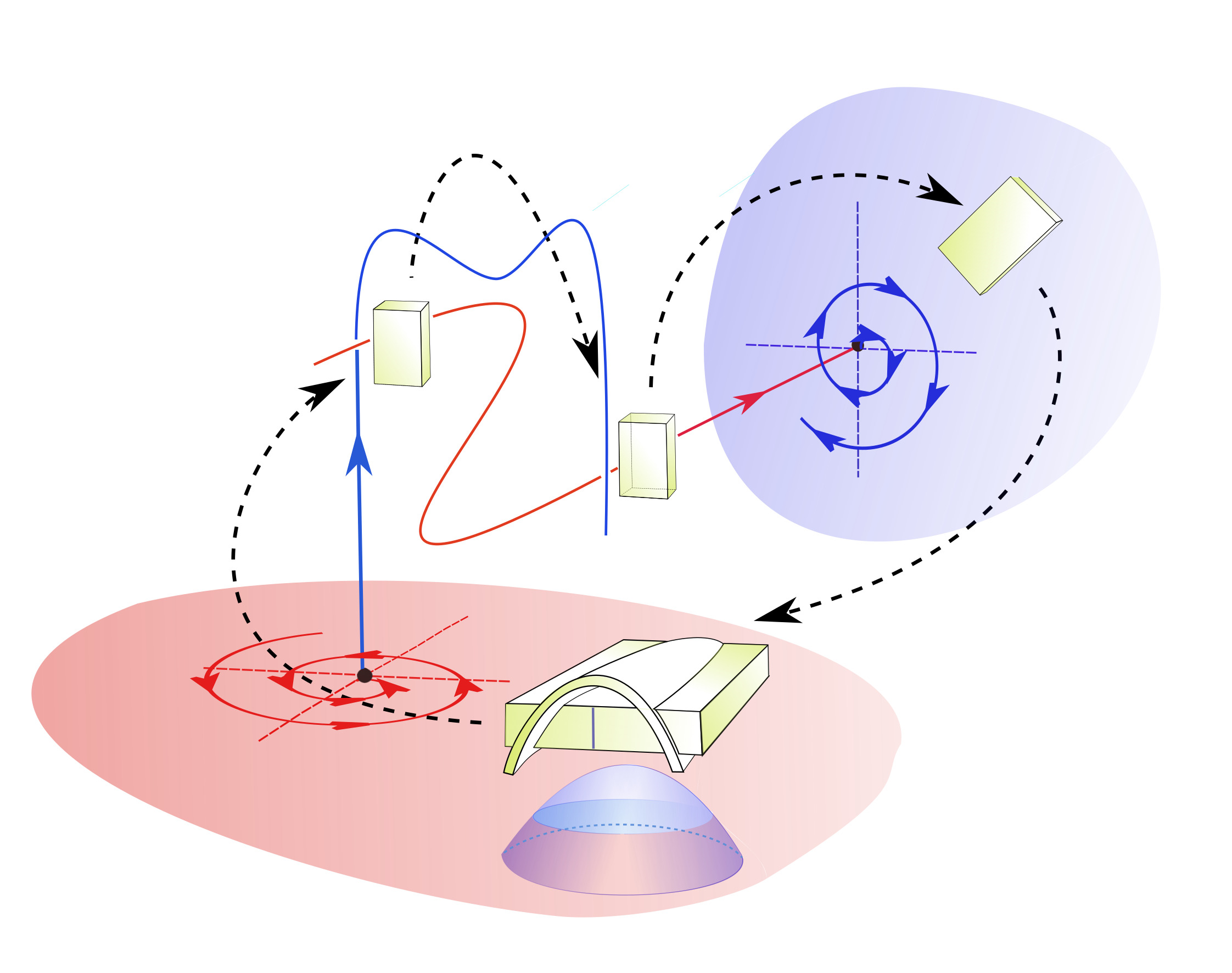}
  \put(95,80){\Large {$Q$}}  
        \put(25,105){\Large $n_k$}     
            \put(92, 209){\Large $N_1$}  
                      \put(160,203){\Large $m_k$}  
\put(235,87){\Large $N_2$}       
     \put(180,150){\Large {$P$}}                
 \end{overpic}
\caption{The points in the blender are admissible points} 
\label{fig:eshoy}
\end{figure}

Recall that $g_{\ve, \bar {\upsilon}_{k}(\mu), \rho (\ve)}= \theta_{\ve, k} \circ f_{\ve, \bar {\upsilon}_{k}(\mu), \rho (\ve)}$ and that for the points in 
$U_{\ve, \bar {\upsilon}_{k}(\mu), \rho (\ve)}$ the map $\theta_{\ve, k}$ is the identity, so $g_{\ve, \bar {\upsilon}_{k}(\mu), \rho (\ve)}= f_{\ve, \bar {\upsilon}_{k}(\mu), \rho (\ve)}$ for points in $U_{\ve, \bar {\upsilon}_{k}(\mu), \rho (\ve)}$. 

\begin{remark}\label{r.admissiblender}{\em{
 Every point of the blender is
$(\ve, \bar {\upsilon}_{k}(\mu), \rho(\ve))$-admissible:
\[
\begin{split}
\Upsilon_{\ve,\bar {\upsilon}_{k}(\mu),\rho(\ve)} & \subset \bigcap_{j\in \mathbb{Z}} f_{\ve,\bar {\upsilon}_{k}(\mu), \rho(\ve)}^
{j \,(N_2+m_k+N_1+n_k)} (U_{\ve, \bar {\upsilon}_{k}(\mu), \rho(\ve)})\\
&= \bigcap_{j\in \mathbb{Z}} g_{\ve,\bar {\upsilon}_{k}(\mu), \rho(\ve)}^
{j \,(N_2+m_k+N_1+n_k)} (U_{\ve, \bar {\upsilon}_{k}(\mu), \rho(\ve)}).
\end{split}
\]}}
\end{remark}

\begin{notation}{\em{In what follows, if there is no possibility of  misunderstanding, we will simply write:
\begin{itemize}
\item
$f_{\ve,k,\mu}$ and $g_{\ve,k, \mu}$ in the places of
$f_{\ve, \bar {\upsilon}_{k}(\mu), \rho (\ve)}$ and $g_{\ve, \bar {\upsilon}_{k}(\mu), \rho (\ve)}$,
\item 
$\Lambda_{\ve, k ,\mu}$ and $\Upsilon_{\ve, k ,\mu}$ 
 in the place of  ${\Lambda}_{\bar \upsilon_k (\mu),\rho(\ve)}$ and $\Upsilon_{f_{\ve,\bar \upsilon_k (\mu),\rho(\ve)}}$,
\item
$ \mathcal{R}_{\ve, k,\mu}$ in the place of $\mathcal{R}_{\bar \upsilon_k (\mu),\rho(\ve)}$.
\end{itemize}
 }}
\end{notation}

\section{Proof of Theorem~\ref{t:1}: Intersections between the two-dimensional manifolds}
\label{s.proofof2}

Throughout this and the next section, we will assume that
$f\in \mathcal{H}^r_{\BH}(M)$, $r \geqslant 2$, and consider the perturbations
$g_{\ve,k,\mu}$  of $f$. By Proposition~\ref{p.blendersparatodos} the set $\Upsilon_{\ve,k,\mu}$ in Remark~\ref{r.admissiblender}
 is a blender of 
$g_{\ve,k,\mu}$. 

We prove that the unstable manifolds of the
blenders  $\Upsilon_{\ve,k,\mu}$ of 
$g_{\ve,k,\mu}$ transversely intersects the stable manifold of the saddle $Q$.

\begin{prop}\label{p.why}
For every small $\ve>0$,   large $k$, and $\mu \in (-10,-9)$ it holds
$$
W^\mathrm{u}
\big(
{\Upsilon}_{g_{\ve,k,\mu}},g_{\ve,k,\mu}
\big) \pitchfork
 W^\mathrm{s}\big(Q_{g_{\ve,k,\mu}},g_{\ve,k,\mu}\big)\neq\emptyset.
$$
\end{prop}

The proof  of this proposition is inspired by  \cite[Proposition 1, Chapter 6.4]{PalTak:93}. Here there are additional difficulties
due to the
 heterodimensional nature of the bifurcation.
 A comparison between the two settings is done in Section~\ref{ss.comparison}.
   
\begin{notation}
{\em{Recall definitions of $f_\ast$, 
 $F_\ast$, $\Upsilon_\ast$, and $U_\ast$, with
$\ast = (\ve,\bar \upsilon_k(\mu),\rho(\ve))$, 
in Section~\ref{s.orbitsanditineraries}. If there is no possible misunderstanding, 
 in what follows $\ve$ and $\rho(\ve)$ will be omitted,
simply writing $f_{\bar \upsilon_k(\mu)}$,  $F_{\bar \upsilon_k(\mu)}$,
$\Upsilon_{\bar \upsilon_k(\mu)}$, and $U_{\bar \upsilon_k(\mu)}$.

By Proposition~\ref{p.blendersparatodos},
the sequence of maps
$F_{\bar \upsilon_k(\mu)}=F_{\ve,\bar {\upsilon}_{k}(\mu), \rho(\ve)}$ converges (on compacta) to the family of endomorphisms
$G_{\xi, \mu, \bar \eta}$
 for some fixed $\xi$ and  $\bar \eta$. 
 Hence we write $G_{\mu}$ in the place of $G_{\xi, \mu, \bar \eta}$.
}}
\end{notation}

This section is organised as follows. In Section~\ref{ss.auxiliarydiagonal}, we introduce an  auxiliary  one-dimensional foliation  $\mathcal{D}$ in $U_P$ which subfoliates the reference domain of the blenders. 
The main property of this foliation is
 that the strong unstable foliation of the blender approaches to $\mathcal{D}$.
We translate  the foliation $\mathcal{D}$ to $U_{\widetilde Y}$  by $f^{N_2}_{\bar {\upsilon}_{k}(\mu)}$
and thereafter by $\Phi_k^{-1}$ to $\mathbb{R}^3$, obtaining in this way a foliation $\widetilde{\mathcal{D}}_{\bar {\upsilon}_k(\mu)}$. 
In Lemma~\ref{l.hankook}, we see that  the leaves of $\widetilde{\mathcal{D}}_{\bar {\upsilon}_k(\mu)}$ 
 converge to parabolas when $k\to \infty$. 
   In Section~\ref{ss.anglesexpansionmanifold}, using the foliation $\mathcal{D}$, for the admissible points and returns in Section~\ref{s.orbitsanditineraries},
we study the expansion of vectors and how the angles change, see Lemma~\ref{l.maiNiam}. In 
Section~\ref{ss.anglesexpansioninR},
we translate these estimates for the map $F_{\bar {\upsilon}_{k}(\mu)}$. 
In Section~\ref{ss.separatrices}, we study the separatrices of the saddles of  the blenders nearby heterodimensional tangencies.
Finally, in Section~\ref{ss.S2} we conclude the proof of  Proposition~\ref{p.why}.}

\subsection{The auxiliary one dimensional foliation}
\label{ss.auxiliarydiagonal}
We start with some preliminary constructions. We consider first auxiliary foliations $\mathcal{F}^{\ru}_{R, \bar\upsilon_k (\mu)}$,
$R=P,Q$,
defined on the neighbourhoods $U_P$ and $U_Q$ as the natural extensions of the invariant local manifolds of the saddles in the cycle.
The leaf  $\mathcal{F}^{\ru}_{P, \bar\upsilon_k (\mu)} (A)$ of the point $A\in U_P$
of the
foliation  $\mathcal{F}^{\ru}_{P, \bar\upsilon_k (\mu)}$ is the intersection 
of the set $U_P$
and the plane
parallel to the coordinate plane $yz$ containing $A$.
Similarly, the leaf
 $\mathcal{F}^{\rs}_{P, \bar\upsilon_k (\mu)} (A)$
of   $\mathcal{F}^{\rs}_{P, \bar\upsilon_k (\mu)}$ is the intersection of  $U_P$
and the straight line parallel to the axis
 $x$ containing $A$. 
 Thus these leaves are ``parallel''  to $W^\ru_{\mathrm{loc}} (P)$ and 
 $W^\rs_{\mathrm{loc}} (P)$, respectively.
 The leaves  of the foliations $\mathcal{F}^{\ru}_{Q, \bar\upsilon_k (\mu)}$
 and 
$\mathcal{F}^{\rs}_{Q, \bar\upsilon_k (\mu)}$ are defined similarly and  are ``parallel"
to $W^\ru_{\mathrm{loc}} (Q)$ and 
 $W^\rs_{\mathrm{loc}} (Q)$. Note that the foliations $\mathcal{F}^{\ast}_{P,Q, \bar\upsilon_k (\mu)}$, $\ast=\rs, \ru$, do not depend on $\bar\upsilon_k (\mu)$.
 
 For $\ast=\rs, \ru$, we ``transport"  the foliations   $\mathcal{F}^{\ast}_{Q, \bar\upsilon_k (\mu)}$
 from $U_X\subset U_Q$ to
 $U_{\widetilde{X}}\subset U_P$ by the transition $f^{N_1}_{\bar {\upsilon}_k(\mu)}$ and
 continue denoting the resulting foliations by $\mathcal{F}^{\ast}_{Q, \bar\upsilon_k (\mu)}$.
 Similarly, 
 we ``transport"  the foliations   $\mathcal{F}^{\ast}_{P, \bar\upsilon_k (\mu)}$
 from  $U_Y\subset U_P$ to 
 $U_{\widetilde{Y}}\subset U_Q$ by  $f^{N_2}_{\bar {\upsilon}_k(\mu)}$ and
 continue denoting these foliations by $\mathcal{F}^{\ast}_{P, \bar\upsilon_k (\mu)}$.
 Note that these extensions do depend on 
$ \bar\upsilon_k (\mu)$.
 
We now consider an auxiliary one-dimensional  foliation $\mathcal{D}$
in $U_P$. For that consider the family of curves
$$
\ell_{(s,a)} \eqdef \big\{\,(s,a,-a)+(0,t,t)\,:\, t \in \mathbb{R}\,\big\} \cap U_P, \quad s,a\in \mathbb{R}.
$$
and define the {\em{diagonal foliation}}  of $U_P=[-a_P,a_P]^3$ by 
\begin{equation}
\label{e.Ele}
\mathcal{D}\eqdef \big\{\,\ell_{(s,a)}\,\colon\,  a,s \in [-a_P,a_P]\,\big\}.
\end{equation}
Note that $\mathcal{D}$ ``subfoliates" the leaves
of  $\mathcal{F}^{\ru}_{P, \bar\upsilon_k (\mu)}$  in $U_P$, see Figure~\ref{fig:diagonalf}.

Consider the domain 
$\Delta_k (\mu)$   of the blender $\Upsilon_{\bar {\upsilon}_k(\mu)}$ 
 in \eqref{e.DeltaKDelta}.
By  \cite[Lemma 3]{DiaPer:19}, for every large $k$, 
the coordinates $(x,1+y,1+z)	\in U_P$ of the points in
$f^{m_k+N_1 +n_k}_{\bar {\upsilon}_k(\mu)}\big(\Delta_k (\mu)\big)$ are close to $Y=(0,1,1)$. Moreover, they have Landau symbols
$$
x=O({\lambda_P}^{m_k}), \quad
y=O(\sigma_P^{-m_k}\,{\sigma_Q}^{-n_k}), \quad 
{z}=O(\sigma_P^{-m_k}\,{\sigma_Q}^{-n_k}).
$$
These conditions and  equation  \eqref{e.espectroparazero} imply that
\begin{equation}
\label{e.containedcontained}
f^{m_k+N_1 +n_k}_{\bar {\upsilon}_k(\mu)}(\Delta_k (\mu))\subset 
\bigcup_{s\in J_k}  \big\{\,\ell_{(s,a)}
\, : \,  a \in [-a_P,a_P]\,\big\}
\end{equation}
where
\begin{equation}\label{e.intervalJk}
J_k \eqdef  [\,- \sigma_Q^{-2n_k}\,\sigma_P^{-2m_k} a_P\, ,\, \sigma_Q^{-2n_k}\,\sigma_P^{-2m_k} a_P\,].
\end{equation}

As above, we consider the intersection of the leaves of 
$\mathcal{D}$ with $U_Y$ and ``transport''  them
by $f^{N_2}_{\bar\upsilon_k (\mu)}$, obtaining the following foliation of  $U_{\widetilde Y}$
(see Figure~\ref{fig:diagonalf}):
\begin{equation}
\label{e.curvelsa}
\mathcal{D}_{\bar\upsilon_k (\mu)}
\eqdef \big\{ \, \ell_{(s,a,\bar\upsilon_k (\mu))} \eqdef f^{N_2}_{\bar\upsilon_k (\mu)} (\ell_{(s,a)} \cap U_Y)\, \, \colon \, a,s \in [-a_P,a_P]\, \big\}.
\end{equation}
Similarly, we let
 \begin{equation}
 \label{e.curvasltilde} 
 \widetilde{\mathcal{D}}_{\bar {\upsilon}_k(\mu)} \eqdef \big\{\,
 \widetilde \ell_{(s,a,\bar\upsilon_k (\mu))}\eqdef \Phi_k^{-1} (\ell_{(s,a,\bar\upsilon_k (\mu))} )  \,\, \colon \, a,s \in [-a_P,a_P] \,\big\}.
\end{equation}

\subsubsection{Comparison of the homoclinic and heteroclinic settings}
\label{ss.comparison}
Our heterodimensional analysis is inspired by the one in \cite[Chapter 6.4]{PalTak:93} for homoclinic tangencies. Let us 
highlight some key  differences and similarities. For that recall that
 \cite{PalTak:93} considers a surface diffeomorphism with saddle $R$ having a homoclinic tangency $Z$. There are associated auxiliary
local stable and unstable foliations $\mathcal{W}^{\ast}_{\mu}$, $\ast=\rs,\ru$,
 defined on a neighbourhood of $Z$, here $\mu$ refers to a ``renormalisation'' parameter unfolding the tangency.
In  \cite{PalTak:93} the renormalisation scheme converges to the quadratic family 
$\varphi_\mu (x,y) = (y, y^2+\mu)$.

The construction in \cite{PalTak:93} implicitly uses the fact that (for suitable parameters) the family $\varphi_\mu$ has a fixed point whose unstable manifold has 
an ``infinite'' separatrix. Here we have a property with the same flavour stated in  Lemma~\ref{l.lem}.

The foliations  $\mathcal{W}^{\rs}_{\mu}$ and  $\mathcal{W}^{\ru}_{\mu}$ converge (in the charts of the corresponding renormalisation scheme) to foliations whose leaves
 are horizontal lines and parabolas of the form $(x, x^2+\mu)$, respectively. Here, the foliation  $\mathcal{D}_{\bar {\upsilon}_k(\mu)}$
 plays the role of the unstable foliation  $\mathcal{W}^{\ru}_{\mu}$.  Lemma~\ref{l.hankook} states a  convergence property of the foliation $\mathcal{D}_{\bar {\upsilon}_k(\mu)}$ 
 (involving some projections). The stable foliation
  $\mathcal{F}^{\rs}_{Q, \bar\upsilon_k (\mu)}$ defined above is similar to 
  $\mathcal{W}^{\rs}_{\mu}$.

Both foliations $\mathcal{W}^{\rs}_{\mu}$ and $\mathcal{W}^{\ru}_{\mu}$ foliate
 a neighbourhood of the tangency $Z$. Here, we have that
 the  foliation $\mathcal{D}_{\bar {\upsilon}_k(\mu)}$  
covers a neighbourhood of the heterodimensional tangency $\widetilde Y$ and therefore
of 
the reference domain  $\Delta_k(\mu)$ of the blender  $\Upsilon_{\bar {\upsilon}_k(\mu)}$, see
\eqref{e.containedcontained}. There is a similar assertion for 
  $\mathcal{F}^{\rs}_{Q, \bar\upsilon_k (\mu)}$.

In  \cite{PalTak:93}, the unstable leaves of the limit thick horseshoes approach to parabola of the limit unstable foliation
and  its projection along stable leaves ``covers'' several fundamental domains of the local unstable manifold of the saddle $R$.
Here, we have a similar property:
the  leaves of the strong unstable foliation
of the blenders $\Upsilon_{\bar {\upsilon}_k(\mu)}$ 
are close to the leaves of $\mathcal{D}_{\bar {\upsilon}_k(\mu)}$.  Due to the lack of domination of our setting, it is not possible to get 
a similar covering property. Instead, we prove that ``projections'' of the leaves of  $\mathcal{D}_{\bar {\upsilon}_k(\mu)}$ covers a fixed proportion of
a fixed fundamental domain of $W^{\ru}_{\mathrm{loc}} (Q)$. This will be enough to see that the blenders are involved in the robust cycles 
with the saddle $Q$ and 
are homoclinically related to the saddle $P$.

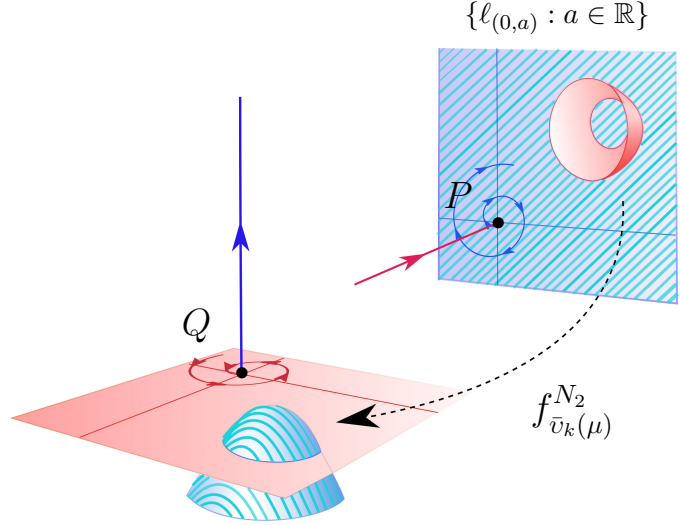
\begin{figure}
\centering
\begin{overpic}[scale=0.1,
]{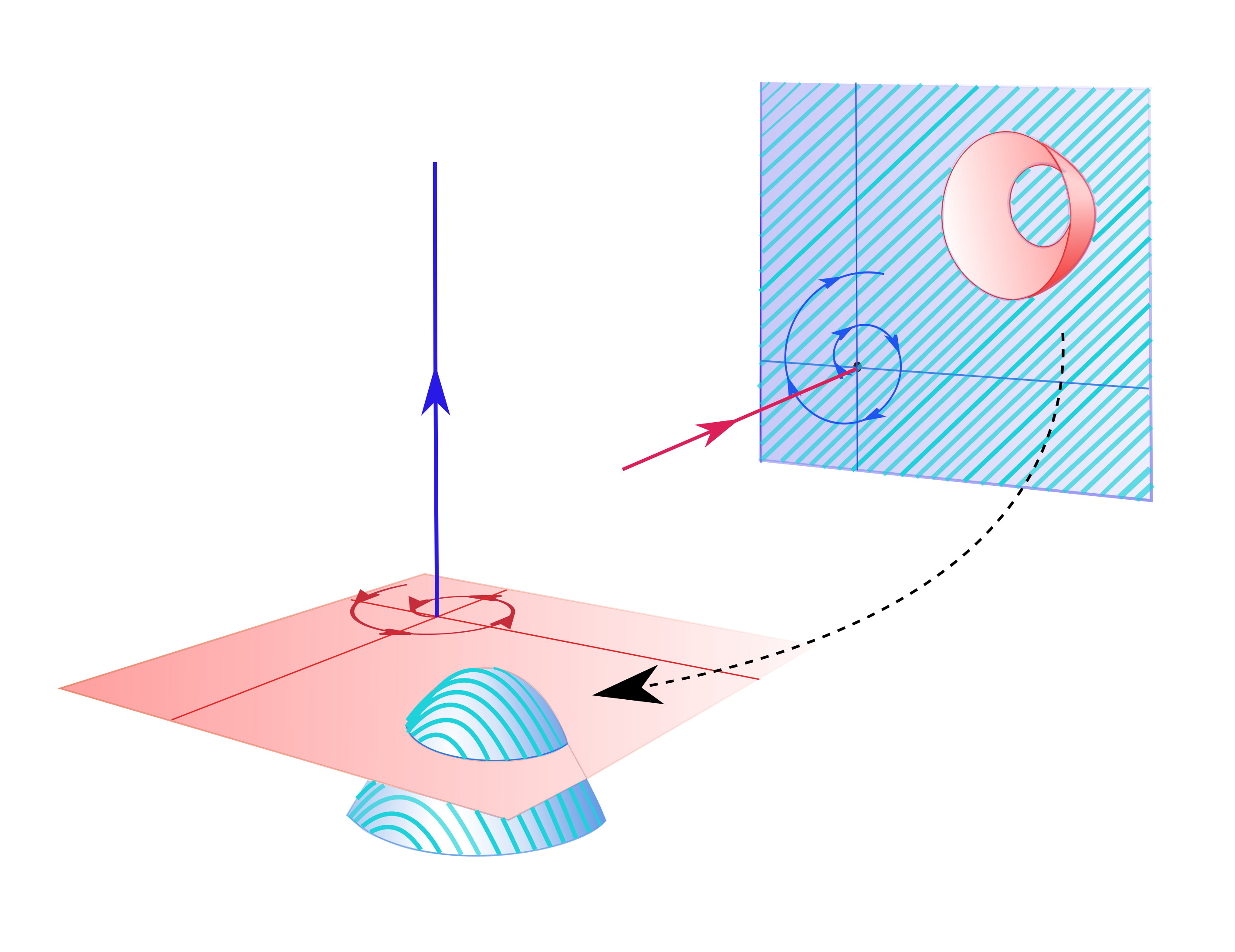}
 \put(78,92){\Large{$Q$}}
     \put(98,75){{$\bullet$}}
  \put(177,140){\Large{$P$}}
     \put(195,132){{$\bullet$}}
        \put(210,60){\Large{$f^{N_2}_{\bar {\upsilon}_k(\mu)}$}}
              \put(185,210){{$\{\ell_{(0,a)}: a\in \mathbb{R}\}$}}
\end{overpic}
\caption{The diagonal foliation on the leaf $\mathcal{F}^{\ru}(P)$.
}\label{fig:diagonalf}
\end{figure}

\subsubsection{Convergence to parabolas}
We now go to the details of our construction. 
Consider:
\begin{itemize}
\item 
the projection
$\pi_{1,2}\colon \mathbb{R}^3\to \mathbb{R}^2$ given by $\pi_{1,2}(x,y,z)=(x,y)$;
\item
the maps $\alpha,\,\beta:\mathbb{R}\to \mathbb{R}$ defined by
\begin{equation}
\label{e.alphabeta}
\begin{split}
\alpha(a)&\eqdef \sqrt{2}\,\beta_2\,(b_2-b_3)\,a,
\\
\beta(\mu,s,a)&\eqdef
\mu 
+
b_1\,\varsigma_2\,s
+(b_2+b_3-b_4)\, \varsigma_2\,a^2, 
\end{split}
\end{equation}
where $b_1,b_2,b_3,b_4$ are parameters associated to the heterodimensional tangency in \eqref{e.Bisoftheform}
and $\varsigma_2$ is  as in the definition of 
the H\'enon-like maps in \eqref{e.E-henonlike};

\item
the family of  curves
\begin{equation}\label{e.pesadez}
\overline 
\ell_{(s,a, \bar\upsilon_k (\mu))} (t)\eqdef
\pi_{1,2}\Big(\widetilde \ell_{(s,a,{\bar{\upsilon}_k(\mu)})}(t)\Big) =
\pi_{1,2} \Big(
\Phi_k^{-1}\big(\ell_{(s,a,{\bar{\upsilon}_k(\mu)})}(t)\big)\Big);
\end{equation}
\item
the re-scaling maps
$\widehat{s}_k, \widehat{a}_k, \widehat{t}_k\colon [-a_P, a_P]\to \mathbb{R}$,
given by
\begin{equation}\label{e.rescaling}
\begin{split}
\widehat{s}_k(s)&\eqdef \sigma_Q^{-2n_k}\,\sigma_P^{-2m_k} s,
\\
\widehat{a}_k(a)&\eqdef   \sigma_Q^{-n_k}\,\sigma_P^{-m_k} a,\\
\widehat{t}_k(t)&\eqdef  \frac{\sqrt{2}}{\beta_2(b_2+b_3+b_4)}\, \sigma_Q^{-2n_k}\,\sigma_P^{-2m_k} t.
\end{split}
\end{equation}
\end{itemize}
Noting that
$\widehat{s}_k(s)\in J_k$ for 
every $s\in [-a_P,a_P]$, we  
can define
the curves
$$
\widehat \ell_{(s,a,{\bar{\upsilon}_k(\mu)})}  (t)\eqdef 
\overline\ell_{(\,\widehat{s}_k(s)\, ,\, \widehat{a}_k(a)\, ,\, \bar{\upsilon}_k(\mu)\,)} (\,\widehat{t}_k (t)\,), \quad 
{(s,a) \in [-a_P, a_P]}
$$
and the sets
$$ 
\widehat L_{(s,a,{\bar{\upsilon}_k(\mu)})}\eqdef  \big\{\,\widehat \ell_{(s,a,{\bar{\upsilon}_k(\mu)})}  (t)\, : \, t\in [-a_P, a_P]\,\big\}.
$$

\begin{lemma}\label{l.hankook} 
The sequence of sets $ \widehat L_{(s,a,{\bar{\upsilon}_k(\mu)})}$
converges 
 to the parabola
  \begin{equation}
\label{e.familiadeparaolas}
 y= x^2+\alpha(a)\, x+\beta(a,s,\mu),
\end{equation}
when $k\to \infty$.
The convergence is $C^r$ uniform  
on compact sets of $\mathbb{R}^2$.
\end{lemma}

Next remark  will be used in Section~\ref{ss.separatrices}. It
relates the  parabolas in Lemma~\ref{l.hankook},
hence the foliation $\widetilde{\mathcal{D}}_{\bar {\upsilon}_k(\mu)}$,
 and
the H\'enon-like maps. It will play a key role in the arguments for controlling the size of unstable sets of blenders.
\begin{remark}
\label{r.bigangles}
{\em{Recall the reference domain $\Delta = [-4,4]^2 \times [-40,22]$ in \eqref{e.deltacube} of the blender of $G_\mu$ and
that
$$
\pi_{1,2}( G_\mu (x,y,z))=
(y,\mu+y^2+\kappa\,y\,z+\eta\,z^2).
$$
Hence 
for each fixed $z_0$
we have that
$\pi_{1,2}( G_\mu (x,y,z_0))$ is a curve of the family in \eqref{e.familiadeparaolas}.  
Therefore  given any  $K>0$ there is $y(K)$ such that for every $z_0\in [-40,22]$ 
and every $y \geqslant y(K)$ 
the angle between the curve 
$\pi_{1,2}( G_\mu (x,y,z_0))$ and the parallel to the $x$-axis at the point $(0,y)$ are  strictly bigger than $2\arctan (K)$, see Figure~\ref{fig:parabolasandangles}.}}
\end{remark}

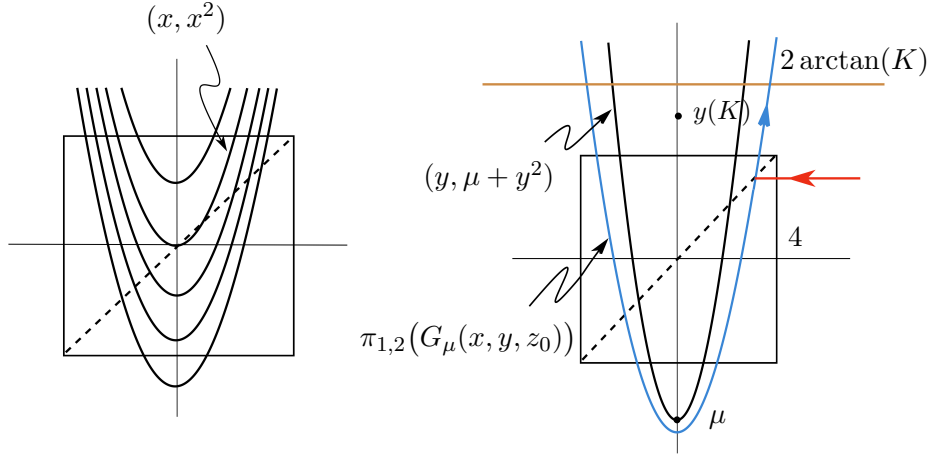
\begin{figure}
\centering
\begin{overpic}[scale=0.032,
]{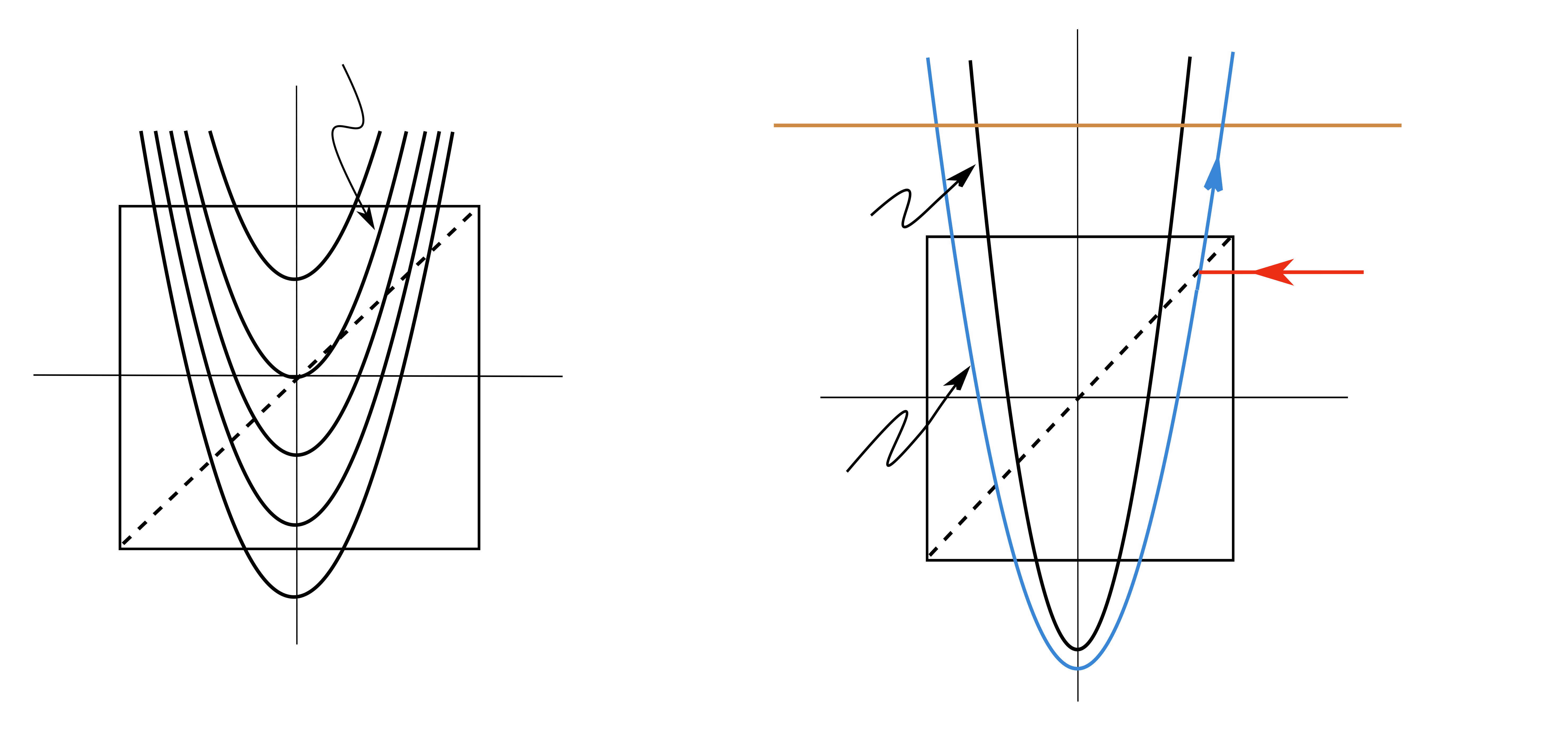}
  \put(60,170){$(x,x^2)$}  
     \put(141,49){$\pi_{1,2}\big(G_\mu(x,y,z_0)\big)$}  
          \put(165,110){$(y,\mu+y^2)$}  
                 \put(273,20){$\mu$}  
            \put(300,154){$2\arctan(K)$}  
                   \put(267,135){\small{$y(K)$}}
 \put(258,130){\Huge{$\cdot$}}  
    \put(257.5,15){\Huge{$\cdot$}}  
        \put(303,87){{$4$}}          
 \end{overpic}
\caption{The limit family of parabolas of  the sets $\widehat{L}_{(s,0,\bar \mu_k(\mu))}$ (left). Angles between parabolas and horizontal lines (right)} 
\label{fig:parabolasandangles}
\end{figure}

\begin{proof}[Proof of Lemma~\ref{l.hankook}]
We first give an explicit calculation of the coordinates of the curves 
$ \ell_{(s,a,\bar\upsilon_k (\mu))}$ in \eqref{e.curvelsa}.
  Consider the parameterisation of $ \ell_{(s,a,\bar\upsilon_k (\mu))} \subset  U_{\widetilde{Y}}$ given by
\[
\begin{split}
\ell_{(s,a,\bar\upsilon_k (\mu))}(t)&\eqdef
f^{N_2}_{\bar{\upsilon}_k(\mu)}\big(s,1+t+a,1+t-a\big)\\
&\eqdef
\Big(\ell^1_{(s,a,{\bar{\upsilon}_k(\mu)})}(t),\ell^2_{(s,a,{\bar{\upsilon}_k(\mu)})}(t),\ell^3_{(s,a,{\bar{\upsilon}_k(\mu)})}(t)\Big).
\end{split}
\]
Recalling the expression 
$f^{N_2}_{\bar{\upsilon}_k(\mu)}$
 in~\eqref{e.transition2}, Remark~\ref{r.supUpsilontrans}, and \eqref{e.mu}, 
we get
\[
\begin{split}
\ell^1_{(s,a,{\bar{\upsilon}_k(\mu)})}(t)=\,\,&1+a_1 s  + 
(a_2-a_3) a +(a_2 + a_3)t\\
&+H_1(s,t+a,t-a)-\lambda^{m_k}_P\,a_1,
\\
\ell^2_{(s,a,{\bar{\upsilon}_k(\mu)})}(t)=\,\,&b_1 s+(b_2+b_3-b_4)\, a^2+(b_2+b_3+b_4)\,t^2 
+ 2(b_2-b_3)\,a\,t
\\
&+ H_2(s,t+a,t-a)
+\sigma_Q^{-n_k}
+
\sigma_Q^{-2n_k}\,\sigma_P^{-2m_k}\,\mu
-\lambda^{m_k}_P\,b_1,
\\
\ell^3_{(s,a,{\bar{\upsilon}_k(\mu)})}(t)=\,\,
&1+c_1 s +(c_2 + c_3)t+H_3(s,t+a,t-a)-\lambda^{m_k}_P\,c_1.
\end{split}
\]
Write
\[
\overline 
\ell_{(s,a, \bar\upsilon_k (\mu))} (t)\eqdef
\Big(\overline 
\ell^1_{(s,a, \bar\upsilon_k (\mu))}
(t),\overline {\ell}^2_{(s,a,{\bar{\upsilon}_k(\mu)})}(t)\Big).
\]
From  the definitions of
$\overline 
\ell_{(s,a, \bar\upsilon_k (\mu))} (t)$ in
\eqref{e.pesadez}
and
of $\Phi_k^{-1}$ in Remark~\ref{r.coortilde},
we get 
\begin{equation*}
\begin{split}
\overline{\ell}^1_{(s,a,{\bar{\upsilon}_k(\mu)})}(t)=\,\,&c_1\,\varsigma_2\,\varsigma^{-1}_5\,\sigma_Q^{n_k}\,\sigma_P^{m_k} s  
+\frac{\beta_2(b_2+b_3+b_4)}{\sqrt{2}}\,\sigma_Q^{n_k}\,\sigma_P^{m_k}\, t
\\
&+
\,\varsigma_2\,\varsigma^{-1}_5\,\sigma_Q^{n_k}\,\sigma_P^{m_k} H_3(s,t+a,t-a)
-\lambda^{m_k}_P\,\sigma_Q^{n_k}\,\sigma_P^{m_k}\,\varsigma_2\,\varsigma^{-1}_5\,c_1,
\\
\overline{\ell}^2_{(s,a,{\bar{\upsilon}_k(\mu)})}(t)=\,\,&
b_1\,\varsigma_2\, 
\sigma_Q^{2n_k}\,\sigma_P^{2m_k}
s
+
\varsigma_2\,(b_2+b_3-b_4)\,\sigma_Q^{2n_k}\,\sigma_P^{2m_k} a^2\\
&+
\frac{\beta^2_2(b_2+b_3+b_4)^2}{2}\sigma_Q^{2n_k}\,\sigma_P^{2m_k}\,t^2 
+ 
2(b_2-b_3)\,\varsigma_2\,\sigma_Q^{2n_k}\,\sigma_P^{2m_k}\,a\,t\\
&
+\varsigma_2\,
\sigma_Q^{2n_k}\,\sigma_P^{2m_k}\,H_2(s,t+a,t-a)
+\mu -\lambda^{m_k}_P\,\sigma_Q^{2n_k}\,\sigma_P^{2m_k}\,\varsigma_2\,b_1.
\end{split}
\end{equation*}

Recalling the re-scaling maps
$\widehat{s}_k, \widehat{a}_k, \widehat{t}_k$,
in
\eqref{e.rescaling}
and performing the corresponding substitutions, a straightforward calculation implies that
\[
\widehat \ell_{(s,a,{\bar{\upsilon}_k(\mu)})} (t)\ \to 
(t, t^2+\alpha ({a}) \,t+\beta (\mu , {s},{a})),
\]
where  $\alpha$ and $\beta$ are as in \eqref{e.alphabeta}. This  ends the proof of the lemma.
\end{proof}

\subsection{Estimates of angles and expansion for admissible iterations}
\label{ss.anglesexpansionmanifold} 
We prove the lemma below, which is
 version of   \cite[Claim 1 in Chapter 6.4]{PalTak:93} in our context. 
 To state this lemma
 consider the projection
$\pi_{2,3}\colon \mathbb{R}^3 \to \mathbb{R}^2$ defined by $\pi_{2,3}(x,y,z)=(y,z)$ and
recall the definitions of the foliation $\mathcal{D}_{\bar {\upsilon}_{k}(\mu)}$ of $U_{\widetilde{Y}}$ in \eqref{e.curvelsa}
and of the set $U_{\bar {\upsilon}_{k}(\mu)}$ of points with  an admissible itinerary as in
\eqref{e.baltazar}.

  \begin{lemma}\label{l.maiNiam} 
  There is a constant $C>0$ such that for every $\mu \in (-10,-9)$ and every sufficiently large $k$ 
 the following holds:
 
Consider $K>0$, a $\bar {\upsilon}_{k}(\mu)$-admissible point  $w\in U_{\bar {\upsilon}_{k}(\mu)}$,  and a vector 
$v
\in T_{w}M
$ 
such that  
\begin{equation}
\label{e.angle}
\mathrm{angle}\big(v,\mathcal{F}^{\rs}_{Q, \bar {\upsilon}_{k}(\mu)} \big)\geqslant K \, \sigma_Q^{-n_k}\sigma_P^{-m_k}.
\end{equation}
Then
\begin{enumerate}
\item
$\Vert  \pi_{2,3} \big( D_w(f_{\bar {\upsilon}_{k}(\mu)}^{N_2+m_k+N_1+n_k}) (v) \big)\Vert \geqslant C\, K\, \Vert v \Vert$,
\item
$\mathrm{angle}\big(D_w(f_{\bar {\upsilon}_{k}(\mu)}^
{N_2+m_k+N_1+n_k})(v) ,\mathcal{F}^{\ru}_{P,\bar {\upsilon}_{k}(\mu)} \big) = O(\lambda_P^{m_k}\sigma_Q^{n_k})\to 0$ as $k\to \infty$, and
\item
$\mathrm{angle} \big(D_w(f_{\bar {\upsilon}_{k}(\mu)}^
{N_2+m_k+N_1+n_k})(v),\mathcal{D}_{\bar {\upsilon}_{k}(\mu)} \big)\to 0, \quad
\mbox{as $k\to\infty$}.
$
\end{enumerate}
 \end{lemma}
 
We begin with some preliminary estimates.
For $R=P,Q$, consider the  coordinate vector fields   
 $
\Big\{\frac{\partial}{\partial\,x_R},\quad 
\frac{\partial}{\partial\,y_R},\quad 
\frac{\partial}{\partial\,z_R}\Big\},
$
defined on the neighbourhood $U_R$ and tangent to the corresponding foliations 
$\mathcal{F}^{\ast}_{R, \bar\upsilon_k (\mu)}$, $\ast=\ru,\rs$. The derivatives of the transition maps
 $\mathfrak{T}_1= f^{N_1}$ 
in  \eqref{e.Aisoftheform} and 
 $\mathfrak{T}_2= f^{N_2}$ 
in
 \eqref{e.Bisoftheform} satisfy
 \begin{equation}
\label{e.DTrans12}
\begin{split}
D \mathfrak{T}_1 \Big(\frac{\partial}{\partial\,x_Q}\Big)
&=
\alpha_1 \frac{ \partial}{\partial\,x_P},\\
D \mathfrak{T}_1  \Big(\frac{\partial}{\partial\,y_Q}\Big)&=
\alpha_2\frac{\partial}{\partial\,x_P}+
\beta_2\frac{\partial}{\partial\,y_P},
\\
D \mathfrak{T}_1  \Big(\frac{\partial}{\partial\,z_Q}\Big)&=
\alpha_3\frac{\partial}{\partial\,x_P}+
\gamma_3\frac{\partial}{\partial\,z_P},
\\
D \mathfrak{T}_2 \Big(\frac{\partial}{\partial\,x_P}\Big)
&=
a_1\frac{\partial}{\partial\,x_Q}+
b_1\frac{\partial}{\partial\,y_Q}+
c_1\frac{\partial}{\partial\,z_Q},
\\
D \mathfrak{T}_2 \Big(\frac{\partial}{\partial\,y_P}\Big)&=
a_2\frac{\partial}{\partial\,x_Q}+
c_2\frac{\partial}{\partial\,z_Q},
\\
D \mathfrak{T}_2 \Big(\frac{\partial}{\partial\,z_P}\Big)&=
a_3\frac{\partial}{\partial\,x_Q}+
c_3\frac{\partial}{\partial\,z_Q}.
\end{split}
\end{equation}

Given a point  $w\in U_{\bar {\upsilon}_{k}(\mu)}$ and
a
vector
\begin{equation}
\label{e.vasin}
v=v_x\,\frac{\partial}{\partial\,x_Q}+
v_y\,\frac{\partial}{\partial\,y_Q}+
v_z\,\frac{\partial}{\partial\,z_Q},
\end{equation}
we give the explicit expression of $D_w(f_{\bar {\upsilon}_{k}(\mu)}^{N_2+m_k+N_1+n_k})(v)$. It follows the sketch of the step by step calculations of this derivative,  for details  see~\cite[Section 7]{DiaPer:19}.

Recalling the linearising coordinates of $f$ at $Q$ in \eqref{e.linear-local} and
the definitions of  $\mathfrak{c}_k, \mathfrak{s}_k$ in
\eqref{e.fraksequences}, we have
$$
D_w(f_{\bar {\upsilon}_{k}(\mu)}^{n_k})(v)=\lambda_Q^{n_k}(\mathfrak{c}_k-\mathfrak{s}_k)\, v_x\, \frac{\partial}{\partial x_Q}
+
\sigma_Q^{n_k}\, v_y \, \frac{\partial}{\partial y_Q}
+
\lambda_Q^{n_k}(\mathfrak{c}_k+\mathfrak{s}_k)\, v_z\, \frac{\partial}{\partial z_Q}.
$$
Write
$$
D_w(f_{\bar {\upsilon}_{k}(\mu)}^{N_1+n_k})(v)  =
v_{x, {\bar {\upsilon}_{k}(\mu)}} \, \frac{\partial}{\partial x_P}
+v_{x,{\bar {\upsilon}_{k}(\mu)}}\, \frac{\partial}{\partial y_P}
+v_{z,{\bar {\upsilon}_{k}(\mu)}} \, \frac{\partial}{\partial z_P}.
$$
Using \eqref{e.DTrans12} we have that
\begin{equation}
\label{e.unaformula}
\begin{split}
v_{x, {\bar {\upsilon}_{k}(\mu)}} &=\lambda_Q^{n_k}(\mathfrak{c}_k-\mathfrak{s}_k)\, v_x\, \alpha_1
+
\sigma_Q^{n_k}\, v_y\, \alpha_2
+
\lambda_Q^{n_k}(\mathfrak{c}_k-\mathfrak{s}_k)\, v_z\,\alpha_3
,\\
v_{y, {\bar {\upsilon}_{k}(\mu)}} &=  \sigma_Q^{n_k}\, v_y \, \beta_ 2,\\
v_{z, {\bar {\upsilon}_{k}(\mu)}} &= \lambda_Q^{n_k}(\mathfrak{c}_k+\mathfrak{s}_k)\, v_z\,\gamma_3.
\end{split}
\end{equation}
Write
$$
D_w(f_{\bar {\upsilon}_{k}(\mu)}^{m_k+N_1+n_k})(v)  =
\widetilde v_{x, {\bar {\upsilon}_{k}(\mu)}} \, \frac{\partial}{\partial x_P}
+\widetilde v_{y,{\bar {\upsilon}_{k}(\mu)}}\, \frac{\partial}{\partial y_P}
+\widetilde v_{z,{\bar {\upsilon}_{k}(\mu)}} \, \frac{\partial}{\partial z_P}.
$$
Using \eqref{e.unaformula}, the linearising coordinates of $f$ at $P$ in \eqref{e.linear-local}, and
the definitions of  $\tilde{\mathfrak{c}}_k, \tilde{\mathfrak{s}}_k$ in
\eqref{e.fraksequences}, we have that 
\begin{equation}
\label{e.componentesvuelta}
\begin{split}
\widetilde v_{x,{\bar {\upsilon}_{k}(\mu)}}
&=\lambda_P^{m_k}
\big(
\lambda_Q^{n_k}(\mathfrak{c}_k-\mathfrak{s}_k)\, v_x\, \alpha_1
+
\sigma_Q^{n_k}\, v_y\, \alpha_2
+
\lambda_Q^{n_k}(\mathfrak{c}_k-\mathfrak{s}_k)\, v_z\,\alpha_3
\big),\\
\widetilde v_{y,{\bar {\upsilon}_{k}(\mu)}}&=
\sigma_P^{m_k}\, \sigma_Q^{n_k}\,
\tilde{\mathfrak{c}}_k\, v_y\,\beta_2  
-
\sigma_P^{m_k}\, \lambda_Q^{n_k}\, 
\tilde{\mathfrak{s}}_k\, 
(\mathfrak{c}_k+\mathfrak{s}_k)\,
v_z\, \gamma_3,\\
\widetilde v_{z,{\bar {\upsilon}_{k}(\mu)}}&=
\sigma_P^{m_k}\, \sigma_Q^{n_k}\,
\tilde{\mathfrak{s}}_k\, v_y\,\beta_2
+
\sigma_P^{m_k}\,\lambda_Q^{n_k}\,
\tilde{\mathfrak{c}}_k\, 
(\mathfrak{c}_k+\mathfrak{s}_k)\,
v_z\, \gamma_3.
\end{split}
\end{equation}
Finally, we write
$$
D_w(f_{\bar {\upsilon}_{k}(\mu)}^{N_2+m_k+N_1+n_k})(v)  =
\widehat v_{x, {\bar {\upsilon}_{k}(\mu)}} \, \frac{\partial}{\partial x_Q}
+\widehat v_{x,{\bar {\upsilon}_{k}(\mu)}}\, \frac{\partial}{\partial y_Q}
+\widehat v_{z,{\bar {\upsilon}_{k}(\mu)}} \, \frac{\partial}{\partial z_Q}.
$$
Using \eqref{e.DTrans12} and \eqref{e.componentesvuelta}
we get
\begin{equation}
\label{e.componentesvuelta1}
\begin{split}
\widehat v_{x,{\bar {\upsilon}_{k}(\mu)}}
&=
a_1\,\widetilde v_{x,{\bar {\upsilon}_{k}(\mu)}}
+
a_2\,\widetilde v_{y,{\bar {\upsilon}_{k}(\mu)}}
+
a_3\,\widetilde v_{z,{\bar {\upsilon}_{k}(\mu)}},
\\
\widehat v_{y,{\bar {\upsilon}_{k}(\mu)}}&=
b_1\, \widetilde v_{x,{\bar {\upsilon}_{k}(\mu)}},
\\
\widehat v_{z,{\bar {\upsilon}_{k}(\mu)}}&=
c_1\,\widetilde v_{x,{\bar {\upsilon}_{k}(\mu)}}
+
c_2\,\widetilde v_{y,{\bar {\upsilon}_{k}(\mu)}}
+
c_3\,\widetilde v_{z,{\bar {\upsilon}_{k}(\mu)}}.
\end{split}
\end{equation}

 \begin{proof}[Proof of Lemma~\ref{l.maiNiam}] 
Take an unitary vector $v$, write it as in \eqref{e.vasin}, and
 let $\alpha_{k,\mu} \eqdef
\mathrm{angle}\big(v,\mathcal{F}^{\rs}_{Q, \bar {\upsilon}_{k}(\mu)} \big)$. Note that by hypothesis \eqref{e.angle}
\begin{equation}
\label{e.coordinatevy}
|v_y|=\sin(\alpha_{k,\mu})\geqslant \sin(K\, \sigma_Q^{-n_k}\, \sigma_P^{-m_k}) \approx K\, \sigma_Q^{-n_k}\sigma_P^{-m_k}.
\end{equation}

To prove item (1), recall 
the definitions of
$\widehat v_{z,{\bar {\upsilon}_{k}(\mu)}}$ in  \eqref{e.componentesvuelta1} and 
 of $\widetilde v_{x,{\bar {\upsilon}_{k}(\mu)}}$, $\widetilde v_{y,{\bar {\upsilon}_{k}(\mu)}}$, and  $\widetilde v_{z,{\bar {\upsilon}_{k}(\mu)}}$ 
in  \eqref{e.componentesvuelta} and note that
$$
\dfrac{\widehat v_{z,{\bar {\upsilon}_{k}(\mu)}}}
{ \sigma_P^{m_k} \, \sigma_{Q}^{n_k} }=
\dfrac{c_1\,\widetilde v_{x,{\bar {\upsilon}_{k}(\mu)}}}
{ \sigma_P^{m_k} \, \sigma_{Q}^{n_k} }+
\dfrac{c_2\,\widetilde v_{y,{\bar {\upsilon}_{k}(\mu)}}}
{ \sigma_P^{m_k} \, \sigma_{Q}^{n_k} }+
\dfrac{c_3\,\tilde v_{z,{\bar {\upsilon}_{k}(\mu)}}}
{ \sigma_P^{m_k} \, \sigma_{Q}^{n_k} }.
$$
Recalling that $\sigma_P,\sigma_Q>1$,
$\sigma_P^{m_k}\,\lambda_Q^{n_k}\to \tau^{-1} \xi$ (see \eqref{e.limitofsojourn}) and 
$\tilde{\mathfrak{c}}_k,\,\tilde{\mathfrak{s}}_k\to 1/\sqrt{2}$, 
$\mathfrak{c}_k\to 0,\,\mathfrak{s}_k\to 1$
(see Remark~\ref{r.convergenceofconvergents}), we get
$$
\dfrac{\,\widetilde v_{x,{\bar {\upsilon}_{k}(\mu)}}}
{ \sigma_P^{m_k} \, \sigma_{Q}^{n_k} }\to 0,
\qquad
\dfrac{\,\widetilde v_{y,{\bar {\upsilon}_{k}(\mu)}}}
{ \sigma_P^{m_k} \, \sigma_{Q}^{n_k} }\to \frac{\beta_2\,v_y}{\sqrt{2}},
\qquad
\dfrac{\,\widetilde v_{z,{\bar {\upsilon}_{k}(\mu)}}
}
{ \sigma_P^{m_k} \, \sigma_{Q}^{n_k}}\to \frac{\beta_2\,v_y}{\sqrt{2}}.
$$
Thus, for every large enough $k$, we get
\begin{equation}
\label{e.otramas}
|\widehat v_{z,{\bar {\upsilon}_{k}(\mu)}}
| \approx \sqrt{2}\,\sigma_P^{m_k} \, \sigma_{Q}^{n_k}\, | c_2+c_3|\,|\beta_2|\,|v_y|.
\end{equation}
Note that  $c_2+c_3\neq 0$, see \eqref{e.ctes1sem2} and
\eqref{e.Bisoftheform}, and $\beta_2 \ne 0$, see \eqref{e.Aisoftheform}.
Finally, from~\eqref{e.coordinatevy} and \eqref{e.otramas} it follows
\begin{equation}
\label{e.item1}
\begin{split}
\Vert \pi_{2,3}(D_w(f_{\bar {\upsilon}_{k}(\mu)}^{N_2+m_k+N_1+n_k})(v))
\Vert 
&\geqslant
|\widehat v_{z,{\bar {\upsilon}_{k}(\mu)}}|
\\
&\approx \sqrt{2}\,\sigma_P^{m_k} \, \sigma_{Q}^{n_k}\, | c_2+c_3|\,|\beta_2|\,|v_y|
\\
&\geqslant C \, K,
\end{split}
\end{equation}
 where $C=\sqrt{2}\,|\beta_2|\, | c_2+c_3|$. This proves the first item in the lemma.

To prove item (2), 
let $\beta_{k,\mu}$ be the angle between the vector 
$D_w(f_{\bar {\upsilon}_{k}(\mu)}^{m_k+N_1+n_k})(v)$ and the unstable foliation 
$\mathcal{F}^u_{P,\bar {\upsilon}_{k}(\mu)}$ (tangent to $\frac{\partial}{\partial y_P}$ and $\frac{\partial}{\partial z_P}$). Then
$$
\sin(\beta_{k,\mu})=
\frac{
| 
\widetilde v_{x,{\bar {\upsilon}_{k}(\mu)}}
|}
{
\Vert D_w(f_{\bar {\upsilon}_{k}(\mu)}^{m_k+N_1+n_k})(v) \Vert }
\leqslant 
\frac{
| 
\widetilde v_{x,{\bar {\upsilon}_{k}(\mu)}}
|}
{
\Vert \pi_{2,3}(D_w(f_{\bar {\upsilon}_{k}(\mu)}^{m_k+N_1+n_k})(v)) \Vert }
\lesssim
\frac{
 \lambda_P^{m_k} 
\sigma_Q^{n_k}
}
{C \, K},
$$
where in the last inequality we use \eqref{e.item1} and that $\widetilde v_{x,{\bar {\upsilon}_{k}(\mu)}}= {O} ( \lambda_P^{m_k} 
\sigma_Q^{n_k})$, see \eqref{e.componentesvuelta}.
Note that by  \eqref{e.espectroparazero}
$$
\lambda_P^{m_k}
\sigma_Q^{n_k}\leqslant \lambda_P^{m_k}
\sigma_Q^{2n_k}\sigma_P^{2m_k}\to 0.
$$
Therefore,
$$
\beta_{k,\mu} \approx \sin(\beta_{k,\mu})=O(\lambda_P^{m_k} 
\sigma_Q^{n_k})\to 0,
$$ 
proving the second item in the lemma.

To prove the last item in the lemma, recall again the definitions of $\widetilde v_{y,{\bar {\upsilon}_{k}(\mu)}}$ and  $\widetilde v_{z,{\bar {\upsilon}_{k}(\mu)}}$ 
in  \eqref{e.componentesvuelta} and
 note that
\begin{equation}
\label{e.angulosdiagonales}
\begin{split}
\frac{ |\widetilde v_{y, \bar{\upsilon}_{k}(\mu)}|}{|\widetilde v_{z,\bar{\upsilon}_{k}(\mu)}|}
&=\frac{|\sigma_P^{m_k}\, \sigma_Q^{n_k}\,
\tilde{\mathfrak{c}}_k\, v_y\,\beta_2  
-
\sigma_P^{m_k}\, \lambda_Q^{n_k}\, 
\tilde{\mathfrak{s}}_k\, 
(\mathfrak{c}_k+\mathfrak{s}_k)\,
v_z\, \gamma_3|}{|
\sigma_P^{m_k}\, \sigma_Q^{n_k}\,
\tilde{\mathfrak{s}}_k\, v_y\,\beta_2
+
\sigma_P^{m_k}\,\lambda_Q^{n_k}\,
\tilde{\mathfrak{c}}_k\, 
(\mathfrak{c}_k+\mathfrak{s}_k)\,
v_z\, \gamma_3|}
\\
&=
\frac{\Big|
\tilde{\mathfrak{c}}_k\, v_y\,\beta_2  
-
\frac{
\sigma_P^{m_k}\, \lambda_Q^{n_k}\, 
\tilde{\mathfrak{s}}_k\, 
(\mathfrak{c}_k+\mathfrak{s}_k)\,
v_z\, \gamma_3}{\sigma_P^{m_k}\, \sigma_Q^{n_k}}\Big|}
{
\Big|
\tilde{\mathfrak{s}}_k\, v_y\,\beta_2
+
\frac{
\sigma_P^{m_k}\,\lambda_Q^{n_k}\,
\tilde{\mathfrak{c}}_k\, 
(\mathfrak{c}_k+\mathfrak{s}_k)\,
v_z\, \gamma_3}{\sigma_P^{m_k}\, \sigma_Q^{n_k}}\Big|}
\to 1,
\end{split}
\end{equation}
where for the limit we use again that 
$\tilde{\mathfrak{c}}_k,\,\tilde{\mathfrak{s}}_k\to 1/\sqrt{2}$ and 
$\mathfrak{c}_k\to 0,\,\mathfrak{s}_k\to 1$.

 This implies that  the angle between $D_w(f_{\bar {\upsilon}_{k}(\mu)}^{m_k+N_1+n_k})(v)$ and 
 the diagonal foliation $\mathcal{D}$ in 
 \eqref{e.Ele} tends to $0$ as $k$ goes to infinity.  Hence,
 by the definition of $\mathcal{D}_{\bar {\upsilon}_{k}(\mu)}$, 
  the angle between $D_w(f_{\bar {\upsilon}_{k}(\mu)}^{N_2+ m_k+N_1+n_k})(v)$ 
 and
 $\mathcal{D}_{\bar {\upsilon}_{k}(\mu)}$  also tends to $0$ as $k$ goes to infinity. This ends the proof of the lemma.
\end{proof}

\subsection{Separatrices of the saddles of  blenders nearby heterodimensional tangencies}
\label{ss.separatrices}
We now study the stable and strong unstable separatrices of the  reference saddle $P^+_{\bar {\upsilon}_k(\mu)}$ of the blenders
$\Lambda_{\bar {\upsilon}_k(\mu)}$
of $F_{\bar {\upsilon}_k(\mu)}$
 in $\mathbb{R}^3$, recall Notation~\ref{n.referencesaddles}. 
Our goal are the angular estimates in Lemma~\ref{l.pirelli}.

Let us go to the details.
Consider first 
 blenders for the
endomorphisms $G_\mu$.
Take $\mu \in (-10,-9)$ and the blender $\Lambda_{\mu}=\Lambda_{\xi, \mu, \bar \eta}$ of 
$G_{\mu}$ and its reference fixed point $P^+_{\mu}=(p_\mu^+, p_\mu^+, \widetilde p_\mu^+)$ in  \eqref{e.referencesaddle}.

Denote by $\sigma_{\mu}^{\ru\ru}$  the separatrix\footnote{That is, one of the connected components of  
$W^{\ru\ru}(P^+_{\mu}) \setminus \{P^+_{\mu}\}$} of
$W^{\ru\ru}(P^+_{\mu})$
contained in $\{y\geqslant p_\mu^+\}$ and by 
 $\sigma_{\mu}^{\rs}$  the separatrix of
$W^{\rs}(P^+_{\mu})$
contained in $\{x\geqslant p_\mu^+, \, y= p_\mu^+, \, z= \widetilde p_\mu^+ \}$.
 We consider the 
curve $\widetilde \sigma_{\mu}^{\rs}\eqdef \{P^+_{\mu}\} \cup \sigma_{\mu}^{\rs}$. We now introduce the  ingredients of our construction which are depicted in Figure~\ref{fig:separatrices}.

\begin{figure}
\centering
\begin{overpic}[scale=0.08,
]{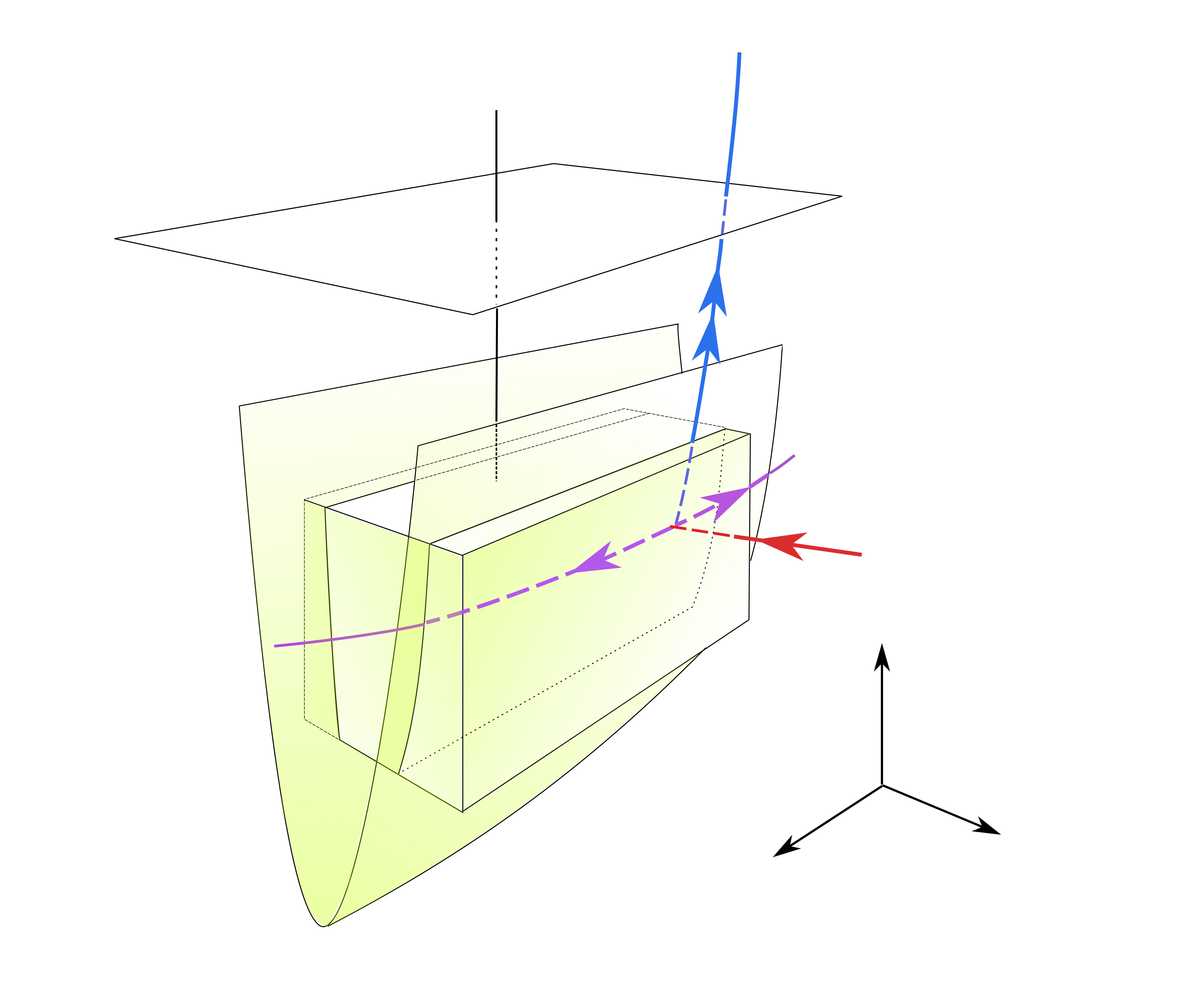}
  \put(156,95){\Large{$P^+_\mu$}}  
        \put(210,117){\Large $\sigma_{\mu}^{\rs\rs}$}     
       \put(162,115){$\bullet$}      
     \put(240,52){\Large $x$}  
            \put(220,83){\Large $y$}  
                     \put(185,47.5){\Large{$z$}}  
          \put(117,188){$\bullet$}  
                  \put(127,190){$y(K)$} 
               \put(173,191){$\bullet$}  
                         \put(185,206){$y^\star_{\mu}(K)$}     
         \put(152,225){\Large $\sigma_{\mu}^{\ru\ru}$}                           
 \end{overpic}
\caption{Separatrices of the  reference saddles of the blender} 
\label{fig:separatrices}
\end{figure}
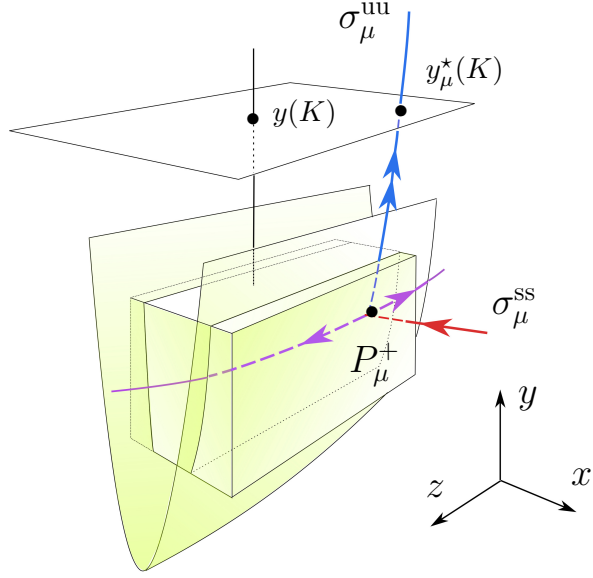

\medskip

\noindent{\em{$\bullet$ Strong unstable separatrices.}}
Given $K>0$ define $y(K)$  as in Remark~\ref{r.bigangles} and consider  the point 
 $$
 y^\star_{\mu} (K) \eqdef \sigma_{\mu}^{\ru\ru}\cap \{ y=y(K)\}.
 $$ 
 Let 
$\sigma_{K, \mu}^{\ru\ru}$  be the segment of  $\sigma_{\mu}^{\ru\ru}$ bounded by  $y^\star_{\mu} (K)$  and  $P^+_{\mu}$
and consider its fundamental domain  
$$
\widetilde \sigma_{K, \mu}^{\ru\ru}\eqdef  \sigma_{K, \mu}^{\ru\ru} \setminus G^{-1}_{\mu}(\sigma_{K, \mu}^{\ru\ru}).
$$
After increasing $K$, if necessary, we can assume that
the angles between the curve 
$\pi_{1,2}( \widetilde \sigma_{K, \mu}^{\ru\ru} )$ and the lines parallel to the $x$-axis are  strictly bigger than $\arctan (K)$.

For  $\mu \in (-10,-9)$ and  large $k$,  we consider the continuations of the objects defined above for $G_\mu$:  
\begin{itemize}
\item
the separatrices
$\sigma^{\ru\ru}_{\bar \upsilon_k (\mu)}$ of $W^{\ru\ru} (P^+_{\bar \upsilon_k (\mu)})$, 
\item 
the points 
 $
 y^\star_{\bar \upsilon_k (\mu)} (K)\eqdef  
 \sigma^{\ru\ru}_{\bar \upsilon_k (\mu)} \cap
 \{ y=y(K)\},
 $
 \item 
 the curves  $\sigma_{K,\bar \upsilon_k (\mu)}^{\ru\ru}$ 
 of  $\sigma_{\bar \upsilon_k (\mu)}^{\ru\ru}$ 
 bounded by  $y^\star_{{\bar \upsilon_k (\mu)}} (K)$  and  $P^+_{{\bar \upsilon_k (\mu)}}$, 
 and 
 \item
 the fundamental domain 
  $$
  \widetilde \sigma_{K, \bar \upsilon_k (\mu)}^{\ru\ru}=
   \sigma_{K, \bar \upsilon_k (\mu)}^{\ru\ru} \setminus  F_{\bar {\upsilon}_{k}(\mu)}^{-1} ( \sigma_{K, \bar \upsilon_k (\mu)}^{\ru\ru}).
   $$ 
  \end{itemize} 
 By Remark~\ref{r.bigangles} and  continuity, for every $\mu\in (-10,-9)$ and large $k$, 
 the angles between the curve 
$\pi_{1,2}( \widetilde \sigma_{K,\bar \upsilon_k (\mu)}^{\ru\ru})$ and the lines parallel to the $x$-axis  are strictly bigger than $\arctan (K)$.  

\medskip

\noindent{\em{$\bullet$ Stable separatrices.}}
 For large $k$,  define the continuations 
 $\widetilde \sigma_{\bar {\upsilon}_{k}(\mu)}^{\rs}$
 of
$\widetilde \sigma_{\mu}^{\rs}$ for 
$F_{\bar {\upsilon}_{k}(\mu)}$ contained in a separatrix of 
$W^{\rs}(P^+_{\bar {\upsilon}_{k}(\mu)})$
and whose boundary contains
$P^+_{\bar {\upsilon}_{k}(\mu)}$.

\medskip

\noindent{\em{$\bullet$ Angles between the separatrices.}}
Note that, by equation~\eqref{e.conodecone}, the tangent space of $W^{\ru\ru} (P^+_\mu)$ at $P^+_\mu$ is contained in a cone field transverse to the horizontal line
(parallel to the $x$-axis).  In particular, equation \eqref{e.conodecone} implies that for every $\mu \in (-10,-9)$
the angle between $\sigma^{\ru\ru}_\mu$ and $\widetilde \sigma^{\rs}_\mu$ at
$P^+_\mu$ is bounded from below by  $\pi/7$. 
Thus, for every $k$ large enough,
the angle between 
 $\sigma_{\bar {\upsilon}_{k}(\mu)}^{\ru\ru}$ and  $ \widetilde \sigma_{\bar {\upsilon}_{k}(\mu)}^{\rs}$ at $P_{\bar {\upsilon}_{k}(\mu)}^+$ is bigger than $\pi/8$.

For the next lemma recall the definition of the curves
$\widetilde \ell_{(s,a,\bar\upsilon_k (\mu))}$ in \eqref{e.curvasltilde} and of the parameter interval
$J_k$ in \eqref{e.intervalJk}.

\begin{lemma}\label{l.pirelli}
For every $K>0$
there is $k_0=k_0(K)\geqslant 1$ such that for every $\mu \in (-10,-9)$ and every $k\geqslant k_0$ 
the angles between
\begin{itemize}
\item 
the lines parallel to the $x$-axis in $\mathbb{R}^2$ and the curves
$\pi_{1,2} (\widetilde \sigma^{\ru\ru}_{K, \bar {\upsilon}_k(\mu)})$
 are at least $\arctan(K)$,
\item
the curves
$\big( \widetilde \ell_{(s,a,\bar\upsilon_k (\mu))}\big)_{a\in [-a_P,a_P], \, s \in J_k}$
 and
$\pi_{1,2} (\widetilde \sigma^{\rs}_{\bar \upsilon_k (\mu)})$ are at least $\frac{\pi}{8}$,
\item 
the lines parallel to the $x$-axis in $\mathbb{R}^2$ and $\pi_{1,2} (
\widetilde \sigma^{\rs}_{\bar \upsilon_k (\mu)})$ are at most $K^{-1}$,
\item 
the curves
$\big( \widetilde \ell_{(s,a,\bar\upsilon_k (\mu))} \big)_{a\in [-a_P,a_P],\, s \in J_k}$ and
the curves $\pi_{1,2} (\widetilde \sigma^{\ru\ru}_{K, \bar {\upsilon}_k(\mu)})$ are at most $K^{-1}$.
\end{itemize}
\end{lemma}

\begin{proof}
The first three items of the lemma follow from the discussion before the lemma.
The last item follows from Lemma~\ref{l.hankook} and  Remark~\ref{r.bigangles}.
\end{proof}

\subsection{Estimates of angles and expansion for iterations of $F_{{\bar {\upsilon}_{k}(\mu)}}$}
\label{ss.anglesexpansioninR}
By the formula of $\Phi_k$ in Remark~\ref{r.coortilde}, we have that $D_w\Phi_k$ is a diagonal matrix that does not depend on the point $w$.
Hence we will omit this dependence.
 Thus, after identifying the tangent spaces in $U_{\widetilde Y}$ 
 with $\mathbb{R}^3$, we have that
\[
\begin{split}
||D_w F_{\bar {\upsilon}_{k}(\mu)}(v)|| &= || D\Phi_k^{-1} 
 \big(D_{\Phi_k(w)} f^{N_2+m_k+N_1+n_k}_{\bar {\upsilon}_{k}(\mu)}(D {\Phi_k)}(v)  \big)||\\
&= || D_{\Phi_k(w)} f^{N_2+m_k+N_1+n_k}_{\bar {\upsilon}_{k}(\mu)}(v) ||.
\end{split}
\]
Observe also  that the angles in Lemma~\ref{l.pirelli}  are taken with respect the coordinates in the $xy$-plane in $\mathbb{R}^3$.
To get these angles in $U_{\widetilde Y}$ we need to replace each angle $\alpha$ by 
$\arctan ( \sigma_P^{-m_k}\sigma_Q^{-n_k} \tan (\alpha))$.  

We have the following consequence of Lemma~\ref{l.maiNiam}. For that recall also the definition of the foliation  
$\widetilde{\mathcal{D}}_{\bar {\upsilon}_{k}(\mu)}$ in \ref{e.curvasltilde}:

  \begin{lemma}\label{l.newmaiNiam} 
  There is a constant $C>0$ such that for every $\mu \in (-10,-9)$ and every sufficiently large $k$ 
 the following holds:
 
Take $K>0$,   a point $w \in \Phi_k^{-1} (U_{{\bar {\upsilon}_{k}(\mu)}} )$,
 and 
a  unitary vector 
$v=(v_1, v_2, v_3)\in \mathbb{R}^3
$ 
such that  
\begin{equation}
\label{e.newangle}
|v_2| \geqslant \arctan (K).
\end{equation}
Then
\begin{enumerate}
\item
$\Vert  \pi_{1,2} \big( D_w F_{\bar {\upsilon}_{k}(\mu)} (v) \big)\Vert \geqslant C\, K\, \Vert v \Vert$,
\item
$\mathrm{angle}\big(D_w F_{\bar {\upsilon}_{k}(\mu)}(v) ,\Phi_k^{-1} (\mathcal{F}^{\ru}_{P,\bar {\upsilon}_{k}(\mu)}) \big) = 
\arctan \big(O(\lambda_P^{m_k}\sigma_Q^{2n_k} \sigma_P^{m_k})\big)\to 0$ as $k\to \infty$, and
\item
$\mathrm{angle} \big(D_w F_{\bar {\upsilon}_{k}(\mu)} (v) , \widetilde{\mathcal{D}}_{\bar {\upsilon}_{k}(\mu)} \big)\to 0, \quad
\mbox{as $k\to\infty$}.
$
\end{enumerate}
 \end{lemma}

\begin{proof}
The first two items of the lemma are direct translations of the corresponding items of Lemma~\ref{l.maiNiam}. For the convergence to zero in the second item we use  \eqref{e.espectroparazero}.
 The third item follows from  \eqref{e.angulosdiagonales} and the definition of $\Phi_k$.
 \end{proof}

\subsection{End of the proof of Proposition~\ref{p.why}}\label{ss.S2}
Using the local coordinates in $U_Q$, fixed small $\delta>0$ consider  the neighbourhood 
$$
U_X (\delta) \eqdef [-\delta, \delta] \times [1-\delta, 1+\delta] \times [-\delta, \delta] 
$$
 of
the heteroclinic point $X=(0,1,0)$ 
and the two-disc (see Figure~\ref{fig:twodimcon})
$$
S(\delta) \eqdef  [-\delta, \delta] \times \{ {1-\delta} \} \times [-\delta, \delta] \subset  \partial U_X (\delta).
$$
For large $k>0$ and $\mu \in (-10,-9)$ define the set
$$
\widehat S_{\bar {\upsilon}_{k}(\mu)} (\delta)\eqdef  C\Big( f^{-n_k}_{\bar {\upsilon}_{k}(\mu)} (0,1-\delta,0),  f^{-n_k}_{\bar {\upsilon}_{k}(\mu)} (S(\delta)) \cap U_Q\Big),
$$
recall that $C(x,A)$ is the connected component of the set $A$ containing the point $x$.
We finally let 
$$
\widetilde S_{\bar {\upsilon}_{k}(\mu)} (\delta) \eqdef \widehat S_{\bar {\upsilon}_{k}(\mu)} (\delta) \cap U_{\widetilde Y}
\quad \mbox{and} \quad 
S_{\bar {\upsilon}_{k}(\mu)} (\delta)\eqdef  \Phi_k^{-1} (\widetilde S_{\bar {\upsilon}_{k}(\mu)} (\delta))\subset \mathbb{R}^3.
$$

\begin{lemma}
\label{l.thefinalintersection}
For every $\delta>0$ 
there is $k_0$ such that for every $k\geqslant k_0$ and every $\mu\in (-10,-9)$ it holds
$$
W^{\ru\ru}(P_{\bar {\upsilon}_{k}(\mu)}^+, F_{\bar {\upsilon}_{k}(\mu)} ) \pitchfork S_{\bar {\upsilon}_{k}(\mu)} (\delta)
\ne \emptyset.
$$
\end{lemma}

\begin{proof}
Let $C>0$ be the constant in Lemma~\ref{l.newmaiNiam}  and take large  $K>0$ with $C\,K= \tau >1$.
Applying Lemma~\ref{l.pirelli} to $K$, we get $k_0$ such that
the angles
between the lines parallel to the $x$-axis in $\mathbb{R}^2$ and $\pi_{1,2} (\widetilde \sigma^{\ru\ru}_{K, \bar {\upsilon}_k(\mu)})$
 are at least $\arctan(K)$. Thus, after increasing $k_0$, we get  the angular condition \eqref{e.newangle} in Lemma~\ref{l.newmaiNiam}.
 Hence  
 $$
 \mathrm{lenght} \big( F_{\bar {\upsilon}_{k}(\mu)} (\widetilde \sigma^{\ru\ru}_{K, \bar {\upsilon}_k(\mu)} )\big)   \geqslant \tau\,  
  \mathrm{lenght} \big(
 \widetilde \sigma^{\ru\ru}_{K, \bar {\upsilon}_k(\mu)} \big).
$$
The curve $F_{\bar {\upsilon}_{k}(\mu)} (\widetilde \sigma^{\ru\ru}_{K, \bar {\upsilon}_k(\mu)})$ is contained in the strong unstable separatrix 
$\sigma_{\bar \upsilon_k (\mu)}^{\ru\ru}$ 
of  $W^{\ru\ru} (P^+_{{\bar \upsilon_k (\mu)}})$. By construction, the curve
$F_{\bar {\upsilon}_{k}(\mu)} (\widetilde \sigma^{\ru\ru}_{K, \bar {\upsilon}_k(\mu)})$
 also satisfies the angular condition  \eqref{e.newangle}.
The proof now follows inductively:  the curves $F^n_k (\widetilde \sigma^{\ru\ru}_{K, \bar {\upsilon}_k(\mu)})$ are
contained in $\sigma_{\bar \upsilon_k (\mu)}^{\ru\ru}$  and their lengths
grow exponentially. This implies that this separatrix transversally intersects the two-disc $S_{\bar {\upsilon}_{k}(\mu)} (\delta)$.
\end{proof}

Note that,   for every sufficiently large $k$, there are defined
the continuations $Z^\pm_{\bar\upsilon_k(\mu)}=Z^\pm_{\ve, \bar\upsilon_k(\mu)}$  of the transverse homoclinic point $Z^\pm_{\ve}$ of $Q$ in
Proposition~\ref{p.Lnbis}. These points are transverse homoclinic points of
$Q$ for $f_ {\bar\upsilon_k(\mu)}$ and (in the coordinates in $U_Q$) are of the form 
$$
Z^\pm_{\bar\upsilon_k(\mu)}=(0, 1\pm \zeta^\pm_{\bar {\upsilon}_k(\mu)}, 0), \qquad
\zeta^\pm_{\bar {\upsilon}_k(\mu)} \in  (0 ,\delta).
$$
For each $k$ and $\mu$, we can also consider a disc $W^\pm_{\bar\upsilon_k(\mu)}$ in $W^\rs(Q, f_{\bar\upsilon_k(\mu)})$
centred at $Z^\pm_{\bar\upsilon_k(\mu)}$ which has uniform size and is uniformly  
transverse to 
$W^\ru_{\mathrm{loc}}(Q, f_{\bar\upsilon_k(\mu)})$, meaning that the angles between 
$W^\pm_{\bar\upsilon_k(\mu)}$ and $W^\ru_{\mathrm{loc}}(Q, f_{\bar\upsilon_k(\mu)})$
at
$Z^\pm_{\bar\upsilon_k(\mu)}$ are uniformly bounded from below.

\begin{figure}
\centering
\begin{overpic}[scale=0.1,
]{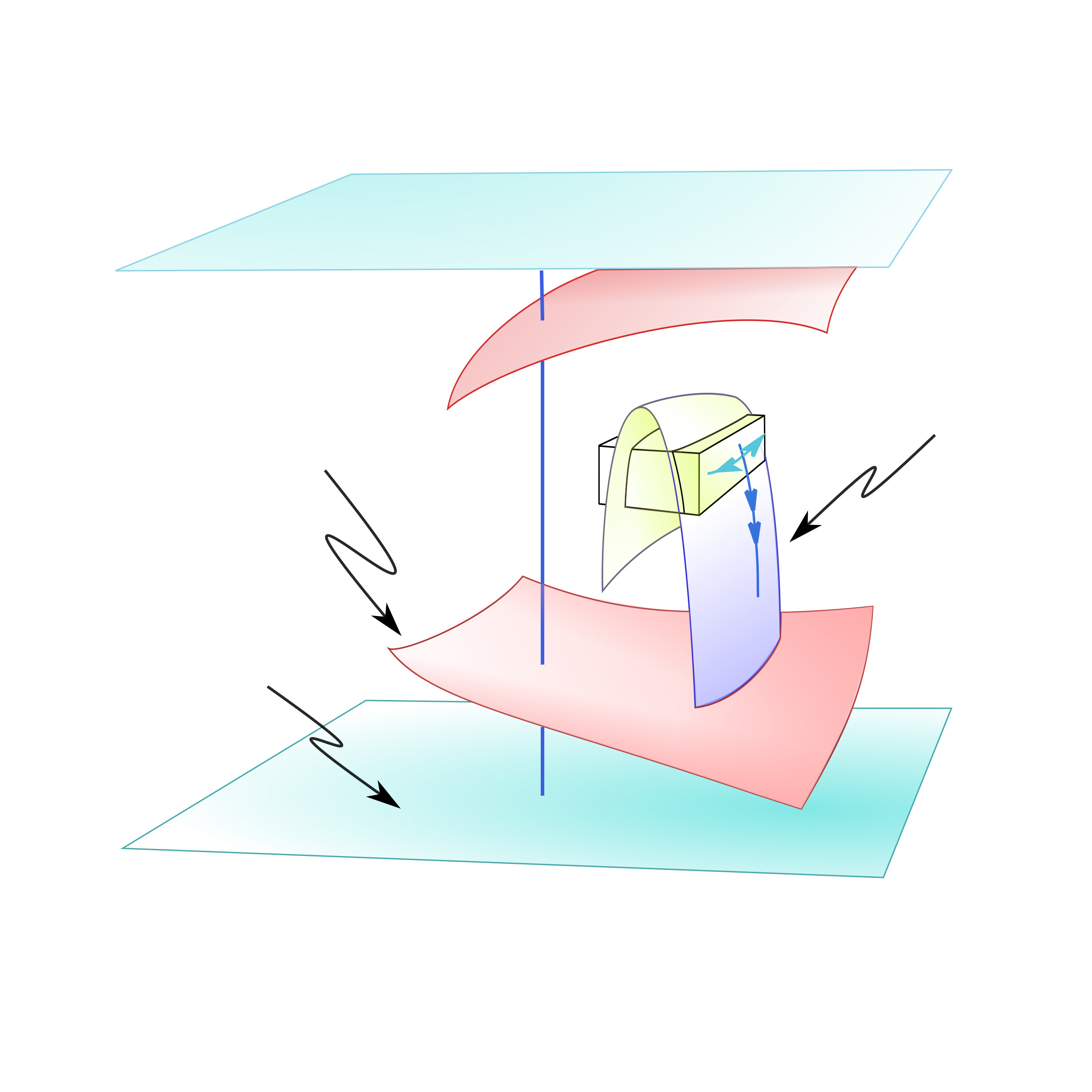}
  \put(210,163){\Large{$ \widetilde \gamma^{\ru\ru}_{\bar {\upsilon}_k(\mu)}$}}  
      \put(42,153){\Large{$W^{\rs}(Q)$}}  
        \put(121.5,95){$\bullet$}  
               \put(125,85){\Large $Z^-_{\bar\upsilon_k(\mu)}$} 
              \put(121,136){$\bullet$}  
                   \put(105,132){\Large $X$}  
             \put(32,103){\Large $S(\delta)$}           
 \end{overpic}
\caption{Two-dimensional connection between the blender and the saddle $Q$.} 
\label{fig:twodimcon}
\end{figure}

We are now ready to conclude the proof 
of Proposition~\ref{p.why}. Let $\gamma^{\ru\ru}_{\bar {\upsilon}_k(\mu)}$ be the segment of $\sigma^{\ru\ru}_{\bar {\upsilon}_k(\mu)}$
joining $P_{\bar {\upsilon}_{k}(\mu)}^+$ and $S_{\bar {\upsilon}_{k}(\mu)} (\delta)$ and consider  (see Figure~\ref{fig:twodimcon})
$$
 \widetilde \gamma^{\ru\ru}_{\bar {\upsilon}_k(\mu)} \eqdef
f^{n_k}_{\bar {\upsilon}_{k}(\mu)} (\Phi_k (\gamma^{\ru\ru}_{\bar {\upsilon}_k(\mu)})).
$$

Consider the domain 
$\Delta_k (\mu)$   of the blender $\Upsilon_{\bar {\upsilon}_k(\mu)}$ 
 in \eqref{e.DeltaKDelta}.
The calculations in  
\cite[Step A, eq. (35)]{DiaPer:19} imply that for every large $k$
the coordinates $(x,1+y,z)\in U_Q$ of the points in
$f^{n_k}_{\bar {\upsilon}_k(\mu)}(\Delta_k (\mu))$ are close to 
$X=(0,1,0)$ and they
have Landau symbols
$$
x=O({\lambda_Q}^{n_k}), \quad
y=O(\sigma_P^{-2m_k}\,{\sigma_Q}^{-n_k}), \quad 
{z}=O({\lambda_Q}^{n_k}).
$$
Hence $f^{n_k}_{\bar {\upsilon}_{k}(\mu)}  (P^+_{\bar {\upsilon}_{k}(\mu)})$ 
converges to $X$ as $k \to  \infty$.
It follows now from Lemma~\ref{l.thefinalintersection}  that
the curves $\widetilde \gamma^{\ru\ru}_{\bar {\upsilon}_k(\mu)}$ accumulate to the disc 
$$
\{0\} \times [1-\delta,1] \times \{0\} \subset W^{\ru}_{\loc} (Q, f^{n_k}_{\bar {\upsilon}_{k}(\mu)} ).
$$
This implies that for sufficiently large $k$
the curve $\widetilde \gamma^{\ru\ru}_{\bar {\upsilon}_k(\mu)}$ transversely intersects the disc $W^-_{\ve, \bar\upsilon_k(\mu)}\subset
W^{\rs}_{\loc} (Q, f_{\bar {\upsilon}_k(\mu)})$.  As the curve $\widetilde \gamma^{\ru\ru}_{\bar {\upsilon}_k(\mu)}$ is contained in the unstable manifold of the blender this ends the proof of Proposition~\ref{p.why}.

\section{Proof of Theorem~\ref{t:1}: Intersections between one-dimensional manifolds}
\label{s.1-conecc}

We now prove that the stable manifold the blender
 $\Upsilon_{\ve,k,\mu}$ of $g_{\ve,k,\mu}$   robustly intersects the unstable manifold of the saddle $Q$.

\begin{prop}\label{p.1dim} 
For every small $\ve>0$,   large $k$, and $\mu \in (-10,-9)$,
there is a $C^r$ neighbourhood $\mathcal{V}_{\ve,k,\mu}$ of 
$ g_{\ve,k,\mu}$  consisting 
of diffeomorphisms $g$
such that 
$$
W^\mathrm{s}
\big(
{\Upsilon}_g,g
\big) \cap
 W^\mathrm{u}\big(Q_g,g\big)\neq\emptyset,
$$
where
$\Upsilon_g$ and $Q_g$ are the continuations of
 $\Upsilon_{\ve,k,\mu}$ and $Q$. Moreover, this intersection can be chosen quasi-transverse.
\end{prop}

By Lemmas~\ref{l.opensuperposition} and \ref{l.superpositionregion} this proposition follows from the next result:

\begin{lemma} \label{l.adiskinthesuperpositionregion} 
For every  small $\ve>0$,   large $k$, and $\mu \in (-10,-9)$, the unstable manifold
 $W^\mathrm{u}(Q,g_{\ve,k,\mu})$ contains a $\ru\ru$-disc in the  superposition region 
 of  $\Upsilon_{\ve,k,\mu}$.  
 \end{lemma}
 
  \subsection{Proof of Lemma~\ref{l.adiskinthesuperpositionregion} }
We import some ingredients from Section~\ref{ss.blender-horseshoesreno}. 
Consider the disc $L$ in \eqref{e.soloL}. By Remark~\ref{r.opa}, for sufficiently large $k$
the set
\[
\Phi_{k}^{-1}\circ \mathcal{R}_{\ve, k,\mu}(g_{\ve,k,\mu} )\circ\Phi_{k} 
\big(
L \big) \subset \mathbb{R}^3
\]
 contains a disc 
 in the superposition region of the blender 
 $\Lambda_{\ve,k,\mu}$
 in \eqref{e.ijk-Grandeblender}.

Take small $\delta=\delta(\ve)>0$ and
the segment 
$L^{\mathrm{u}}_{1,\ve}\eqdef L^{\mathrm{u}}_{1,\ve}(\delta)\subset W^{\mathrm u}(Q,f_\ve)$ in \eqref{e.1seg1}.
 By the definition of $g_{\ve,k,\mu}$, this disc is contained in 
$W^{\mathrm u}(Q,g_{\ve,k,\mu})$. Thus
\begin{equation}
\label{e.Juexplicit}
J^{\mathrm{u}}_{\ve,k,\mu}\eqdef 
g_{\ve,k, \mu}^{N_1}
 \big(L^{\mathrm{u}}_{1,\ve}\big) \subset 
 W^{\mathrm u}(Q,g_{\ve,k,\mu}).
\end{equation}

Note that the transition $g_{\ve,k, \mu}^{N_1}$ does not depend on $\mu$, thus in what follows 
we will omit the dependence on $\mu$ of the sets $J^{\mathrm{u}}_{\ve,k,\mu}$ writing just $J^{\mathrm{u}}_{\ve,k}$.
We will prove that 
there is a compact subdisc $\widehat{J}^{\mathrm{u}}_{\ve,k}$ of 
$J^{\mathrm{u}}_{\ve,k}$ such
that the $C^r$ distance between the discs
$$
 \Phi_{k}^{-1} \circ g_{\ve,k,\mu}^{N_2+m_k} (\widehat{J}^{\mathrm{u}}_{\ve,k}) \quad
 \mbox{and}
 \quad
 \Phi_{k}^{-1} \circ \mathcal{R}_{\ve, k,\mu}(g_{\ve,k,\mu})\circ \Phi_k(L)  
 $$ 
 goes to zero as $k\to \infty$. 
 As the latter set contains a disc in the superposition region of the blender,
 this implies the lemma (see Figure~\ref{fig:1-D}).
 We now go to the details of the proof.

\begin{figure}
\centering
\begin{overpic}[scale=0.1,
]{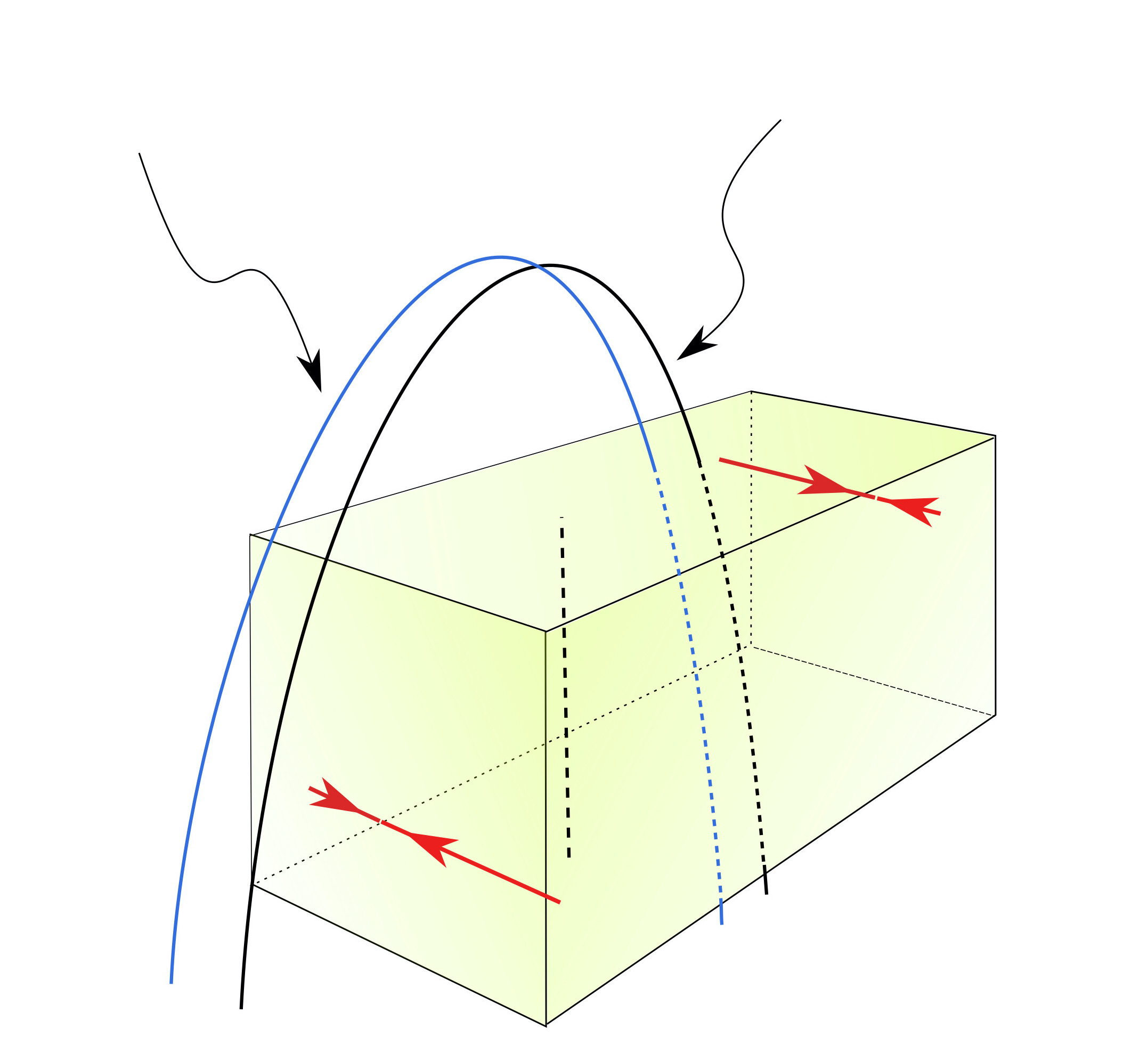}
  \put(171,73){\Large{$\Delta$}}  
  \put(100,105){\Large{$L$}}  
        \put(70,40){$\bullet$}  
             \put(160,103){$\bullet$}            
        \put(110,183){\large  $\Phi_{k}^{-1}\circ \mathcal{R}_{\ve, k,\mu}(g_{\ve,k,\mu})\circ \Phi_k(L)$}     
 \put(-10,180){\large        
        $\Phi_{k}^{-1} \circ g_{\ve,k,\mu}^{N_2+m_k} (\widehat{J}^{\mathrm{u}}_{\ve,k})  
        $}       
 \end{overpic}
\caption{One-dimensional connection between the blender and the saddle $Q$.} 
\label{fig:1-D}
\end{figure}

 Recall the definition of $\Phi_{k}=
\Psi_{k}\circ\Theta_{\bar \varsigma}$  in~\eqref{e.kcor} and consider the parameterisation of $\Phi_k(L)$
given by
\[
\gamma_k:[-4,4]\to M,\quad
\gamma_k(t)\eqdef \Psi_k \circ \Theta_{\bar \varsigma} (0,t,0)= \Psi_k (0, \varsigma_2^{-1} t, 0).
\]

We will provide a  parameterisation
$\gamma_{\ve,k}:[-4,4]\to M$
 of   $\widehat  J_{\ve,k}^\ru$ 
such that 
\begin{equation}
\label{e.horadealmuerzo}
\lim_{k\to \infty} \,
\big\Vert 
\big(
\Phi_{k}^{-1}\circ g_{\ve,k,\mu}
^{N_2+m_k+N_1+n_k}\circ\gamma_k
\,-\,\Phi_{k}^{-1}\circ g_{\ve,k,\mu}
^{N_2+m_k}\circ\gamma_{\ve,k}\big)|_{[-4,4]}
\big\Vert_{r}= 0.
\end{equation}
As the set $\Psi_{k}^{-1}\circ g_{\ve,k,\mu}
^{N_2+m_k+N_1+n_k}\circ \gamma_k([-4,4])$ contains a disc in the superposition region of the blenders (see Remark~\ref{r.opa}) this proves Lemma~\ref{l.adiskinthesuperpositionregion}.

 To get equation \eqref{e.horadealmuerzo} we observe that
\begin{equation*}
\begin{split}
\big\Vert \,
&\Phi_{k}^{-1}\circ g_{\ve,k,\mu}
^{N_2+m_k+N_1+n_k}\circ\gamma_k
\,-\,\Phi_{k}^{-1}\circ g_{\ve,k,\mu}
^{N_2+m_k}
\circ
{\gamma}_{\ve,k}\,
\big\Vert_{r}=
\\
&\,\,\big\Vert  \,
{\Theta}_{\bar \varsigma}^{-1} \circ
\Psi_{k}^{-1}\circ g_{\ve,k,\mu}
^{N_2+m_k+N_1+n_k}\circ \gamma_k
\,-\, {\Theta}_{\bar \varsigma}^{-1} \circ \Psi_{k}^{-1}\circ g_{\ve,k,\mu}
^{N_2+m_k}\circ\gamma_{\ve,k}
\,
\big\Vert_{r} \leqslant \\
& \,\, \,C \, \big\Vert  
\,
\Psi_{k}^{-1}\circ g_{\ve,k,\mu}
^{N_2+m_k+N_1+n_k} \circ \gamma_k
\,-\,\Psi_{k}^{-1}\circ g_{\ve,k,\mu}
^{N_2+m_k}\circ\gamma_{\ve,k}
\,
\big\Vert_{r},
\end{split}
\end{equation*}
where $C$ is an upper bound of the $C^r$ norm of ${\Theta}_{\bar \varsigma}^{-1}$ in the cube $\Delta$.
Thus to conclude the proof of the proposition  it is enough to prove the following:

\begin{lemma} \label{lsl.casicasi}
There is a parameterisation ${\gamma}_{\ve,k}:[-4,4]\to M$ 
 of   $\widehat  J_{\ve,k}^\ru$ such that 
$$
\lim_{k\to \infty}\,
\big\Vert  
\big (
\Psi_{k}^{-1}\circ g_{\ve,k,\mu}
^{N_2+m_k+N_1+n_k} 
\circ
\gamma_k
\,-\,
\Psi_{k}^{-1}\circ g_{\ve,k,\mu}
^{N_2+m_k}
\circ
\gamma_{\ve,k}
\big)|_{[-4,4]}
\big\Vert_{r} 
=0.
$$

\end{lemma}

We begin by giving the explicit
parameterisation of  $\gamma_{\ve,k}$  of   $\widehat  J_{\ve,k}^\ru$. 

\subsubsection{The parameterisations
$\gamma_{\ve,k}$} \label{sss.parameter}
 We first write the segment 
$J^\mathrm{u}_{\ve,k}$ in \eqref{e.Juexplicit} in local coordinates.
Recalling the definitions of $L^\mathrm{u}_{1,\ve}$ in  \eqref{e.1seg},
$\mathfrak{T}_{1,1,\ve}=f_{\ve}^{N_1}|_{U_{1,\ve}}$ in    \eqref{e.newtransition}, and 
$\theta_{\ve,k}$ in \eqref{e.dtildetheta}
(where $\bar\tau_k$ is defined),  we have 
\begin{equation*}
\begin{split}
&J^\mathrm{u}_{\ve,k}=
\big\{
\widetilde{X}_{1,\ve}
+A\big(t\,\textbf{e}_2+\bar{\rho}_{1,\ve}(t)\big)+\widetilde{H}^{1}_{\ve} \big(t\,\textbf{e}_2+\bar{\rho}_{1,\ve}(t)\big)
+\\
&\qquad 
+\Pi_{\delta}\big(\, A\big(t\, \textbf{e}_2+\bar{\rho}_{1,\ve}(t)\big)+\widetilde{H}^{1}_\ve\big(t\,\textbf{e}_2+\bar{\rho}_{1,\ve}(t)\big)\,\big)
\bar{\tau}_k
: |t|<\delta=\delta (\ve)
\big\}.
\end{split}
\end{equation*} 
Using \eqref{e.componentes}, \eqref{e.Ee}, $\bar{\rho}_{1,\ve}(0)=0$, and Remark~\ref{r.regular},
we get  $\widehat{\delta}\in (0,\delta)$ such that for every $|t|<\widehat \delta$ it holds
\[
A\big(t\,\textbf{e}_2+\bar{\rho}_{1,\ve}(t)\big)+\widetilde{H}^{1}_\ve \big(t\,\textbf{e}_2+\bar{\rho}_{1,\ve}(t))
\in B (\mathbf{0},\delta).
\]
As the map $\Pi_{\delta}$
 is igual to $1$ in  $B (\mathbf{0},\delta)$ (recall \eqref{e.bump}), we can consider the subdisc
\begin{equation*}
\widetilde J^\mathrm{u}_{\ve,k} \eqdef
\big\{
\widetilde{X}_{1,\ve}
+A\big(t\,\textbf{e}_2+\bar{\rho}_{1,\ve}(t)\big)+\widetilde{H}^1_\ve\big(t\,\textbf{e}_2+\bar{\rho}_{1,\ve}(t)\big)
+
\bar{\tau}_k
: |t| \leqslant  \widehat \delta\,
\big\} \subset J^\mathrm{u}_{\ve,k}.
\end{equation*}
For convenience, we write $\widetilde J^\mathrm{u}_{\ve,k}$ in the following compact form:
\[
\widetilde J^\mathrm{u}_{\ve,k}=\big\{ \widetilde{X}_{1,\ve}+t\,A(\bv_{1,\ve})
+\widetilde{\rho}_{1,\ve}(t)+
\bar{\tau}_k: |t| \leqslant  \widehat\delta\,
\big\},
\] where
\begin{equation*}
\begin{split}
\bv_{1,\ve}&\eqdef
\textbf{e}_2+A^{-1}D\widetilde{H}^1_\ve(\textbf{0})\,\textbf{e}_2,
\\
\widetilde{\rho}_{1,\ve}(t)&\eqdef A\big(\bar{\rho}_{1,\ve}(t)\big)+\widetilde{H}^1_\ve\big(t\,\textbf{e}_2+\bar{\rho}_{1,\ve}(t)\big)-t\,D\widetilde{H}^1_\ve(\textbf{0})\,\textbf{e}_2.
\end{split}
\end{equation*}
Note that, by \eqref{e.componentes}, we have that 
\[
\frac{d}{dt}(\widetilde{\rho}_{1,\ve})(0)
=\widetilde{\rho}_{1,\ve}(0)=\textbf{0}.
\]

We now consider the subdisc $\widehat{J}^\mathrm{u}_{\ve,k}$ of $\widetilde J^\mathrm{u}_{\ve,k}$
obtained by rescaling the parameter $t$ by the factor $\sigma_P^{-2m_k}{\sigma_Q}^{-n_k}\,|\,\varsigma_2^{-1}| \ll \frac{
\widehat \delta}{4}
<1$
as follows:
\[
\widehat{J}^{\mathrm u}_{\ve,k}\eqdef \big\{
\widetilde{X}_{1,\ve}+
\sigma_P^{-2m_k}{\sigma_Q}^{-n_k}\,\varsigma_2^{-1}\,t\,A(\bv_{1,\ve})
+ \widehat \rho_{1,\ve,k}(t) +
\bar{\tau}_k
: |t| \leqslant
4
\big\},
\]
where 
$$
\widehat{\rho}_{1,\ve,k}(t)\eqdef  \widetilde{\rho}_{1,\ve}
\big(\sigma_P^{-2m_k}{\sigma_Q}^{-n_k}\,\varsigma_2^{-1}\,t
\big).
$$
To rewrite the set $\widehat{J}^{\mathrm u}_{\ve,k}$ in a compact form, let
\begin{equation}\label{e.?}
\begin{split}
(\widetilde{w}^{1,\ve}_1,\widetilde{w}^{1,\ve}_2,\widetilde{w}^{1,\ve}_3)&\eqdef D\widetilde{H}^{1}_\ve(\textbf{0})\,\textbf{e}_2,
\\
\widehat{\rho}^{\ell}_{1,\ve,k}(t) &\eqdef  \widetilde{\rho}^\ell_{1,\ve}
\big(\sigma_P^{-2m_k}{\sigma_Q}^{-n_k}\,\varsigma_2^{-1}\,t
\big),\qquad \ell=1,2,3,
\end{split}
\end{equation}
where $\widetilde{\rho}^\ell_{1,\ve}$ is the $\ell$-th coordinate of $\widetilde{\rho}_{1,\ve}$.

\begin{remark}\label{r.convergencew}{\em{
Remark~\ref{r.regular} implies that
$(\widetilde{w}^{1,\ve}_1,\widetilde{w}^{1,\ve}_2,\widetilde{w}^{1,\ve}_3)\to (0,0,0)$ as $\ve\to 0$.}}
\end{remark}

Recalling  that 
$\widetilde{X}_{1,\ve}=(1+\widetilde{x}_{1,\ve},0,0)$, see
\eqref{e.wastedyears} and \eqref{e.simeone}, and
the definitions of $A$ in \eqref{e.Aisoftheform} and of $\bar \tau_k$ in \eqref{e.deftau},
we can write
\[
\widehat {J}^{\mathrm u}_{\ve,k}
=\big\{
\big(1+\widetilde{x}_{1,\ve}+{x}_{\ve,k}(t),{y}_{\ve,k}(t),{z}_{\ve,k}(t)\big): |t|
\leqslant
4
\big\},
\]
where 
\begin{equation}
\label{e.coordellk}
\begin{split}
{x}_{\ve,k}(t)&
\eqdef
 \sigma_P^{-2m_k}{\sigma_Q}^{-n_k}\,\big(\alpha_2+\widetilde{w}^{1,\ve}_1)\,\varsigma_2^{-1}\,t
+\widehat{\rho}^{1}_{1,\ve,k}(t),
\\
{y}_{\ve,k}(t)&
 \eqdef
\sigma_P^{-2m_k}{\sigma_Q}^{-n_k}\,\big(\beta_2+\widetilde{w}^{1,\ve}_2\big)\,\varsigma_2^{-1}\,t
+\widehat{\rho}^{2}_{1,\ve,k}(t)\\
&\,\,\,\,\,\,\,\,+\sigma^{-m_k}_P(\tilde{\mathfrak{c}}_k+\tilde{\mathfrak{s}}_k),
\\
{z}_{\ve,k}(t)&
\eqdef
\sigma_P^{-2m_k}{\sigma_Q}^{-n_k}\,
\widetilde{w}^{1,\ve}_3\,\varsigma_2^{-1}\,t
+\widehat{\rho}^{3}_{1,\ve,k}(t)
+\sigma^{-m_k}_P(\tilde{\mathfrak{c}}_k-\tilde{\mathfrak{s}}_k).
\end{split}
\end{equation}
The announced  parameterisation of $\widehat  J_{\ve,k}^\ru$ is given by
\begin{equation}
\label{e.gammaepsilon}
\gamma_{\ve,k}:[-4,4]\to M,\quad
\gamma_{\ve,k}(t)= (1+ \widetilde x_{1,\ve} + x_{\ve,k}(t), y_{\ve,k}(t), z_{\ve,k}(t)).
\end{equation}

\subsubsection{End of the proof of Lemma~\ref{lsl.casicasi} (thus of Proposition~\ref{p.1dim})}
We calculate separately the two terms in the lemma. This involves some explicit calculations in the renormalisation scheme borrowed from~\cite{DiaPer:19} which are stated in
Section~\ref{s.explicit1}.

\medskip
\noindent{\em{$\bullet$ The term $\Psi_{k}^{-1}\circ g_{\ve,k,\mu}
^{N_2+m_k+N_1+n_k}\big( \gamma_{k}(t) \big)$}.} 
Write
\begin{equation}\label{e.15}
\begin{split}
(\overline x_{\ve, k,\mu}(t), \overline y_{\ve, k,\mu}(t), \overline z_{\ve, k,\mu}(t))&\eqdef
\Psi_{k}^{-1}\circ g_{\ve,k,\mu}^{N_2+m_k+N_1+n_k}\big( \gamma_{k}(t) \big)
\\
&\hspace{-1cm}
 =\Psi^{-1}
_{k}\circ g_{\ve,k,\mu}
^{N_2+m_k+N_1+n_k}\circ \Psi_{k} ( \Theta_{\bar \varsigma} (0, t, 0)).
\end{split}
\end{equation}
Using  the formula in \eqref{e.borrowedfrom},
replacing $f_{ \bar \upsilon_k(\mu), \rho}$ by $g_{\ve,k, \mu}=\theta_{\ve, k} \circ f_{\ve, \bar\upsilon_k(\mu), \rho(\ve)}$ 
 (this leads to a dependence on $\ve$ of the next expressions),
and considering the curve
  $\Theta_{\bar \varsigma}(0,t,0)=(0, \varsigma_2^{-1} t, 0)$ (see
  \eqref{e.todaslastetas}),  from equation~\eqref{e.k-coord} we get
\begin{equation*}
\begin{split}
\overline{x}_{\ve,k,\mu}(t)&=\Big(
a_1\,\lambda_P^{m_k}\,
\sigma_P^{-m_k}\,\alpha_2
+
\big(\tilde{\mathfrak{c}}_k\,a_2+\tilde{\mathfrak{s}}_k\,a_3\big)\,\beta_2
\Big)\varsigma_2^{-1}\,t+
{\sigma_P}^{m_k}\,\sigma_Q^{n_k}\,\mathrm{hot}^x_{\ve, k,\mu}
(t),
\\
\overline{y}_{\ve, k,\mu}(t)
&=
\varsigma_2^{-1}\,\mu +b_1\,\lambda_P^{m_k}
\,{\sigma_Q}^{n_k}\,\alpha_2\,\varsigma_2^{-1}\,t+
\Big(\tilde{\mathfrak{c}}_k^2\,b_2+\tilde{\mathfrak{s}}_k^2\,b_3+
\tilde{\mathfrak{c}}_k\,\tilde{\mathfrak{s}}_k
\,b_4\Big)\,\beta_2^2\,\varsigma_2^{-2}\,t^2 +\\
&\qquad 
+{\sigma_P}^{2m_k}\,\sigma_Q^{2n_k}\,\mathrm{hot}^{y}_{\ve, k,\mu}(t),
\\
\overline{z}_{\ve,k,\mu}(t)&=\Big(
c_1\,\lambda_P^{m_k}\,
\sigma_P^{-m_k}\,\alpha_2
+
\big(\tilde{\mathfrak{c}}_k\,c_2+\tilde{\mathfrak{s}}_k\,c_3\big)\,\beta_2\Big)\,\varsigma_2^{-1}\,t+
{\sigma_P}^{m_k}\,\sigma_Q^{n_k}\,\mathrm{hot}^{z}_{\ve, k,\mu} (t),
\end{split}
\end{equation*}
where $\mathrm{hot}^{\ast}_{\ve, k,\mu} (t)$, $\ast=x,y,z$, are high order terms\footnote{In \cite{DiaPer:19} these high order terms are denoted by
  $\mathrm{h.o.t.}^{*}$,  $\mathrm{h.o.t.}^{**}$,  $\mathrm{h.o.t.}^{***}$.}.

\begin{remark}[Lemma 8.3 in \cite{DiaPer:19}]
\label{r.nomennescio}
{\em{The terms
$$
{\sigma_P}^{m_k}\,\sigma_Q^{n_k}
\,\mathrm{hot}^{x}_{\ve, k,\mu}(t),
\quad
{\sigma_P}^{2m_k}\,\sigma_Q^{2n_k}
\,\mathrm{hot}^{y}_{\ve, k,\mu}(t), 
\quad
{\sigma_P}^{m_k}\,\sigma_Q^{n_k}
\,\mathrm{hot}^{z}_{\ve, k,\mu}(t),
$$
go to zero in the $C^r$ topology as $k$ goes to infinity.
}}
\end{remark}

\medskip
\noindent{\em{$\bullet$ The term  
$\Psi_{k}^{-1}\circ g_{\ve,k,\mu}
^{N_2+m_k}\big({\gamma}_{\ve,k}(t)\big)$.}}

Recalling the parameterisation of ${\gamma}_{\ve,k}(t)$ in \eqref{e.gammaepsilon}, write
\begin{equation}
\label{e.tildegammacoord}
 \begin{split}
 &(\widetilde x_{\ve,k,\mu}(t),\widetilde y_{\ve,k,\mu}(t),\widetilde z_{\ve,k,\mu}(t))
\eqdef 
\Psi_{k}^{-1}\circ g_{\ve,k,\mu}
^{N_2+m_k}\big({\gamma}_{\ve,k}(t)\big)\\
&\quad\quad\quad=
\Psi^{-1}
_{k}\circ g_{
\ve,k,\mu}
^{N_2+m_k}\big(1+\widetilde x_{1,\ve}+{x}_{\ve,k}(t),{y}_{\ve,k}(t),{z}_{\ve,k}(t)\big).
\end{split}
\end{equation}
Using equation \eqref{e.coordellk}  and
 the linearity of $g_{\ve,k,\mu}$ in $U_P$, we get 
$$
g_{
\ve,k,\mu}
^{m_k}\big(1+\widetilde x_{1,\ve}+{x}_{\ve,k}(t),{y}_{\ve,k}(t),{z}_{\ve,k}(t)\big)
\eqdef \big(\widehat{x}_{\ve,k}(t),1
+\widehat{y}_{\ve,k}(t),
1+\widehat{z}_{\ve,k}(t)\big),
$$
where
\begin{equation*}
\begin{split}
\widehat{x}_{\ve,k}(t)&=
{\lambda_P}^{m_k}\,(1+\widetilde{x}_{1,\ve})+
{\lambda_P}^{m_k}\sigma_P^{-2m_k}
{\sigma_Q}^{-n_k}\,\big(\alpha_2+\widetilde{w}^{1,\ve}_1\big)\,\varsigma_2^{-1}\,t
\\ 
&\,\,\,\,\,\,\,\,+{\lambda_P}^{m_k}\widehat{\rho}^{1}_{1,\ve,k}(t),
\\
\widehat{y}_{\ve,k}(t)&=
\big(
\tilde{\mathfrak{c}}_k(\beta_2+\widetilde{w}^{1,\ve}_2)
-\widetilde{w}^{1,\ve}_3\,\tilde{\mathfrak{s}}_k
\big)\,
\sigma_P^{-m_k}
\,{\sigma_Q}^{-n_k}
\,\varsigma_2^{-1}\,t
+{\sigma_P}^{m_k}\,
u_{\ve,k}(t),
\\
\widehat{z}_{\ve,k}(t)&=
\big(
\tilde{\mathfrak{s}}_k(\beta_2+\widetilde{w}^{1,\ve}_2)
+\widetilde{w}^{1,\ve}_3\,\tilde{\mathfrak{c}}_k
\big)\,
\sigma_P^{-m_k}
\,{\sigma_Q}^{-n_k}
\,\varsigma_2^{-1}\,t
+{\sigma_P}^{m_k}\,
v_{\ve,k}(t),
\end{split}
\end{equation*}
with  
 \begin{equation}
 \label{e.sinnombre}
 \begin{split}
u_{\ve, k}(t)&\eqdef 
\tilde{\mathfrak{c}}_k\,\widehat{\rho}^{2}_{1,\ve,k}(t)-
\tilde{\mathfrak{s}}_k\,
\widehat{\rho}^{3}_{1,\ve,k}(t),\\
v_{\ve,k}(t)&\eqdef
\tilde{\mathfrak{s}}_k\,
\widehat{\rho}^{2}_{1,\ve,k}(t)+
\tilde{\mathfrak{c}}_k\,
\widehat{\rho}^{3}_{1,\ve,k}(t).
\end{split}
\end{equation}
Recalling 
 the expressions of $\widehat{\rho}^{2}_{1,\ve,k}(t)$, $\widehat{\rho}^{3}_{1,\ve,k}(t)$ in ~\eqref{e.?}, for $i=2,3$ we get
  \begin{equation}
  \label{e.Landaus}
  O(u_{\varepsilon,k}(t))
=  O(v_{\varepsilon,k}(t))=  
  O( \widehat{\rho}^{i}_{1,\ve,k}(t))= 
 O(\sigma_P^{-4m}
{\sigma_Q}^{-2n_k}).
\end{equation}
Thus the Landau symbols of $\widehat{x}_{\ve,k}(t),
\widehat{y}_{\ve,k}(t)$, and $\widehat{z}_{\ve,k}(t)$ are of the form
\begin{equation*}\label{e.siM}
\widehat{x}_{\ve,k}(t)
=O({\lambda_P}^{m_k}),\qquad  \widehat{y}_{\ve,k}(t)
=O(\sigma_P^{-m_k}{\sigma_Q}^{-n_k})
=
\widehat{z}_{\ve,k}(t).
\end{equation*}
Note that
\[
(\widetilde x_{\ve,k,\mu}(t),\widetilde y_{\ve,k,\mu}(t),\widetilde z_{\ve,k,\mu}(t))
=
\Psi_{k}^{-1}\circ g^{N_2}_{\ve,k,\mu}  \big(\widehat{x}_{\ve,k}(t),1+\widehat{y}_{\ve,k}(t),
1+\widehat{z}_{\ve,k}(t)\big).
\]
Recalling the definitions of $\Psi_{k}$ in \eqref{e.chart} and
of $f_{\bar \upsilon_k(\mu)}^{N_2}$ 
(see Remark~\ref{r.supUpsilontrans}) and that
$g_{\ve,k, \mu}^{N_2}=f_{\ve, \bar \upsilon_k(\mu), \rho(\ve)}^{N_2}=f_{\bar \upsilon_k(\mu)}^{N_2}$
 in the neighbourhood $f_{\bar \upsilon_k(\mu)}^{-N_2} (B(\widetilde Y, \rho (\ve)))$ of $Y$ 
that we
 are considering, 
we get
\begin{equation*}
\begin{split}
\widetilde x_{\ve,k,\mu}(t)&=
a_1\,{\lambda_P}^{m_k}\sigma_P^{m_k}{{\sigma_Q}}^{n_k}\,\widetilde x_{1,\ve}
\\
&
\,\,\,\,\,\,
+\Big(
a_1\,{\lambda_P}^{m_k}{\sigma_P}^{-m_k}\,(\alpha_2+\widetilde{w}^{1,\ve}_1)+a_2\big
(
\tilde{\mathfrak{c}}_k\,(\beta_2+ \widetilde{w}^{1,\ve}_2)-\widetilde{w}^{1,\ve}_3\,\tilde{\mathfrak{s}}_k
\big)
\\
&
\,\,\,\,\,\,
+
a_3
\big(
\tilde{\mathfrak{s}}_k\,(\beta_2+ \widetilde{w}^{1,\ve}_2)+\widetilde{w}^{1,\ve}_3\,\tilde{\mathfrak{c}}_k
\big)
\Big)\,
\,\varsigma_2^{-1}\,t+
{\mathrm{Hot}}^{x}_{\ve,k,\mu}(t);
\\
\widetilde y_{\ve,k,\mu}(t)&=\varsigma_2^{-1}\,\mu+ 
b_1\,{\lambda_P}^{m_k}\sigma_P^{2m_k}{{\sigma_Q}}^{2n_k}\,\widetilde x_{1,\ve}+
b_1\,{\lambda_P}^{m_k}\sigma_Q^{n_k}\,(\alpha_2+\widetilde{w}^{1,\ve}_1)\,\varsigma_2^{-1}\,t
\\
&
\,\,\,\,\,\,
+
\Big(
b_2\,\big(\tilde{\mathfrak{c}}_k(\beta_2+\widetilde{w}^{1,\ve}_2)-\widetilde{w}^{1,\ve}_3\,\tilde{\mathfrak{s}}_k\big)^2+
b_3\,\big(\tilde{\mathfrak{s}}_k(\beta_2+\widetilde{w}^{1,\ve}_2)+\widetilde{w}^{1,\ve}_3\,\tilde{\mathfrak{c}}_k\big)^2\\
&
\,\,\,\,\,\,
+
b_4\,\big(\tilde{\mathfrak{c}}_k(\beta_2+\widetilde{w}^{1,\ve}_2)-\widetilde{w}^{1,\ve}_3\,\tilde{\mathfrak{s}}_k\big)
\big(\tilde{\mathfrak{s}}_k(\beta_2+\widetilde{w}^{1,\ve}_2)+\widetilde{w}^{1,\ve}_3\,\tilde{\mathfrak{c}}_k\big)
\Big)
\,\varsigma^{-2}_2\,t^2\\
&
\,\,\,\,\,\,
+
{\mathrm{Hot}}^{y}_{\ve,k,\mu}(t);
\\
\widetilde z_{\ve,k,\mu}(t)&=
c_1\,{\lambda_P}^{m_k}\sigma_P^{m_k}{{\sigma_Q}}^{n_k}\,\tilde x_{1,\ve}
\\
&
\,\,\,\,\,\,
+\Big(
c_1\,{\lambda_P}^{m_k}{\sigma_P}^{-m_k}\,(\alpha_2+\widetilde{w}^{1,\ve}_1)+c_2\big
(
\tilde{\mathfrak{c}}_k\,(\beta_2+ \widetilde{w}^{1,\ve}_2)-\widetilde{w}^{1,\ve}_3\,\tilde{\mathfrak{s}}_k
\big)
\\
&
\,\,\,\,\,\,
+
c_3
\big(
\tilde{\mathfrak{s}}_k\,(\beta_2+ \widetilde{w}^{1,\ve}_2)+\widetilde{w}^{1,\ve}_3\,\tilde{\mathfrak{c}}_k
\big)
\Big)\,
\,\varsigma_2^{-1}\,t+
\mathrm{Hot}^{z}_{\ve,k,\mu}(t),
\end{split}
\end{equation*}
where $\mathrm{Hot}^{\ast}_{\ve, k,\mu} (t)$, $\ast=x,y,z$, are high order terms.
Their explicit expressions  can be found in Section~\ref{ss.highorderterms}.

\medskip
\noindent{\em{$\bullet$ Comparing the terms in Lemma~\ref{lsl.casicasi}}}
We are now ready to estimate the difference between the coordinates of the points in~\eqref{e.15} and~\eqref{e.tildegammacoord}. 
For that writing
$$
 \delta^w_{\ve,k,\mu}(t)\eqdef
\widetilde w_{\ve,k,\mu}(t)- \overline{w}_{\ve,k,\mu}(t),
\quad  w=x,y,z,
$$
we obtain 
\begin{equation*}
\begin{split}
 \delta^x_{\ve,k,\mu}(t)&=
a_1\,{\lambda_P}^{m_k}\sigma_P^{m_k}{{\sigma_Q}}^{n_k}\,\widetilde x_{1,\ve}
+\Big(
a_1\,{\lambda_P}^{m_k}{\sigma_P}^{-m_k}\,\widetilde{w}^{1,\ve}_1+a_2\big
(
\tilde{\mathfrak{c}}_k\, \widetilde{w}^{1,\ve}_2-\tilde{\mathfrak{s}}_k\,\widetilde{w}^{1,\ve}_3
\big)
\\
&+
a_3
\big(
\tilde{\mathfrak{s}}_k\, \widetilde{w}^{1,\ve}_2+\tilde{\mathfrak{c}}_k\,\widetilde{w}^{1,\ve}_3
\big)
\Big)\,
\,\varsigma_2^{-1}\,t+
{\mathrm{Hot}}^{x}_{\ve,k,\mu}(t)-\sigma^{m_k}_{P}\,\sigma^{n_k}_{Q}{\mathrm{hot}}^{x}_k
(t);
\end{split}
\end{equation*}
\begin{equation*}
\begin{split}
 \delta^y_{\ve,k,\mu}(t)&=
b_1\,{\lambda_P}^{m_k}\sigma_P^{2m_k}{{\sigma_Q}}^{2n_k}\,\widetilde x_{1,\ve}+
b_1\,{\lambda_P}^{m_k}\sigma_Q^{n_k}\,\widetilde{w}^{1,\ve}_1\,\varsigma_2^{-1}\,t
\\
&+
\Big(
b_2\,
\big[
2\,\beta_2\,\widetilde{w}^{1,\ve}_2\,\tilde{\mathfrak{c}}_k^2+(\widetilde{w}^{1,\ve}_2)^2\,\tilde{\mathfrak{c}}_k^2
-2\,\tilde{\mathfrak{c}}_k\,\tilde{\mathfrak{s}}_k\,\widetilde{w}^{1,\ve}_3(\beta_2+\widetilde{w}^{1,\ve}_2)
+(\widetilde{w}^{1,\ve}_3)^2\,\tilde{\mathfrak{s}}_k^2
\big]
\\
&+
b_3\,\big[
2\,\beta_2\,\widetilde{w}^{1,\ve}_2\,\tilde{\mathfrak{s}}_k^2+(\widetilde{w}^{1,\ve}_2)^2\,\tilde{\mathfrak{s}}_k^2
+2\,\tilde{\mathfrak{c}}_k\,\tilde{\mathfrak{s}}_k\,\widetilde{w}^{1,\ve}_3(\beta_2+\widetilde{w}^{1,\ve}_2)
+(\widetilde{w}^{1,\ve}_3)^2\,\tilde{\mathfrak{c}}_k^2
\big]\\
&
+
b_4\,\big[
2\,\beta_2\,\widetilde{w}^{1,\ve}_2\,\tilde{\mathfrak{s}}_k\,\tilde{\mathfrak{c}}_k+(\widetilde{w}^{1,\ve}_2)^2\,\tilde{\mathfrak{s}}_k\,\tilde{\mathfrak{c}}_k
\\
&+
(\tilde{\mathfrak{c}}_k^2-\tilde{\mathfrak{s}}_k^2)\,\widetilde{w}^{1,\ve}_3(\beta_2+\widetilde{w}^{1,\ve}_2)
-(\widetilde{w}^{1,\ve}_3)^2\tilde{\mathfrak{s}}_k\,\tilde{\mathfrak{c}}_k
\big]
\Big)
\,\varsigma^{-2}_2\,t^2
\\
&
+
{\mathrm{Hot}}^{y}_{\ve,k,\mu}(t)-
\sigma_P^{2m_k}{{\sigma_Q}}^{2n_k}\,
{\mathrm{hot}}^{y}_{k}(t);
\end{split}
\end{equation*}
\begin{equation*}
\begin{split}
 \delta^z_{\ve,k,\mu}(t)&=
c_1\,{\lambda_P}^{m_k}\sigma_P^{m_k}{{\sigma_Q}}^{n_k}\,\widetilde x_{1,\ve}
+\Big(
c_1\,{\lambda_P}^{m_k}{\sigma_P}^{-m_k}\,\widetilde{w}^{1,\ve}_1+c_2\big
(
\tilde{\mathfrak{c}}_k\, \widetilde{w}^{1,\ve}_2-\tilde{\mathfrak{s}}_k\,\widetilde{w}^{1,\ve}_3
\big)
\\
&+
c_3
\big(
\tilde{\mathfrak{s}}_k\, \widetilde{w}^{1,\ve}_2+\tilde{\mathfrak{c}}_k\,\widetilde{w}^{1,\ve}_3
\big)
\Big)\,
\,\varsigma_2^{-1}\,t+
{\mathrm{Hot}}^{z}_{\ve,k,\mu}(t)-\sigma^{m_k}_{P}\,\sigma^{n_k}_{Q}{\mathrm{hot}}^{z}_k (t).
\end{split}
\end{equation*}

We now prove that  
 $\big\Vert \delta^x_{\ve,k,\mu}\,|_{[-4,4]} \big\Vert_{r} \to 0$.
The proofs of 
 $\big\Vert \delta^\ast_{\ve,k,\mu}\,|_{[-4,4]}\big\Vert_{r} \to 0$, $\ast=y,z$,
 are 
similar and hence omitted. 
For that write
$$
\delta^x_{\ve,k,\mu}(t)=\mathcal{A}_{\ve,k,\mu}(t) +
 {\mathrm{Hot}}^{x}_{\ve,k,\mu}(t)-\sigma^{m_k}_{P}\,\sigma^{n_k}_{Q}{\mathrm{hot}}^{x}_k (t),
$$
where  $\mathcal{A}_{\ve,k,\mu}(t)$ denotes the affine part of
$\delta^x_{\ve,k,\mu}(t)$.

\begin{cl}\label{cl.inappendix}
$\,$
\begin{enumerate}
\item
$\lim_{k\to \infty} \, \big\Vert\ \mathcal{A}_{\ve,k,\mu}\,|_{[-4,4]} 
\big\Vert_{r} = 0$,
\item
$\lim_{k\to \infty} \, \big\Vert\sigma^{m_k}_{P}\,\sigma^{n_k}_{Q}\,{\mathrm{hot}}^{x}_{\ve,k,\mu} \,|_{[-4,4]}
\big\Vert_{r}= 0,
$
\item
$\lim_{k\to \infty} \,  \big\Vert
{\mathrm{Hot}}^{x}_{\ve,k,\mu}\,|_{[-4,4]}\big\Vert_{r} = 0$.
\end{enumerate}
\end{cl}
Clearly, this claim implies Lemma~\ref{lsl.casicasi}. 

\begin{proof}[Proof of Claim~\ref{cl.inappendix}]
For the  first item recall that by \eqref{e.espectroparazero} we get
$\lambda_P^{m_k}\sigma_P^{2m_k}{{\sigma_Q}}^{2n_k}\to 0$ and
that by Remark~\ref{r.convergencew} the norm of
 $(\widetilde{w}^{1,\ve}_1,\widetilde{w}^{1,\ve}_2,\widetilde{w}^{1,\ve}_3)$ is small.
The second item
was stated in Remark~\ref{r.nomennescio}.
The proof of  last item 
is postponed to
Section~\ref{sss.claim}.
\end{proof}
The proof of Proposition~\ref{p.1dim} is now complete.

\section{Proof of Theorem~\ref{t:1}: homoclinic relations}
\label{s.homocliinic relation}
We now prove that blender
 $\Upsilon_{\ve,k,\mu}$ and the saddle $P$ are homoclinically related.

\begin{prop}\label{p.homoclinicallyrelated}
For every small $\ve>0$,   large $k$, and $\mu \in (-10,-9)$ there is $g$ arbitrarily $C^r$ close to $g_{\ve,k,\mu}$ such that
$$
W^\mathrm{u}
\big(
{\Upsilon}_{g},g
\big) \pitchfork
 W^\mathrm{s}\big(P,g\big)\neq\emptyset \quad \mbox{and} \quad
 W^\mathrm{s}
\big(
{\Upsilon}_{g},g
\big) \pitchfork
 W^\mathrm{u}\big(P,g\big)\neq\emptyset.
 $$
\end{prop}

\begin{figure}
\centering
\begin{overpic}[scale=0.07,
]{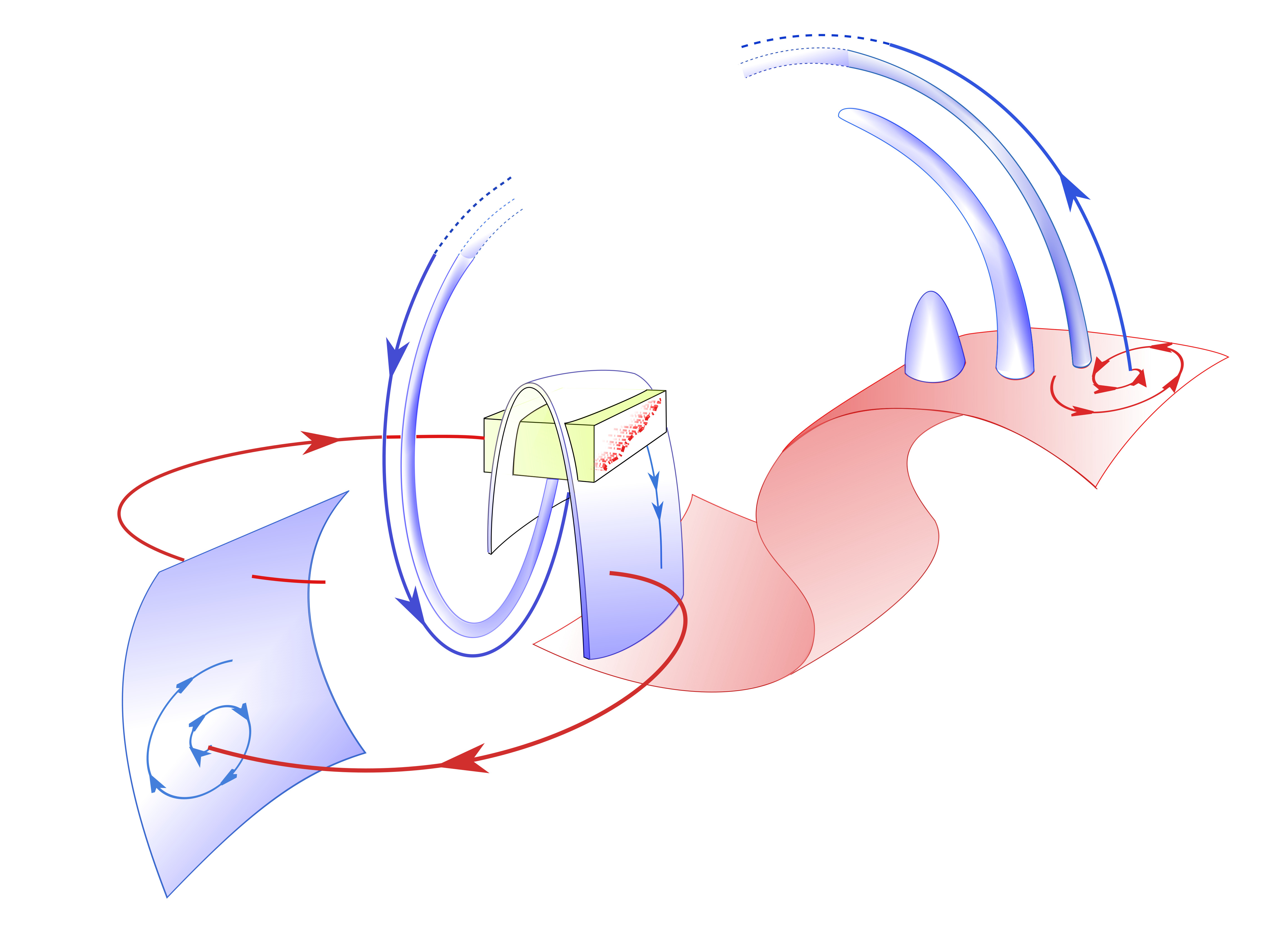}
  \put(65,53){\Large $P$}  
           \put(50,49){$\bullet$} 
     \put(162,152){\Large$\Upsilon_g$}  
  \put(290,165){\Large $Q$}  
           \put(283,146){$\bullet$}         
 \end{overpic}
\caption{Robust homoclinic tangencies} 
\label{fig:tatudo}
\end{figure}
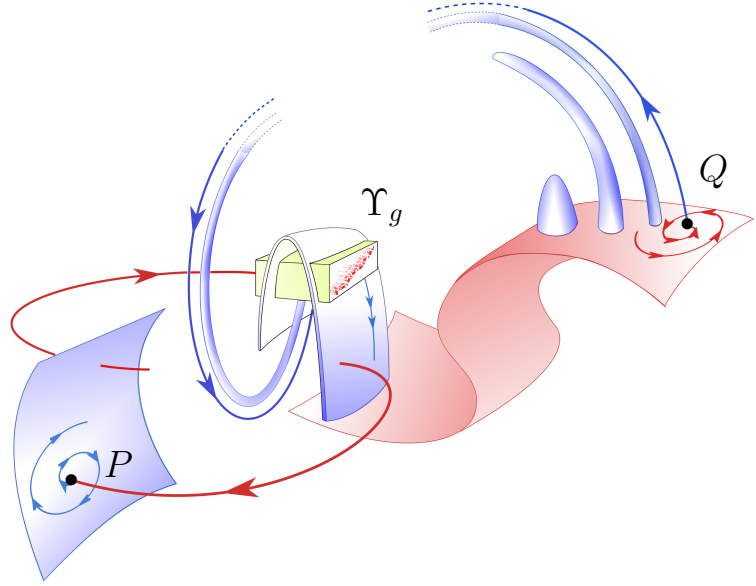

\begin{proof}
Note that Propositions~\ref{p.why} and \ref{p.1dim} hold for perturbations $g$ of  $g_{\ve,k,\mu}$.
Recall the definition of the quasi-transverse heteroclinic point $X_{2,\ve}$ in Proposition~\ref{p.Lnbis} and observe that 
this intersection point is preserved by the local perturbations we have considered.
Therefore $X_{2,\ve}$ is also a heteroclinic point of $g_{\ve,k,\mu}$. Note also that after a perturbation
$g$ of  $g_{\ve,k,\mu}$ (that does not destroy the heteroclinic points) we can assume that the argument of  $D g (Q)$ is irrational. Thus the hypotheses of Lemma~\ref{l.closure} hold.

By  Proposition~\ref{p.why},
$
W^\mathrm{u}
\big(
{\Upsilon}_g,g
\big) \pitchfork
 W^\mathrm{s}\big(Q,g\big)\neq\emptyset.
$ Thus, after a perturbation, that we continue denoting by $g$ and does not change the argument of the complex eigenvalue of $Dg(Q)$, we can assume that
there is a
disc $S\subset W^\mathrm{u}(
{\Upsilon}_{g},g)$ that transversely intersects $W^\mathrm{s}_{\mathrm{loc}} (Q, g)$ along a curve $\gamma$ with nontrivial radial projection.
Lemma~\ref{l.closure} now implies that $W^{\mathrm{s}} (P, g)$ transversally intersects $S\subset W^\mathrm{u} (\Upsilon_{g} , g)$, proving the first item of the proposition. Note also that $g$ has a robust cycle associated to $Q$ and $\Upsilon_g$.

To prove the second  intersection, note that $W^\mathrm{u} (P,g)$ and $W^\mathrm{s}_{\mathrm{loc}} (Q,g)$ interesects transversely along a curve
 with a nontrivial radial projection.
Note also that by Proposition~\ref{p.1dim},
$W^\mathrm{s}(
{\Upsilon}_g,g)$ and 
$W^\mathrm{u}\big(Q,g\big)$ intersects quasi-transversely at some point $Z$. Consider a curve
$D\subset W^\mathrm{s}(
{\Upsilon}_g,g)$ and containing the point $Z$ in its interior.  Arguing as above and applying Lemma~\ref{l.closure}, we have that the negative iterates of $D$
(contained in $W^\mathrm{s}(
{\Upsilon}_g,g)$ transversally intersects $W^\mathrm{s} (P,g)$. This completes the proof of the proposition.
\end{proof}

\section{Proof of Theorem~\ref{t:1}: homoclinic tangencies}
\label{s.homocliinictangencies}

In this section, we consider perturbations $g_{\ve,k,\mu}$ of the diffeomorphisms $f$ in  
$\mathcal{H}^r_{\BH,\mathrm{e}^+}(M)$. Next proposition implies the part of Theorem~\ref{t:1} about homoclinic tangencies.

\begin{prop}\label{p.tagnecies}
For every small $\ve>0$,   large $k$, and $\mu \in (-10,-9)$ there is $g$ arbitrarily $C^r$ close to $g_{\ve,k,\mu}$ with a $C^r$ robust homoclinic tangency associated to blender-horseshoe ${\Upsilon}_g$.
\end{prop}

\begin{proof}
Note that close to the original heterodimensional tangency, the manifold
$W^\mathrm{u}
\big(
P,g_{\ve,k,\mu}
\big)$ intersect $W^\mathrm{s}
\big(
Q,g_{\ve,k,\mu}
\big)$ 
 in closed curve denoted by $C_{\ve,k,\mu}$. Let $S_{\ve,k,\mu}$ be the two-dimensional compact disc contained in $W^\mathrm{u}
\big(
P,g_{\ve,k,\mu}
\big)$ bounded by  $C_{\ve,k,\mu}$.
By the $\lambda$-lemma, the  forward iterates $g^i_{\ve,k,\mu}(S_{\ve,k,\mu})$ of $S_{\ve,k,\mu}$ accumulated to the unstable manifold of $Q$.
 By Lemma~\ref{l.adiskinthesuperpositionregion},  the unstable manifold of $Q$ contains a disc in the superposition region of the blender ${\Upsilon}_{\ve,k,\mu}$. Thus  there are infinitely many iterates of $S_{\ve,k,\mu}$ 
 containing $\ruu$-tubes in
  the superposition region of the blender, see Figure~\ref{fig:tatudo} and recall Definition~\ref{d.stripstubesetc}.
 Corollary~\ref{c.tang} implies that the manifold $W^\mathrm{u}
\big(
P,g_{\ve,k,\mu}
\big)$ and $W^\mathrm{s}_{\mathrm{loc}}
\big(
\Upsilon_{\ve,k,\mu}
,g_{\ve,k,\mu}
\big)$ have a $C^r$ robust tangency. The proposition follows noting that by
Proposition~\ref{p.homoclinicallyrelated} the point $P$ and  the blender $\Upsilon_{\ve,k,\mu}$ are homoclinically related.
\end{proof}

\section{Calculations in the renormalisation scheme} \label{s.explicit1}

We collect some calculations from \cite{DiaPer:19} that we used in the previous sections.

\subsection{The renormalisation formula} \label{ss.renormalisationformula}
First, recall  the perturbations $f_{\bar\upsilon, \rho}$ of $f\in \mathcal{H}_{\BH}^r(M^3)$
 in  \eqref{e.thefamily}. For that, we  borrow from  \cite[Section 7.3]{DiaPer:19} 
the explicit formula for compositions 
of the form
 \begin{equation}
\label{e.borrowedfrom}
\Psi_k^{-1}\circ 
f_{ \bar \upsilon_k(\mu),\rho}^{N_2+m_k+N_1 +n_k}\circ 
 \Psi_k(x,y,z)\eqdef (\check{x}_{k,\mu,\rho}, \check{y}_{k,\mu, \rho}, \check{z}_{k,\mu,\rho}).
\end{equation}
Here, the
iterates corresponding to $n_k$ occurs in $U_Q$, the $N_1$ iterates correspond to the transition $\mathfrak{T}_1$,   the iterates corresponding to $m_k$ occurs in $U_P$, and
the $N_2$ iterates correspond to the transition $\mathfrak{T}_2$. 
      
 Consider the  heteroclinic points $X, Y$ in the cycle (see Section~\ref{ss.semilocal})
 and write
 \begin{equation}\label{e.hayquesepararlos}
 \begin{split}
\textbf{x}_{k,\mu,\rho} (x,y,z) &\eqdef f^{n_k}_{\bar \upsilon_k (\mu), \rho}\circ \Psi_k(x,y,z)-X,\\
\widehat{\textbf{x}}_{k,\mu,\rho} (x,y,z) &\eqdef f^{m_k+N_1+n_k}_{\bar \upsilon_k (\mu),\rho}\circ \Psi_k(x,y,z)-Y.
\end{split}
\end{equation}
 
 Using the notation in Section~\ref{s.ren},
 the composition in \eqref{e.borrowedfrom}  reads as follows:
\begin{equation}
\label{e.k-coord}
\begin{split}
\breve{x}_{k,\mu,\rho}&=
a_1\,\lambda_P^{m_k}\,
{\lambda_Q}^{n_k}\,
\big(
\alpha_1\,(\mathfrak{c}_k\,x-\mathfrak{s}_k\,z)
+\alpha_3\,(\mathfrak{s}_k\,x+\mathfrak{c}_k\,z)
\big)\\
&\quad +
a_1\,\lambda_P^{m_k}\,
\sigma_P^{-m_k}\,\alpha_2\,y
+
\big(\tilde{\mathfrak{c}}_k\,a_2+\tilde{\mathfrak{s}}_k\,a_3\big)\,\beta_2\,y\\
&\quad +
\sigma_P^{m_k}\,\lambda_Q^{n_k}\,\gamma_3
\,\big(\tilde{\mathfrak{c}}_k\,a_3-\tilde{\mathfrak{s}}_k\,a_2\big)\,
(\mathfrak{s}_k\,x+\mathfrak{c}_k\,z)
\\
&\quad +
{\sigma_P}^{m_k}\,\sigma_Q^{n_k}\,\mathrm{hot}^{{x}}_{k,\mu,\rho},
\\
\breve{y}_{k,\mu,\rho}
&=
\mu +b_1\,\lambda_P^{m_k}
\,{\sigma_Q}^{n_k}\,\alpha_2\,y
\\
&\quad +b_1\,{\lambda_P}^{m_k}\,
{\sigma_P}^{m_k}\,
\lambda_Q^{n_k}\,\sigma_Q^{n_k}\,
\big(
\alpha_1\,(\mathfrak{c}_k\,x-\mathfrak{s}_k\,z)
\alpha_3\,(\mathfrak{s}_k\,x+\mathfrak{c}_k\,z)
\big)\\
&\quad+
\Big(\tilde{\mathfrak{c}}_k^2\,b_2+\tilde{\mathfrak{s}}_k^2\,b_3+
\tilde{\mathfrak{c}}_k\,\tilde{\mathfrak{s}}_k
\,b_4\Big)\,\beta_2^2\,y^2
\\
&\quad 
+\sigma_P^{2m_k}\,\lambda_Q^{2n_k}
\Big(
\tilde{\mathfrak{s}}_k^2\,b_2+\tilde{\mathfrak{c}}_k^2\,b_3-\tilde{\mathfrak{c}}_k\,\tilde{\mathfrak{s}}_k\,b_4
\Big)\,\gamma_3^2\,
(\mathfrak{s}_k\,x+\mathfrak{c}_k\,z)
^2
\\
&\quad
+\sigma_P^{m_k}\,\lambda_Q^{n_k}\,
\Big(
2\tilde{\mathfrak{c}}_k\,\tilde{\mathfrak{s}}_k\,(b_3-b_2)
(\tilde{\mathfrak{c}}_k^2\,-\tilde{\mathfrak{s}}_k^2)\,b_4
\Big)\beta_2\,\gamma_3
(\mathfrak{s}_k\,x\,y+\mathfrak{c}_k\,y\,z)
\\
&\quad
+{\sigma_P}^{2m_k}\,\sigma_Q^{2n_k}\,\mathrm{hot}^{y}_{k,\mu,\rho},
\\
\breve{z}_{k,\mu, \rho}
&=
c_1\,\lambda_P^{m_k}\,
{\lambda_Q}^{n_k}\,
\big(
\alpha_1\,(\mathfrak{c}_k\,x-\mathfrak{s}_k\,z)
+\alpha_3\,(\mathfrak{s}_k\,x+\mathfrak{c}_k\,z)
\big)\\
&\quad +
c_1\,\lambda_P^{m_k}\,
\sigma_P^{-m_k}\,\alpha_2\,y
+
\big(\tilde{\mathfrak{c}}_k\,c_2+\tilde{\mathfrak{s}}_k\,c_3\big)\,\beta_2\,y\\
&\quad +\sigma_P^{m_k}\,\lambda_Q^{n_k}\,\gamma_3
\,\big(\tilde{\mathfrak{c}}_k\,c_3-\tilde{\mathfrak{s}}_k\,c_2\big)\,
(\mathfrak{s}_k\,x+\mathfrak{c}_k\,z)
\\
&\quad +
{\sigma_P}^{m_k}\,\sigma_Q^{n_k}\,\mathrm{hot}^{z}_{k,\mu,\rho}, 
\end{split}
\end{equation}
where $\mathrm{hot}^{\ast}_{k,\mu,\rho}= \mathrm{hot}^{\ast}_{k,\mu,\rho}(x,y,z)$, $\ast=x,y,z$,
 are higher order terms 
 whose Laundau's symbols satisfy the following conditions
 (see~\cite[Lemma 8.3]{DiaPer:19}). Write 
$$
\widehat{H}_2 (\textbf{x}_{k,\mu, \rho} ) \eqdef \widetilde{H}_2 (\textbf{x}_{k,\mu,\rho} )-
\lambda_Q^{2n_k}\,\tilde{\rho}_{2,k},
$$
where $\tilde{\rho}_{2,k}$ is defined in \eqref{e.removflatconditions}.
Then
\begin{itemize}
 \item [(i)] 
${O} \big(\mathrm{hot}^x_{k,\mu,\rho} (x,y,z)  \big)= 
{O} \big( {\lambda_P}^{m_k} 
\,\widetilde{H}_1 (\textbf{x}_{k,\mu,\rho})  \big)+
{O} \big(
\sigma_P^{m_k}\,\widehat{H}_2 (\textbf{x}_{k,\mu,\rho}) \big) 
\\
+ {O}\big(
H_1(\widehat{\textbf{x}}_{k,\mu,\rho})\big),
$
\item[(ii)]
${O} \big(\mathrm{hot}^y_{k,\mu, \rho}(x,y,z) \big)= 
 O(
{\lambda_P}^{m_k})+ 
O\big(\widetilde{H}_1 (\textbf{x}_{k,\mu,\rho} )\big)+
O\big(
\big(\sigma_P^{m_k}\,\widehat{H}_2 (\textbf{x}_{k,\mu,\rho})\big)^2\big)\\
+O\big(\sigma_Q^{-n_k}\,\widehat{H}_2 (\textbf{x}_{k,\mu,\rho} )\big)
+ O \big(H_2 (\widehat{\textbf{x}}_{k,\mu,\rho}) \big)$,
\item[(iii)]
${O} \big(\mathrm{hot}^{z}_{k,\mu,\rho} (x,y,z )\big)=
O\big(
{\lambda_P}^{m_k} 
\widetilde{H}_1 (\textbf{x}_{k,\mu,\rho})\big)+
O\big( \sigma_P^{m_k}\,\widehat{H}_2 (\textbf{x}_{k,\mu,\rho} )\big)
\newline + O\big(
H_3 (\widehat{\textbf{x}}_{k,\mu,\rho} )\big).
$
\end{itemize}

\subsection{The high order  terms of 
$\Psi_{k}^{-1}\circ g_{\ve,k,\mu}
^{N_2+m_k}\big({\gamma}_{\ve,k}(t)\big)$}
\label{ss.highorderterms}
Recall the definition of $\gamma_{\ve,k}$ in \eqref{e.gammaepsilon}. 
We now provide an explicit expression of the high order terms
$\mathrm{Hot}^{\ast}_{\ve,k,\mu}(t)$ in
$\Psi_{k}^{-1}\circ g_{\ve,k,\mu}
^{N_2+m_k}\big({\gamma}_{\ve,k}(t)\big)$.
Define $\textbf{x}_{\ve, k,\mu}$ and $\widehat{\textbf{x}}_{\ve, k,\mu}$ as in 
\eqref{e.hayquesepararlos}, where $f_{\bar \upsilon_k (\mu)}$ is replaced by $f_{\ve,\bar \upsilon_k (\mu)}$
(this  is why  the subscript $\ve$ appears).
Write
\begin{equation*}
\textbf{w}_{\ve,k,\mu}(t)\eqdef  
\textbf{x}_{\ve, k,\mu} \circ \Theta_{\bar \varsigma} (0,t,0),
\qquad
\widehat{\textbf{w}}_{\ve, k,\mu}(t)\eqdef  
\widehat{\textbf{x}}_{\ve,k,\mu} \circ \Theta_{\bar \varsigma} (0,t,0),
\end{equation*}
where 
$\Theta_{\bar \varsigma}$ is an in \eqref{e.todaslastetas}.
We have
\[
\begin{split}
{\mathrm{Hot}}^x_{\ve,k,\mu}(t)
&
=a_1\,
{\lambda_P}^{m_k}\sigma_P^{m_k}{{\sigma_Q}}^{n_k}
\widehat{\rho}^{1}_{\ve,k}(t)
+a_2\,
\sigma_P^{2m_k}
{{\sigma_Q}}^{n_k}
u_{\ve,k}(t)
\\
&\quad+
a_3\,
{\sigma_P}^{2m_k}
{{\sigma_Q}}^{n_k}
v_{\ve,k}(t)+
\sigma_P^{m_k}{{\sigma_Q}}^{n_k}
H_1\big(\widehat{\textbf{w}}_{\ve,k,\mu}(t)\big);
\\
{\mathrm{Hot}}^{y}_{\ve,k,\mu}(t)&=
b_1\,
{\lambda_P}^{m_k}\sigma_P^{2m_k}{{\sigma_Q}}^{2n_k}
\widehat{\rho}^{1}_{\ve,k}(t)
\\
&\quad+
\sigma_P^{2m_k}{\sigma_Q}^{n_k}\,
\Big(2b_2\big(
\tilde{\mathfrak{c}}_k(\alpha_2+\widetilde{w}^{1,\ve}_2)-\widetilde{w}^{1,\ve}_3
\tilde{\mathfrak{s}}_k\big)
\\
&\qquad\qquad+
b_4\big(
\tilde{\mathfrak{s}}_k(\alpha_2+\widetilde{w}^{1,\ve}_2)+
\widetilde{w}^{1,\ve}_3\,\tilde{\mathfrak{c}}_k\big)
\Big)\varsigma_2^{-1}\,t\,u_{\ve,k}(t)
\\
&\quad
+
\sigma_P^{2m_k}{\sigma_Q}^{n_k}\,
\Big(2b_3\big(
\tilde{\mathfrak{s}}_k(\alpha_2+\widetilde{w}^{1,\ve}_2)
+\widetilde{w}^{1,\ve}_3
\tilde{\mathfrak{c}}_k\big)
\\
&\qquad\qquad+
b_4\big(
\tilde{\mathfrak{c}}_k(\alpha_2+\widetilde{w}^{1,\ve}_2)-
\widetilde{w}^{1,\ve}_3\,\tilde{\mathfrak{s}}_k\big)
\Big)\varsigma_2^{-1}\,t\,v_{\ve,k}(t)
\\
&\quad
+
\sigma_P^{4m_k}{\sigma_Q}^{2n_k}\,
\Big(
b_2\,\big(u_{\ve,k}(t)\big)^2+
b_3\big(v_{\ve,k}(t)\big)^2
\\
&\qquad\qquad +
b_4\,u_{\ve,k}(t)\,v_{\ve,k}(t)
\Big)
+
\sigma_P^{2m_k}{{\sigma_Q}}^{2n_k}
H_2\big(\widehat{\textbf{w}}_{\ve,k,\mu}(t)\big);
\\
{\mathrm{Hot}}^{z}_{\ve,k,\mu}(t)&=c_1\,
{\lambda_P}^{m_k}\sigma_P^{m_k}{{\sigma_Q}}^{n_k}
\widehat{\rho}^{1}_{\ve,k}(t)
+c_2\,
\sigma_P^{2m_k}
{{\sigma_Q}}^{n_k}
u_{\ve,k}(t)
\\
&\quad+
c_3\,
{\sigma_P}^{2m_k}
{{\sigma_Q}}^{n_k}
v_{\ve,k}(t)+
\sigma_P^{m_k}{{\sigma_Q}}^{n_k}
H_3\big(\widehat{\textbf{w}}_{\ve,k,\mu}(t)\big).
\end{split}
\]
\subsubsection{Proof of item (3) in Claim~\ref{cl.inappendix}} 
\label{sss.claim}
We claim that
\[
\big\Vert
\big(
\mathrm{Hot}^{x}_{\ve,k,\mu}\,-\,
\sigma_P^{m_k}{{\sigma_Q}}^{n_k}
H_1\circ \widehat{\textbf{w}}_{\ve,k,\mu}\big)\,|_{[-4,4]}
\big\Vert_r\to 0.
\]
For this just 
note that 
\begin{itemize}
\item
$u_{\ve,k}(t)$, $v_{\ve,k}(t)$ and $\widehat{\rho}^\ell_{\ve,k}(t)$, $\ell=2,3$, have the same symbol of Landau  $O(\sigma_P^{-4m_k}\,\sigma_Q^{-2n_k})$, see \eqref{e.sinnombre} and
\eqref{e.Landaus},
\item
 $\widehat \rho_{\ve,k}^1$ is bounded and \eqref{e.espectroparazero}. 
 \end{itemize}
Finally, the convergence
$$
\lim_{k\to \infty} \big\Vert
\sigma_P^{m_k}{{\sigma_Q}}^{n_k}
H_1\circ \widehat{\textbf{w}}_{\ve,k,\mu}\,|_{[-4,4]}
\big\Vert_{r} = 0
$$
follows exactly as in
~\cite[Claim 8.4]{DiaPer:19}.

%
%

\bibliographystyle{plain}

\end{document}